\documentclass[oneside,12pt]{book}
\usepackage{amssymb}
\usepackage{amsfonts}
\usepackage{graphicx}
\usepackage{amsmath}
\usepackage{epsfig,longtable}
\usepackage{makeidx}
\usepackage[french,refpage]{nomencl}
\usepackage{minitoc}
\usepackage{fancyhdr}
\usepackage{fancybox}
\usepackage[perpage]{footmisc}
\usepackage{mathrsfs}
\usepackage{xcolor}
\usepackage[ansinew]{inputenc} 
\usepackage{amsthm}

\usepackage{arabtex}

\usepackage[frenchb]{babel}
\usepackage{niceframe}
\usepackage{tikz}

\usepackage{lettrine}   
\usepackage[pdftex,pdfstartview=FitH,colorlinks=true,linkcolor=blue,urlcolor=blue,%
citecolor=blue]{hyperref}

\font\toto=cmssbx10 scaled 1300
\newfont{\totoa}{cmss12}
\font\totonch=cmr10 scaled 8000
\makeatletter
\@addtoreset {equation}{section}
\@addtoreset {figure}{section}

\def\@makechapterhead#1{{\parindent \z@ \raggedright \normalfont
\interlinepenalty\@M
\ifnum \c@secnumdepth >\m@ne
\hbox{\vtop{\hsize 0.15\hsize
{\toto CHAPITRE}
}
\vtop{
\hsize 0.11\hsize
\vspace {-.3cm}
{\totonch  \thechapter\quad}
}
\vtop{\hsize 0.7\hsize
\vspace*{-0.3cm}\noindent\rule{11.55cm}{0.1cm}\newline
\fi\Huge \bfseries #1\par\nobreak\vskip 160\p@}
}
}}
\def\@schapter#1{\if@twocolumn
\@topnewpage[\@makeschapterhead{#1}]                 \else
\@makeschapterhead{#1}                   \@afterheading
\fi}
\def\@makeschapterhead#1{{\parindent \z@ \raggedright
\normalfont
\interlinepenalty\@M
\hbox{\vtop{ \hsize 0.3\hsize
\noindent\rule{1.05cm}{0cm}\newline
}
\vtop{ \hsize 0.6\hsize
\noindent\rule{11.05cm}{0.1cm}\newline
\Huge \bfseries #1\par\nobreak
\vskip 100\p@
}}
}}

\renewenvironment{proof}[1][Démonstration]{\noindent\textbf{#1.} }{\hfill\rule{0.5em}{0.5em}}

\newtheorem{thm}{Théorème}[chapter]  
\newtheorem{coll}[thm]{Corollaire}
\newtheorem{prop}[thm]{Proposition}
\newtheorem{lemme}[thm]{Lemme}

\newtheorem{con}{Conjecture}
\newtheorem{defi}[thm]{Définition}
\newtheorem{defi+}[thm]{Définition et propriétés}
\newtheorem{exemple}[thm]{Exemple}
\newtheorem{exemples}[thm]{Exemples}

\newtheorem{rmq}[thm]{Remarque}
\newtheorem{rmqs}[thm]{Remarques}

\def\indexname{Index}

{}
\fancypagestyle{plain}{ \fancyhead{}  }
\DeclareFontFamily{U}{rsfs}{\skewchar\font127 }
\DeclareFontShape{U}{rsfs}{ub}{sl}{<5> <6> rsfs5
<7> rsfs7
<8> <9> <10> <10.95> <12> <14.4> <17.28> <20.74> <24.88> rsfs10
}{}
\DeclareSymbolFont{rsfs}{U}{rsfs}{ub}{sl}
\DeclareSymbolFontAlphabet{\mathscr}{rsfs}
\DeclareSymbolFontAlphabet{\d}{rsfs}
\DeclareSymbolFontAlphabet{\C}{rsfs}
\DeclareSymbolFontAlphabet{\L}{rsfs}

\makeglossary
\makeindex

\def\ppcm{\mathrm{ppcm}}
\def\pgcd{\mathrm{pgcd}}

\def\lcm{\mathrm{lcm}}
\begin{document}

\begin{titlepage}

\enlargethispage{2\baselineskip}

\begin{center}
\textbf{République Algérienne Démocratique et Populaire}

\textbf{Ministère de l'Enseignement Supérieur et de la Recherche Scientifique}

\textbf{Université A. MIRA - BEJAIA}

\bigskip

\begin{tabular}{cc}
$\begin{array}{c}
  \hbox{\includegraphics[width=3.5cm]{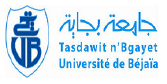}} \\
\end{array}$ & $\begin{array}{c} \hbox{{\bf Faculté des Sciences Exactes~~~~~}}\\

\hbox{{\bf Département de Mathématiques}}
\end{array}$
\end{tabular}

\textbf{Laboratoire de Mathématiques Appliquées (LMA)}

\bigskip

{\bf{\Large {\textbf{THÈSE}}}}

\textbf{EN VUE DE L'OBTENTION DU DIPLÔME DE DOCTORAT}

\bigskip

\textbf{Domaine:} Mathématiques et Informatique \hspace{0.5cm} \textbf{Filière:} Mathématiques\\
\vspace{0.2cm}
\textbf{Spécialité:} Théorie des Nombres\vspace{0.5cm}
\\ {\bf Présentée par} \vspace{0.2cm}
\\ {\bf {\large M. BOUSLA Sid Ali}} \vspace{1cm}
\\ {\Large {\it \textbf{Thème}}}
\rule{1\textwidth}{2pt}
{\Large{\bf{Estimations du plus petit commun multiple de certaines suites d'entiers}}}\vspace{-0.25cm}
\rule{1\textwidth}{2pt}
\vspace{0.7cm}
{\bf Soutenue le 02 décembre 2020 devant le jury composé de:} \vspace{0.2cm}
\begin{tabular}[c]{llllll}
M\textsuperscript{me}. {\sc TAS Saâdia} &  & Professeur & Univ. de Bejaia  & & Présidente \\
M. {\sc FARHI Bakir} &  & M.C.A & Univ. de Bejaia  & & Rapporteur \\
M\textsuperscript{me}. {\sc MOHDEB Nadia} &  & M.C.A & Univ. de  Bejaia & & Examinatrice \\ 
M. {\sc MOUSSAOUI Karim} & & Professeur & Univ. de Bejaia & & Examinateur \\
M. {\sc DAHMANI Abdelnasser} &  & Professeur & C. Univ. de Tamanrasset & & Examinateur \\
M. {\sc HERNANE Mohand Ouamar} &  & Professeur & U.S.T.H.B  & & Examinateur 
\end{tabular}
\vspace{0.7cm}

\textbf{Année Universitaire: 2019/2020}
\bigskip
\end{center}

\end{titlepage}

\frontmatter     
\newpage
\chapter*{Remerciements}~

\lettrine{$\mathscr{J}$}{e} tiens à remercier vivement mon directeur de thèse, le professeur {\sc FARHI Bakir}, de m'avoir introduit à la recherche et patiemment aidé tout au long de ce travail. Je lui suis profondément reconnaissant de m'avoir fait bénéficier de sa grande compétence, de son entière disponibilité, de son orientation depuis ma L2, de sa rigueur dans la rédaction des articles scientifiques et de ses conseils que je n'oublierai jamais. Il a été d'un soutien et d'une attention exceptionnels. Enfin, ses nombreuses relectures et corrections de mes travaux de thèse ont été très appréciables. Cette thèse lui doit beaucoup. Pour tout cela merci.

Je témoigne toute ma gratitude à tous ceux qui ont contribué à ma formation et je
tiens à remercier tous les enseignants du département de mathématiques.

Je suis très honoré de la présence à mon jury de thèse et je tiens à remercier:

Madame {\sc TAS Saâdia} pour m'avoir fait l'honneur de présider le jury de cette thèse.

Madame {\sc MOHDEB Nadia} et Monsieur {\sc MOUSSAOUI Abdelkrim} pour l'honneur qu'ils m'ont fait en acceptant d'être membres de mon jury de thèse. 


Monsieur {\sc DAHMANI Abdelnasser} pour l'honneur qu'il m'a fait par sa présence dans mon jury de soutenance en qualité d'examinateur de mon travail.


Monsieur {\sc HERNANE Mohand Ouamar} pour l'honneur qu'il m'a fait pour avoir accepté d'examiner ce travail.

   

Finalement je remercie ma famille et tous mes amis pour leurs encouragements et leur soutien qui m'ont été bien utiles durant ma thèse.

\begin{flushright}
{\sc BOUSLA Sid Ali}
\end{flushright}

\chapter*{Dédicaces}

\lettrine{$\mathscr{J}$}{e} dédie cette thèse à mes parents, à toute ma famille et mes amis et à tous ceux qui auront la patience de la lire.

%
%
%
%
%
%
%
%
%
%
%

\begin{flushright}
{\sc BOUSLA Sid Ali}
\end{flushright}

\newpage
\tableofcontents

\chapter*{Principales notations et conventions}
On écrit $f=O(g)$ ou d'une façon équivalente $f\ll g$, s'il existe une constante $C>0$ telle que $\left|f(x)\right| \leq C g(x)$, pour tout $x$ dans un voisinage d'une certaine valeur (éventuellement infinie). La première notation est due à Landau et la seconde est due à Vinogradov. Si le rapport $f(x)/g(x)$ tend vers $1$ quand $x$ tend vers $a \in \overline{\mathbb{R}}$, on écrit $f(x) \sim_{a} g(x)$ et on dit que $f$ et $g$ sont équivalentes au voisinage de $a$; de même on écrit $f=o(g)$ si le rapport $f(x)/g(x)$ tend vers zéro. Pour $t \in \mathbb{R}$, on désigne respectivement par $\lfloor t\rfloor$ et $\lceil t\rceil$ la partie entière par défaut et la partie entière par excès du nombre $t$. L'ensemble des nombres premiers est noté $P$, et dans toute la suite la lettre $p$ avec ou sans indice désigne un élément de $P$. Pour deux nombres réels $a$ et $b$, on note $a\mid b$ pour signifier que $a$ divise $b$ (i.e., le rapport $b/a$ est un entier). On désigne par $\pi$ la fonction de comptage des nombres premiers; soit:
\[\pi(x):=\sum_{p \leq x } 1 ~~~~~~(\forall x\in \mathbb{R^+}).\] 
Les fonctions $\psi$ et $\theta$ de Chebyshev sont définies comme suit:
\begin{align*}
\psi(x)&:= \log \ppcm\left(1,2,\dots,\lfloor x\rfloor\right) ~~~~~~(\forall x\geq 1),\\ \theta(x)&:=\sum_{p\leq x}\log p ~~~~~~(\forall x\in \mathbb{R^+}).
\end{align*}
Les fonctions généralisées de Chebyshev sont données par:
\[\begin{array}{c}
 \pi(x;m,k):=\displaystyle\sum_{\begin{subarray}{c} p\leq x \\ p\equiv m\!\!\!\pmod k\end{subarray}}1 \\  ~~~~\displaystyle\theta(x;m,k):=\sum_{\begin{subarray}{c} p\leq x \\ p\equiv m\!\!\!\pmod k\end{subarray}}\log p
\end{array}~~~~~(\forall x\in\mathbb{R^+},~\forall m,k\in\mathbb{N^*}).\]
On désigne par $\varphi$ la fonction indicatrice d'Euler qui à $n\in \mathbb{N^*}$, associe le nombre d'entiers compris entre $1$ et $n$ et premiers avec $n$. La fonction $\mu$ de Möbius est définie par:
\[\mu(n):=\begin{cases} (-1)^{\omega(n)}&\text{si n est sans facteur carré}>1, \\ 0 &\text{dans le cas contraire}.\end{cases}~~~~~~(\forall n\in \mathbb{N^*}),\]
où $\omega(n)$ désigne le nombre de facteurs premiers distincts de $n$. Pour un nombre premier $p$ donné, on désigne par $\vartheta_{p}$ la valuation $p$-adique usuelle (i.e., pour tout $n\in \mathbb{N^*}$, $\vartheta_{p}(n)$ est le plus grand exposant $\alpha\in\mathbb{N}$ tel que $p^{\alpha}$ divise $n$). Le plus petit commun multiple des entiers $a_{1},a_{2},\dots,a_{n}$ est noté $\ppcm(a_{1},a_{2},\dots,a_{n})$ ou bien $\ppcm\lbrace a_{1},a_{2},\dots,a_{n}\rbrace$; leurs plus grand commun diviseur est noté $\pgcd(a_{1},a_{2},\dots,a_{n})$ ou bien $\pgcd\lbrace a_{1},a_{2},\dots,a_{n}\rbrace$. Le cardinal d'un ensemble fini $\mathcal{A}$ est noté $\#\mathcal{A}$. On désigne par $\left(\frac{\cdot}{\cdot}\right)$ le symbole de Legendre. Nous utilisons souvent l'abréviation $\ppcm$ pour alléger l'expression \og plus petit commun multiple\fg{}. Certaines de ces notations sont rappelées localement.

\mainmatter    

\chapter*{Introduction générale}

\addcontentsline{toc}{chapter}{Introduction générale}


\lettrine{$\mathscr{A}$}{voir} une bonne estimation du plus petit commun multiple de termes consécutifs d'une suite d'entiers est un problème difficile et important. Pour la suite usuelle de tous les entiers naturels, la fonction $\psi$ est cruciale à la fois pour l'encadrement de Chebyshev (1852) de la fonction $\pi$ et pour le théorème des nombres premiers d'Hadamard-de la Vallée Poussin (1896) selon lequel on a:
\[\pi(x)\sim_{+\infty}\frac{x}{\log x}.\]
Ce théorème possède plusieurs autres énoncés équivalents, l'un de ces énoncés est celui de Chebyshev (1852), qui est donné par:
\[\log\ppcm(1,2,\dots,n)\sim_{+\infty} n.\]   
Ce qui équivaut à dire que pour tout $\varepsilon>0$, il existe $N=N_{\varepsilon}\in \mathbb{N}$, tel que l'on ait:
\[\left(e-\varepsilon\right)^{n}\leq \ppcm(1,2,\dots,n)\leq \left(e+\varepsilon\right)^{n}~~~~(\forall n\geq N).\]  
Par ailleurs, les résultats les plus significatifs concernant l'estimation effective des nombres $\ppcm(1,2,\dots,n)$ $(n\in\mathbb{N^*})$, sont dus à Chebyshev (1850), Hanson (1972) et Nair (1982). Dans \cite{cheb}, Chebyshev exploite l'idée d'estimer le nombre $\log (n!)$ de deux façons différentes: l'une est analytique et se sert de la formule de Stirling et l'autre est arithmétique et se sert de la formule de Legendre:
\[n!=\prod_{p~\text{premier}}p^{\left\lfloor\frac{n}{p}\right\rfloor+\left\lfloor\frac{n}{p^2}\right\rfloor+\dots}~~~~(\forall n\in\mathbb{N^*}).\]
Cette idée l'avait conduit à l'estimation:
\[ e^{-1}\cdot n^{-\frac{5}{2}}\left(c_{1}\right)^{n}\leq \ppcm(1,2,\dots,n)\leq e\cdot n^{\frac{5}{4}}e^{\frac{5}{4\log 6}\log^{2}n}\left(c_{2}\right)^{n}~~~~(\forall n\in\mathbb{N^*}),\]
avec $c_{1}\simeq 2,51$ et $c_{2}\simeq 3,02$. En utilisant le développement du nombre $1$ en série de Sylvester, Hanson \cite{han} a montré que $\ppcm(1,2,\dots,n)\leq 3^{n}$, pour tout entier $n\geq 1$; quant à Nair \cite{nair}, il exploite l'intégrale $\int_{0}^{1}x^{m-1}\left(1-x\right)^{n-m}\mathrm{d}x$ ($1\leq m\leq n$) pour montrer que $\ppcm(1,2,\dots,n)\geq 2^{n}$, pour tout entier $n\geq 7$. Bien que la minoration de Nair est plus faible que celle de Chebyshev, sa méthode d'obtention est plus brève et permet en outre d'envisager le problème différemment.

Dans le but de donner un analogue à la version explicite du théorème des nombres premiers, plusieurs auteurs se sont intéressés à l'estimation effective du $\ppcm$ de certaines suites d'entiers, comme les suites arithmétiques, les suites polynomiales et les suites à forte divisibilité. On rappelle qu'une suite d'entiers strictement positifs $\boldsymbol{a}=\left(a_n\right)_{n\geq 1}$ est dite \textit{\`a divisibilit\'e} lorsqu'elle v\'erifie la propri\'et\'e:
\[n \mid m \Rightarrow a_n \mid a_m~~~~(\forall n , m \in\mathbb{N^*}).\]
Elle est dite \textit{\`a forte divisibilit\'e} lorsqu'elle v\'erifie la propri\'et\'e plus forte:
\[\pgcd\left(a_n,a_m\right)=a_{\pgcd\left(n,m\right)}~~~~(\forall n,m\in\mathbb{N^*}).\]
La structure g\'en\'erale des suites \`a divisibilit\'e a \'et\'e le sujet d'int\'er\^et de plusieurs auteurs au moins depuis la seconde moiti\'e du $20^{\text{\`eme}}$ si\`ecle. En 1936, Ward \cite{Ward} \'etudie les valuations $p$-adiques de ces suites et découvre certaines de leurs propriétés. En 1990, B\'ezivin et al. \cite{bezi} ont \'etabli une caract\'erisation compl\`ete des suites \`a divisibilit\'e qui sont r\'ecurrentes lin\'eaires. Assez r\'ecemment, Bliss et al. \cite{Bliss} ont montr\'e le r\'esultat (figurant d\'ej\`a implicitement dans un article ant\'erieur de Kimberling \cite{kimb}) selon lequel \og{}le terme g\'en\'eral d'une suite \`a forte divisibilit\'e ${\boldsymbol{a}}=\left(a_n\right)_{n\geq 1}$ peut toujours s'\'ecrire sous la forme:
\[a_{n}=\prod_{d\mid n}u_d~~~~(\forall n\geq 1),\]
pour une certaine suite d'entiers strictement positifs $\left(u_{n}\right)_{n\geq 1}$\fg{}. Ce r\'esultat a permis aux auteurs de \cite{Bliss} d'\'etablir (dans le m\^eme contexte) une expression importante de $\left(u_n\right)_{n\geq 1}$ en fonction de $\left(a_n\right)_{n\geq 1}$, diff\'erente de celle qui s'obtient via la formule d'inversion de M\"obius (voir le théorème \ref{pourc11}). Bien que la r\'eciproque de leur r\'esultat soit fausse, Bliss et al. \cite{Bliss} ont r\'eussi \`a \'etablir une condition n\'ecessaire et suffisante sur une suite $\boldsymbol{u}=\left(u_n\right)_{n\geq 1}$ pour que la suite $\boldsymbol{a}$ d\'efinie par $a_n=\prod_{d\mid n}u_d$ $(\forall n\geq 1)$ soit \`a forte divisibilit\'e (voir le théorème \ref{divseq}). Une autre condition plus pratique, \'equivalente \`a celle-ci, a \'et\'e \'etablie tout r\'ecemment par Nowicki \cite{Nowicki} (voir le th\'eor\`eme \ref{nowi}). Pour une suite à forte divisibilité $\left(a_n\right)_{n\geq 1}$, Myerson (1994) se sert dans \cite{myer} d'un lemme de Kimberling \cite{kimb} pour montrer que: 
\[\ppcm(a_{1},a_{2},\dots,a_{n})~~\text{divise}~~\frac{a_{1}a_{2}\cdots a_{n}}{\left(\displaystyle\prod_{1\leq k\leq n/b_{1}}a_{k}\right)\left(\displaystyle\prod_{1\leq k\leq n/b_{2}}a_{k}\right)\left(\displaystyle\prod_{1\leq k\leq n/b_{3}}a_{k}\right)\cdots},\]
pour toute suite d'entiers strictement positifs $\left(b_{n}\right)_{n\geq 1}$ telle que: $\sum_{k\geq 1}1/b_{k}=1$. Par la suite, Farhi \cite{far} avait établi en 2005 que pour toute suite arithmétique $\left(u_{k}\right)_{k\in \mathbb{N}}$, dont la raison $r$ et le premier terme $u_{0}$ sont strictement positifs et premiers entre eux, on a:
\[\ppcm\left(u_{0},u_{1},\dots,u_{n}\right)\geq u_{0}\left(r+1\right)^{n-1}~~~~\left(\forall n\in\mathbb{N}\right).\]
En outre, Farhi avait conjecturé que l'exposant $(n-1)$ figurant dans cette minoration peut être remplacé par $n$, qui est l'exposant optimal que l'on peut obtenir. Cette conjecture a été confirmée par Hong et al. \cite{hong1} en 2006. Notons que Farhi utilise dans sa démonstration l'intégrale $\int_{0}^{1}x^{m+\frac{u_{0}}{r}-1}(1-x)^{n-m}\mathrm{d}x$ $(0\leq m\leq n)$. Plusieurs d'autres améliorations ont été établies, mais uniquement pour des valeurs de $n$ assez grandes en fonction de $u_0$ et $r$ (voir par exemple \cite{hong1,hongkom,kane}). Concernant la majoration du $\ppcm$ d'une progression arithmétique générale, aucun résultat n'est établi à ce jour et d'ailleurs, cela fait bien partie de notre recherche. La méthode de Hanson (utilisée pour majorer les nombres $\ppcm(1,2,\dots,n)$) semble non généralisable aux progressions arithmétiques. D'autre part, Farhi \cite{far} avait obtenu des minorations non triviales pour le $\ppcm$ d'une certaine classe de suites quadratiques; il obtient en particulier la minoration suivante:
\[\ppcm\left(1^{2}+1,2^2+1,\dots,n^{2}+1\right)\geq 0.32\left(1.442\right)^{n}~~~~\left(\forall n\in\mathbb{N^*}\right).\]
Cette dernière avait été améliorée par Oon \cite{oon} qui a montré plus généralement que pour tous $c,m,n\in\mathbb{N^*}$ tels que $m\leq \left\lceil \frac{n}{2}\right\rceil$, on a: 
\[\ppcm\left(m^{2}+c,(m+1)^2+c,\dots,n^{2}+c\right)\geq 2^{n}.\]
L'idée de Oon consiste à évaluer l'intégrale $\int_{0}^{1}x^{m-1+\sqrt{c}i}\left(1-x\right)^{n-m}\mathrm{d}x$ $(1\leq m\leq n)$ de deux façons différentes, généralisant ainsi la méthode de Nair \cite{nair}. Dans la continuation, Hong et al. \cite{hong2} avaient réussi à étendre la méthode de Oon aux suites polynomiales $(f(n))_{n\in\mathbb{N^*}}$, où $f\in\mathbb{Z}[X]$ est un polynôme non constant à coefficients positifs. Pour ce faire, ils ont dû compléter la méthode de Oon par de nouveaux arguments de la théorie algébrique des nombres.

Dans une autre direction, plusieurs estimations asymptotiques du $\ppcm$ de certaines suites d'entiers avaient été établies par divers auteurs. Parmi ces estimations, celle de Bateman et al. \cite{bat} qui énonce que pour tous $a,b\in\mathbb{Z}$, tels que $b>0$, $a+b>0$ et $\pgcd(a,b)=1$, on a:
\[\log\ppcm\left(a+b,a+2b,\dots,a+nb\right)\sim_{+\infty}\left(\frac{b}{\varphi(b)} \sum_{\begin{subarray}{c} 1\leq m\leq b \\ \pgcd(m,b)=1  \end{subarray}}\frac{1}{m}\right)n,\]
où $\varphi$ désigne la fonction indicatrice d'Euler. Ce résultat s'obtient comme conséquence du théorème des nombres premiers pour les progressions arithmétiques, qui lui même constitue une généralisation du théorème des nombres premiers (voir par exemple \cite{bla}, p. $72$). En désignant par $\left(F_{n}\right)_{n\geq 1}$ la suite de Fibonacci usuelle, définie récursivement par: $F_1=1$, $F_2=1$ et $F_{n+2}=F_n+F_{n+1}$ $(\forall n\geq 1)$; Matiyasevich et al. \cite{Mat} ont montré que:
\[\log\ppcm\left(F_1,F_2,\dots,F_n\right)\sim_{+\infty}\frac{3}{\pi^2}n^2\log\Phi,\]
où $\Phi$ désigne le nombre d'or ($\Phi:=\frac{1+\sqrt{5}}{2}$). Ce résultat a été généralisé par Kiss et al. \cite{kiss} pour les suites de Fibonacci généralisées à la place de $\left(F_n\right)_n$. Une autre estimation un peu plus complexe est due à Cilleruelo \cite{cil} qui énonce que pour tout polynôme quadratique irréductible $f\in \mathbb{Z}[X]$, on a:
\[\log \ppcm\left(f(1),f(2),\dots,f(n)\right)= n\log n +Bn+o(n),\] 
où $B$ est une constante qui dépend de $f$. Cette dernière estimation entraîne, en particulier, que:
\[\log \ppcm\left(f(1),f(2),\dots,f(n)\right)\sim_{+\infty} n\log n.\]
Pour le cas où $f(x)=x^2+1$, Rué et al. \cite{rue} ont amélioré le terme d'erreur de l'estimation précédente en montrant que l'on a pour tout $\alpha<\frac{4}{9}$:
\[\log\ppcm \left(1^2+1,2^2+1,\dots,n^2+1\right)=n\log n +Bn+O\left(\frac{n}{\left(\log n\right)^\alpha}\right).\]
Cilleruelo \cite{cil} a aussi donné une conjecture qui est ouverte à ce jour et qui généralise son estimation. Plus précisément, on a:

\begin{con}[Cilleruelo \cite{cil}]\label{cocil}
Pour tout polynôme irréductible $f\in\mathbb{Z}[X]$ tel que $\deg f\geq 3$, on a:
\[\log \ppcm\left(f(1),f(2),\dots,f(n)\right)\sim_{+\infty} (\deg f-1)n\log n.\] 
\end{con}

\noindent En outre, Cilleruelo avait précisé que cette conjecture est équivalente à dire que pour tout polynôme irréductible de degré $\geq 3$, on a:
\[\#\left\lbrace p~\text{premier};~n\leq p\ll n^{\deg f-1},~p\mid f(k),~p\mid f(j)~\text{pour certains}~1\leq j<k\leq n \right\rbrace=o(n).\]
   
\`A ce stade, nous avons presque introduit tout ce qui existe en littérature à propos des estimations du $\ppcm$ de certaines suites d'entiers.

\begin{center}
\textbf{Le but de la thèse}
\end{center}

Le but de cette thèse consiste en les trois points suivants:
\begin{enumerate}
\item Dans la minoration de Oon, le nombre de termes figurant dans le plus petit commun multiple $L_{c,m,n}:=\ppcm\left(m^{2}+c,(m+1)^2+c,\dots,n^{2}+c\right)$ est strictement plus grand que $n/2$. On se propose dans cette thèse de supprimer cette contrainte, en cherchant des minorations non triviales de $L_{c,m,n}$ pour des entiers strictement positifs $c,m,n$ quelconques (bien entendu $m\leq n$).
\item On s'intéresse aussi à estimer le $\ppcm$ de quelques suites non polynomiales, comme la suite de Fibonacci, les suites de Lucas ou plus généralement les suites à forte divisibilité. Le but ici est au moins d'effectiviser le résultat asymptotique de Matiyasevich et al. \cite{Mat} et sa généralisation par Kiss et al. \cite{kiss}.
\item Nous cherchons également à obtenir une première majoration effective et non triviale pour le $\ppcm$ d'une progression arithmétique finie, ce qui revient à préciser et effectiviser l'estimation asymptotique de Bateman et al. \cite{bat}.
\end{enumerate}  

\begin{center}
\textbf{Aperçu de la thèse}
\end{center}

Cette thèse est composée de cinq chapitres que nous décrivons brièvement ci-dessous:

Dans le premier chapitre, nous présentons des généralités sur les estimations du $\ppcm$ de suites d'entiers. La partie principale de notre travail se situe donc dans les chapitres \ref{ch2}, \ref{ch3}, \ref{ch4} et \ref{ch5}.

Le deuxième chapitre est consacré à l'étude des nombres 
\[L_{c,m,n}:=\ppcm\{m^2+c,(m+1)^2+c,\dots,n^2+c\},\]
où $c,m,n$ sont des entiers strictement positifs tels que $m\leq n$. Plus précisément, nous utilisons des arguments d'algèbre commutative et d'analyse complexe pour obtenir un diviseur rationnel et non trivial de $L_{c,m,n}$. Comme conséquence, nous établissons des nouvelles minorations non triviales pour $L_{c,m,n}$. 

Dans le troisième chapitre, nous obtenons des identités intéressantes concernant le $\ppcm$ de suites à forte divisibilité. Nous appliquons ensuite ces identités pour donner un premier encadrement effectif du $\ppcm$ d'une suite de Fibonacci généralisée; ce qui constitue en particulier une version effective du résultat de Matyasevich et al. \cite{Mat} et sa généralisation par Kiss et al. \cite{kiss}.

Dans le quatrième chapitre, nous établissons une première majoration effective du $\ppcm$ des termes consécutifs d'une suite arithmétique finie et nous donnons quelques conséquences de cette dernière. 

Dans le dernier chapitre, nous adaptons la méthode du quatrième chapitre pour la suite $(n^2+1)_n$ et nous établissons d'une part une amélioration de la minoration de Oon (pour le cas $c=1$) et d'autre part nous obtenons une première majoration effective du $\ppcm$ de cette suite.  
             
\markboth{}{Introduction}

\newpage

\chapter{Généralités sur les estimations du $\ppcm$ de suites d'entiers}\label{ch1}

\section{Introduction}
Ce chapitre rassemble les différents travaux réalisés dans le cadre des estimations du $\ppcm$ de certaines suites d'entiers. On se focalise notamment sur certains résultats effectifs dus à Chebyshev \cite{cheb}, Myerson \cite{myer}, Hanson \cite{han}, Nair \cite{nair}, Farhi \cite{Farhi1,far}, Oon \cite{oon} et Hong et al. \cite{hong1,hong2,hong3}. Nous terminons par une brève présentation de quelques résultats asymptotiques dus à Bateman et al. \cite{bat} et Cilleruelo \cite{cil}.   
 
\section{La méthode de Chebyshev}

Chebyshev fut le premier à avoir obtenu en 1850 des estimations effectives et significatives de la fonction de comptage des nombres premiers $\pi$. Ses célèbres fonctions $\psi$ et $\theta$ (définies ci-dessous) furent les points clef de sa méthode.
     
\begin{defi}
Les fonctions $\theta$, $\psi$, $T$ et $\chi$ de Chebyshev associent à tout réel positif $x$ les réels respectifs $\theta(x)$, $\psi(x)$, $T(x)$ et $\chi(x)$ définis comme suit:
\begin{align*}
\theta(x)&:=\sum_{p\leq x}\log p, \\ \psi(x)&:=\log \ppcm\left(1,2,\dots,\lfloor x\rfloor\right),\\ T(x)&:=\psi\left(x\right)+\psi\left(\frac{x}{2}\right)+\psi\left(\frac{x}{3}\right)+\dots, \\ \chi(x)&:=T(x)+T\left(\frac{x}{30}\right)-T\left(\frac{x}{2}\right)-T\left(\frac{x}{3}\right)-T\left(\frac{x}{5}\right).
\end{align*}
\end{defi}

\begin{rmq}
En se servant du fait que $\ppcm\left(1,2,\dots,\lfloor x\rfloor\right)=\prod_{p^{\alpha}\leq x}p$, où le produit étant étendu à tous les couples $(p,\alpha)$, avec $p$ premier et $\alpha\in \mathbb{N^{*}}$, on montre immédiatement que l'on a pour tout $x\in \mathbb{R^+}$:
\begin{equation}\label{chp}
\psi(x)=\theta(x)+\theta(x^{1/2})+\theta(x^{1/3})+\dots .
\end{equation}
De cette identité se déduit aussi l'expression de $T$ en fonction de $\theta$. On a la proposition suivante:
\end{rmq}

\begin{prop}[Chebyshev \cite{cheb}]\label{lcheb1}
Pour tout réel positif $x$, on a:
\begin{equation}\label{cheb1}
T(x)=\sum_{n\geq 1}\sum_{j\geq 1}\theta\left(\left(\frac{x}{n}\right)^{\frac{1}{j}}\right)=\log \left(\lfloor x\rfloor !\right).
\end{equation}
\end{prop}
\begin{proof}
La première égalité de \eqref{cheb1} résulte immédiatement de \eqref{chp}. Pour montrer la seconde égalité de \eqref{cheb1}, on se sert d'une part de l'identité $\ppcm(1,2,\dots,\lfloor y\rfloor)$ $=\prod_{p^{\alpha}\leq y}p ~~(\forall y\geq 0)$ et d'autre part de la formule des factoriels de Legendre. En effet, pour tout $x\geq 0$, on a:
\begin{align*}
\exp\left(T(x)\right)&=\prod_{j\geq 1}\prod_{n\geq 1}\left(\prod_{p^j\leq x/n}p\right)=\prod_{j\geq 1}\prod_{p^{j}\leq x}p^{\left\lfloor\frac{x}{p^j}\right\rfloor}\\&=\prod_{p\leq x}\left(\prod_{\begin{subarray}{c} j\geq 1 \\ p^{j}\leq x \end{subarray}}p^{\left\lfloor\frac{x}{p^j}\right\rfloor}\right)=\prod_{p\leq x}p^{\sum_{j\geq 1}\left\lfloor\frac{x}{p^j}\right\rfloor}=\lfloor x\rfloor!.
\end{align*}
Ce qui conclut à l'égalité requise et complète cette démonstration.
\end{proof}

\begin{lemme}[Chebyshev \cite{cheb}]\label{lcheb2}
Pour tout réel positif $x$, on a:
\[\psi(x)-\psi\left(\frac{x}{6}\right)\leq \chi(x) \leq \psi(x).\]
\end{lemme}
\begin{proof}
En exprimant $\chi$ en fonction de $\psi$, on obtient pour tout $x\geq 0$:
\[\chi(x)=A_{1}\psi(x)+A_{2}\psi\left(\frac{x}{2}\right)+\dots+A_{n}\psi\left(\frac{x}{n}\right)+\dots,\]
où les coefficients $A_{1},A_{2},\dots,A_{n},\dots$ prennent les valeurs $0$, $1$ ou $-1$ selon les restes de leurs indices modulo $30$. On a plus précisément:\\
$A_{n}=1$, si $n\equiv 1$, $7$, $11$, $13$, $17$, $19$, $23$, $29$ $\pmod {30}$, \\
$A_{n}=0$, si $n\equiv 2$, $3$, $4$, $5$, $8$, $9$, $14$, $16$, $21$, $22$, $25$, $26$, $27$, $28$ $\pmod {30}$,\\
$A_{n}=-1$, si $n\equiv 0$, $6$, $10$, $12$, $15$, $18$, $20$, $24$ $\pmod {30}$.\\
Pour montrer ce fait, il suffit de remarquer que pour tout $i\in\lbrace 1,2,3,5,30\rbrace$, le nombre $T\left(\frac{x}{i}\right)$ est la somme des nombres $\psi\left(\frac{x}{n}\right)$, où $n$ parcourt l'ensemble des multiples de $i$. On a ainsi pour tout $x\geq 0$:
\[\chi(x)=\psi(x)-\psi\left(\frac{x}{6}\right)+\psi\left(\frac{x}{7}\right)-\psi\left(\frac{x}{10}\right)+\psi\left(\frac{x}{11}\right)-\psi\left(\frac{x}{12}\right)+\dots.\]
Constatant que cette série est alternée et que ses termes décroissent continuellement en valeurs absolues, on peut l'encadrer entre son premier terme et la somme de ses deux premiers termes; soit
\[\psi(x)-\psi\left(\frac{x}{6}\right)\leq \chi(x) \leq \psi(x)~~~~(\forall x\geq 0).\]
Le lemme est ainsi démontré.   
\end{proof}

\begin{thm}[Chebyshev \cite{cheb}]\label{chebbb}
pour tout réel $x\geq 1$, on a:
\[Ax-\frac{5}{2}\log x-1 \leq \psi(x)\leq \frac{6}{5}Ax +\frac{5}{4\log 6}\log^{2}x +\frac{5}{4}\log x +1,\]
où $A:=\log\left(\frac{2^{1/2}3^{1/3}5^{1/5}}{30^{1/30}}\right)\simeq 0,92129202$.
\end{thm}
\begin{proof}
En posant $n=\lfloor x\rfloor$, on a l'estimation de Stirling suivante:
\[n^{n}e^{-n}\sqrt{2\pi n} \leq n!\leq n^{n}e^{-n}\sqrt{2\pi n}\cdot e^{\frac{1}{12n}}\]
(voir \cite[Problème $1.15$]{kon}), qui entraîne (en vertu de la proposition \ref{lcheb1}) que:
\begin{align*}
&T(x)\leq \log\left(\sqrt{2\pi}\right)+n\log n -n +\frac{1}{2}\log n +\frac{1}{12n}\\ &~~~~~~\leq \log\left(\sqrt{2\pi}\right)+x\log x -x +\frac{1}{2}\log x +\frac{1}{12},\\& T(x)\geq  \log\left(\sqrt{2\pi}\right)+(n+1)\log (n+1)-(n+1) -\frac{1}{2}\log (n+1)\\&~~~~~~\geq\log\left(\sqrt{2\pi}\right)+x\log x-x -\frac{1}{2}\log x.
\end{align*}
En utilisant ces dernières estimations de la fonction $T$ (qui sont valables pour tout $x\geq 1$), on obtient les estimations suivantes (valables pour tout $x\geq 30$):
\begin{align*}
&T(x)+T\left(\frac{x}{30}\right)\leq 2\log \left(\sqrt{2\pi}\right)+\frac{2}{12}+\frac{31}{30}x\log x -x\log \left(30^{1/30}\right)-\frac{31}{30}x+\log x-\frac{1}{2}\log 30,\\&T(x)+T\left(\frac{x}{30}\right)\geq 2\log \left(\sqrt{2\pi}\right)+\frac{31}{30}x\log x -x\log \left(30^{1/30}\right)-\frac{31}{30}x-\log x+\frac{1}{2}\log 30,\\&T\left(\frac{x}{2}\right)+T\left(\frac{x}{3}\right)+T\left(\frac{x}{5}\right)\leq 3\log\left(\sqrt{2\pi}\right)+\frac{3}{12}+\frac{31}{30}x\log x-x\log\left(2^{1/2}3^{1/3}5^{1/5}\right)-\frac{31}{30}x\\&~~~~~~~~~~~~~~~~~~~~~~~~~~~~~~~~~~+\frac{3}{2}\log x -\frac{1}{2}\log 30,\\ &T\left(\frac{x}{2}\right)+T\left(\frac{x}{3}\right)+T\left(\frac{x}{5}\right)\geq 3\log\left(\sqrt{2\pi}\right)+\frac{31}{30}x\log x-x\log\left(2^{1/2}3^{1/3}5^{1/5}\right)-\frac{31}{30}x\\&~~~~~~~~~~~~~~~~~~~~~~~~~~~~~~~~~~-\frac{3}{2}\log x +\frac{1}{2}\log 30.
\end{align*} 
Il découle de ces dernières estimations et de la définition même de la fonction $\chi$ que l'on a pour tout $x\geq 30$:
\[ Ax-\frac{5}{2}\log x -1\leq \chi(x)\leq Ax+\frac{5}{2}\log x ,\]
avec $A:=\log\left(\frac{2^{1/2}3^{1/3}5^{1/5}}{30^{1/30}}\right)$. D'autre part, un calcul manuel montre que cette double inégalité reste également vraie pour tout $1\leq x\leq 30$; elle est, par conséquent, vraie pour tout $x\geq 1$. Il s'ensuit, en vertu du lemme \ref{lcheb2} que l'on a pour tout $x\geq 1$:
\begin{align}
&\psi(x)-\psi\left(\frac{x}{6}\right)\leq Ax+\frac{5}{2}\log x,\label{der1}\\&\psi(x)\geq Ax-\frac{5}{2}\log x-1.\label{der2}
\end{align} 
L'estimation \eqref{der2} fournit la minoration requise de $\psi(x)$. Pour montrer la majoration requise pour $\psi(x)$, on considère la fonction réelle $f$ définie sur l'intervalle $]0,+\infty[$ par:
\[f(x):=\frac{6}{5}Ax+\frac{5}{4\log 6}\log^{2}x+\frac{5}{4}\log x ~~~~(\forall x>0).\] 
On vérifie immédiatement que l'on a pour tout $x>0$:
\[f(x)-f\left(\frac{x}{6}\right)=Ax+\frac{5}{2}\log x .\]
Cette égalité retranchée membre à membre de \eqref{der1} donne:
\[\psi(x)-f(x)\leq \psi\left(\frac{x}{6}\right)-f\left(\frac{x}{6}\right)~~~~(\forall x>0).\]
En réitérant cette dernière plusieurs fois, on obtient que pour tout $x>0$ et tout $m\in \mathbb{N^*}$, on a:
\[\psi(x)-f(x)\leq \psi\left(\frac{x}{6}\right)-f\left(\frac{x}{6}\right)\leq \psi\left(\frac{x}{6^{2}}\right)-f\left(\frac{x}{6^{2}}\right)\leq \dots \leq \psi\left(\frac{x}{6^{m+1}}\right)-f\left(\frac{x}{6^{m+1}}\right),\]
En prenant $m$ le plus grand entier positif vérifiant $\frac{x}{6^{m}}\geq 1$, on aura $\frac{x}{6^{m+1}}\in [\frac{1}{6},1[$; de sorte que $\psi\left(\frac{x}{6^{m+1}}\right)=0$ et $-f\left(\frac{x}{6^{m+1}}\right)\leq 1$; on obtient donc:
\[\psi(x)\leq 1+f(x)~~~~(\forall x>0)\]
et l'on conclut à la majoration requise pour $\psi(x)$ en substituant dans cette dernière $f(x)$ par l'expression qui la définit. Ce qui complète cette démonstration.
\end{proof}

\noindent Le corollaire suivant est immédiat.
\begin{coll}
Pour tout $n\in \mathbb{N^*}$, on a:
\[ e^{-1}\cdot n^{-\frac{5}{2}}\left(c_{1}\right)^{n}\leq \ppcm(1,2,\dots,n)\leq e\cdot n^{\frac{5}{4}}e^{\frac{5}{4\log 6}\log^{2}n}\left(c_{2}\right)^{n},\]
avec $c_{1}=e^A\simeq 2,51$ et $c_{2}=e^{\frac{6}{5}A}\simeq 3,02$.
\end{coll}

Le théorème \ref{chebbb} entraîne aussi un encadrement effectif pour la fonction $\theta$ de Chebyshev. Un tel encadrement est donné par le corollaire suivant:

\begin{coll}[Chebyshev \cite{cheb}]\label{theta}
Pour tout réel $x\geq 1$, on a:
\[Ax-\frac{12}{5}x^{1/2}-\frac{5}{8\log 6}\log^2 x-\frac{15}{4}\log x-3\leq \theta(x)\leq \frac{6}{5}Ax +\frac{5}{4\log 6}\log^{2}x +\frac{5}{4}\log x +1.\]
En particulier, on a:
\begin{equation}\label{saeq}
Ax+o\left(x\right)\leq\theta(x)\leq \frac{6}{5}Ax+o\left(x\right)~~~~(\forall x\geq 1),
\end{equation}
où les termes en $o$ sont explicites.
\end{coll}
\begin{proof}
Pour $x\geq 1$, on a d'une part (d'après \eqref{chp}) $\theta(x)\leq \psi(x)$, et d'autre part:
\[\theta(x)\geq \theta(x)-\left[\theta(x^{1/2})-\theta(x^{1/3})\right]-\left[\theta(x^{1/4})-\theta(x^{1/5})\right]-\dots= \psi(x)-2\psi(\sqrt{x}),\] 
car les termes $\theta(x^{1/2})-\theta(x^{1/3}),\theta(x^{1/4})-\theta(x^{1/5}),\dots$ sont tous positifs. On conclut au résultat requis grâce à l'encadrement du théorème \ref{chebbb}. 
\end{proof}

Dans ce qui suit, nous indiquons brièvement comment l'on peut obtenir un encadrement effectif de la fonction de comptage des nombres premiers $\pi$. Vu la complexité des estimations de Chebyshev, nous n'allons pas rentrer dans les détails des calculs.

\begin{thm}[Chebyshev \cite{cheb}]\label{dcheb}
Pour tout réel $x\geq 2$, on a:
\[\left(A+o(1)\right)\frac{x}{\log x}\leq\pi(x)\leq \frac{x}{\log x}\left(\frac{6}{5}A+o(1)\right),\]
où les termes en $o$ sont explicites.
\end{thm}
\begin{proof}
Premièrement, on a en vertu de la formule sommatoire d'Abel (voir par exemple \cite[p. 5]{kon}):
\begin{equation}\label{abel1}
\pi(x)=\frac{\theta(x)}{\log x}+\int_{2}^{x}\frac{\theta(t)}{t\log^2 t}\mathrm{d}t~~~~(\forall x\geq 2).
\end{equation}
D'après l'inégalité \eqref{saeq}, on peut estimer l'intégrale précédente comme suit:
\[\int_{2}^{x}\frac{\theta(t)}{t\log^2 t}\mathrm{d}t \leq \frac{6}{5}A\int_{2}^{x}\frac{\mathrm{d}t}{\log^2 t}+o\left(\int_{2}^{x}\frac{\mathrm{d}t}{\log^2 t}\right).\]
Ensuite, comme on a:
\begin{equation}\label{psuite}
\int_{2}^{x}\frac{\mathrm{d}t}{\log^2 t}=\int_{2}^{\sqrt{x}}\frac{\mathrm{d}t}{\log^2 t}+\int_{\sqrt{x}}^{x}\frac{\mathrm{d}t}{\log^2 t}\leq \frac{\sqrt{x}}{\log^2 2}+\frac{4x}{\log^2 x}=O\left(\frac{x}{\log^2 x}\right),
\end{equation} 
il s'ensuit que:
\[\int_{2}^{x}\frac{\theta(t)}{t\log^2 t}\mathrm{d}t=O\left(\frac{x}{\log^2 x}\right).\]
En reportant ceci dans \eqref{abel1} et en utilisant de nouveau l'estimation \eqref{saeq}, on obtient:
\[\pi(x)\leq \frac{6}{5}A\frac{x}{\log x}+O\left(\frac{x}{\log^2 x}\right)=\frac{6}{5}A\frac{x}{\log x}\left(1+o(1)\right).\]
D'autre part, on a en vertu de \eqref{saeq} et de \eqref{abel1}:
\[\pi(x)\geq \frac{\theta(x)}{\log x}\geq A\frac{x}{\log x}+o\left(\frac{x}{\log x}\right)=A\frac{x}{\log x}\left(1+o(1)\right).\]
Ce qui complète cette démonstration.
\end{proof}

\noindent Le corollaire suivant est immédiat:

\begin{coll}[Chebyshev \cite{cheb}]\label{chch}
Il existe $x_0>0$ effectivement calculable, tel que:
\[0,9\frac{x}{\log x}\leq \pi(x)\leq 1,2\frac{x}{\log x}~~~~(\forall x\geq x_0).\]
\end{coll}

\begin{rmq}\label{rrrrm}
Soit $x\geq 2$. Puisque $\ppcm\left(1,2,\dots,\lfloor x\rfloor\right)=\prod_{p^{\nu}\leq x}p$, où le produit étant étendu à tous les couples $(p,\nu)$, avec $p$ premier et $\nu\in \mathbb{N^{*}}$, alors on a:
\[\psi(x)=\sum_{p^\nu\leq x}\log p.\]
Par ailleurs, pour chaque nombre premier $p$ fixé, on a:
\[p^\nu\leq x\Longleftrightarrow \nu\leq \left\lfloor \frac{\log x}{\log p}\right\rfloor.\]
Il y a donc exactement $\left\lfloor\frac{\log x}{\log p}\right\rfloor$ valeurs de $\nu$ tels que $p^\nu\leq x$, ce qui permet d'écrire:
\[\psi(x)=\sum_{p\leq x}\left\lfloor\frac{\log x}{\log p}\right\rfloor\log p.\]
Par suite, grâce à l'encadrement $\lfloor u\rfloor\leq u \leq 2\lfloor u\rfloor$ valable pour $u\geq 1$, on a:
\[\psi(x)\leq \sum_{p\leq x}\log x=\pi(x)\log x=\sum_{p\leq x}\left(\frac{\log x}{\log p}\right)\log p\leq 2\sum_{p\leq x}\left\lfloor\frac{\log x}{\log p}\right\rfloor\log p,\]
soit
\begin{equation}\label{ustt}
\frac{\psi(x)}{\log x}\leq\pi(x)\leq 2\frac{\psi(x)}{\log x}.
\end{equation}
Ce qui permet (via le théorème \ref{chebbb}) d'obtenir un autre encadrement de $\pi(x)$ $(\forall x\geq 1)$.
\end{rmq}

Le corollaire \ref{chch} et la remarque \ref{rrrrm} impliquent qu'il existe un réel $N$ effectivement calculable tel que $\pi(2n-2)>\pi(n)$ pour tout entier $n\geq N$. En explicitant les termes $o(1)$ dans le théorème \ref{dcheb}, Chebychev a pu prouver que cette inégalité est bien valable pour tout entier $n>3$, ce qui constitue une preuve d'une conjecture célèbre, connue sous le nom de \og postulat de Bertrand\fg{} (datant de 1845): 

\noindent\textbf{Le postulat de Bertrand:} Pour tout entier $n\geq 2$, il existe au moins un nombre premier $p$ tel que:
\[n<p<2n.\]

Notons que l'estimation des fonctions liées aux nombres premiers (comme les fonctions de Chebyshev) pour des valeurs de $x$ assez petites se fait généralement à la main en se servant des tables de nombres premiers ou d'un logiciel de calcul. Nous verrons dans les prochaines sections, que l'étude de certaines propriétés arithmétiques des nombres $d_n:=\ppcm(1,2,\dots,n)$ $(n\in\mathbb{N^*})$, permet d'obtenir d'autres estimations effectives plus simples pour les fonctions $\psi$ et $\pi$.

\section{Un multiple du $\ppcm$ de suites à forte divisibilité}\label{for}

Le $\ppcm$ des $n$ premiers termes d'une suite d'entiers naturels divise de toute évidence le produit de ces termes, mais en général ce produit est beaucoup plus grand que le $\ppcm$ en question. Dans cette section, nous présentons un multiple non trivial du nombre $\ppcm (a_{1},a_{2},\dots,a_{n})$ pour une certaine classe de suites $a_{1},a_{2},\dots$ d'entiers strictement positifs.

\begin{defi}
Une suite d'entiers strictement positifs $\left(a_n\right)_{n\geq 1}$ est dite \textit{à divisibilité} lorsqu'elle vérifie la propriété:
\[n \mid m \Longrightarrow a_n \mid a_m~~~~(\forall n , m \in \mathbb{N^*}).\]
Elle est dite \textit{à forte divisibilité} lorsqu'elle vérifie la propriété plus forte:
\[\pgcd\left(a_n,a_m\right)=a_{\pgcd\left(n,m\right)}~~~~(\forall n,m\in\mathbb{N^*}).\]
\end{defi}

\begin{exemples}~

\begin{itemize}
\item Une suite à divisibilité n'est pas forcément à forte divisibilité. En effet, la suite $\left(2^n\right)_{n\geq 1}$ est à divisibilité mais elle n'est pas à forte divisibilité, car on a par exemple: $\pgcd\left(2^2,2^3\right)=4\neq 2^{\pgcd(2,3)}=2$.
\item La suite de tous les entiers strictement positifs $(n)_{n\geq 1}$ est de toute évidence à forte divisibilité.
\item Pour tous $m,n\in\mathbb{N^*}$, si l'on désigne par $r$ le reste de la division euclidienne de $n$ sur $m$, on montre facilement que $(2^r-1)$ est le reste de la division euclidienne de $(2^n-1)$ par $(2^m-1)$. Ce qui entraîne (via l'algorithme d'Euclide) que la suite $(2^n-1)_{n\geq 1}$ est à forte divisibilité.
\item Un autre exemple très important d'une suite à forte divisibilité est la suite de Fibonacci usuelle $(F_{n})_{n\geq 1}$ (voir \cite{vorob}, p. 34) qui est définie récursivement par: $F_{1}=1$, $F_{2}=1$ et $F_{n+2}=F_{n}+F_{n+1}$ pour tout entier $n\geq 1$.
\item Si deux suites $(t_{n})_{n\geq 1}$ et $(s_{n})_{n\geq 1}$ sont à forte divisibilité alors, les suites $(t_{n}^{\alpha})_{n\geq 1}$ $(\alpha\in\mathbb{N^*})$ et $(t_{s_n})_{n\geq 1}$ le sont aussi (voir \cite{kimb}).
\end{itemize}
\end{exemples}

\begin{thm}[Myerson \cite{myer}]\label{ts1}
Soient $(a_{n})_{n\geq 1}$ une suite à forte divisibilité et $(b_{n})_{n\geq 1}$ une suite d'entiers strictement positifs vérifiant $\sum_{n\geq 1}\frac{1}{b_{n}}\leq 1$. Alors, on a pour tout $n\geq 1$:
\begin{equation}\label{ds}
\ppcm(a_{1},a_{2},\dots,a_{n})~~\text{divise}~~\frac{a_{1}a_{2}\cdots a_{n}}{\left(\displaystyle\prod_{1\leq k\leq n/b_{1}}a_{k}\right)\left(\displaystyle\prod_{1\leq k\leq n/b_{2}}a_{k}\right)\left(\displaystyle\prod_{1\leq k\leq n/b_{3}}a_{k}\right)\cdots}.
\end{equation} 
\end{thm}

\begin{rmq}
Un exemple important d'une suite $(b_{n})_{n\geq 1}$ vérifiant les conditions du théorème \ref{ts1} est la suite de Sylvester, qui est définie par $b_{1}=2$ et $b_{n+1}=b_{1}b_{2}\cdots b_{n}+1$ $(\forall n\in\mathbb{N^*})$. On vérifie facilement que pour tout $n\in\mathbb{N^*}$, on a:
\begin{equation}\label{sylv1}
b_{n+1}=b_{n}^{2}-b_{n}+1
\end{equation}
et
\begin{equation}\label{sylv2}
\sum_{\ell=1}^{n}\frac{1}{b_{\ell}}+\frac{1}{b_{1}b_{2} \cdots b_{n}}=1.
\end{equation}
Nous ferons appel à ces identités dans la section suivante.
\end{rmq}

\begin{exemple}\label{exemple1}
En prenant dans le théorème \ref{ts1} $a_{n}=n$ $(\forall n\geq 1)$ et $(b_{n})_{n\geq 1}$ la suite de Sylvester, on obtient que pour tout $n\in \mathbb{N^*}$, on a:
\[\ppcm (1,2,\dots,n)~~divise~~\frac{n!}{\lfloor n/2 \rfloor ! \lfloor n/3 \rfloor ! \lfloor n/7\rfloor !\lfloor n/43\rfloor !\cdots},\]
où $2,3,7,43,\dots$ est la suite de Sylvester.
\end{exemple}

\noindent Nous partageons la preuve du théorème \ref{ts1} en plusieurs lemmes.\\
Le lemme suivant est une conséquence immédiate du principe d'inclusion-exclusion (voir Hua \cite[p. 10]{Hua}).

\begin{lemme}[\cite{Hua}, Theorem 7.3]\label{pourc10}
Soient $n\in\mathbb{N^*}$ et $S=\left\lbrace x_1,x_2,\dots,x_n\right\rbrace$ un sous ensemble de $\mathbb{N^*}$. Alors, on a l'identité suivante:
\[\prod_{i=1}^{n}x_i=\ppcm\left(S\right)\prod_{\begin{subarray}{c} A\subset S \\ \left| A\right| \geq 2\end{subarray}}\pgcd\left(A\right)^{(-1)^{\left| A\right|}},\]
où $\left| A\right|$, $\pgcd\left(A\right)$ et $\ppcm\left(S\right)$ désignent respectivement le cardinal de l'ensemble $A$, le plus grand diviseur commun des éléments de $A$ et le $\ppcm$ des éléments de $S$.
\end{lemme}

\noindent Fixons une suite à forte divisibilité $\left(a_{n}\right)_{n\geq 1}$ et considérons $u_{1},u_{2},\dots$ les nombres rationnels strictement positifs définis récursivement par la formule:
\begin{equation}\label{eqqq}
a_{n}=\prod_{d \mid n}u_{d}.
\end{equation}
D'après la formule d'inversion de Möbius (voir par exemple \cite[p. 35]{ten}), on a:
\begin{equation}\label{corrr}
u_{n}=\prod_{d \mid n}{a}^{\mu(d)}_{n/d}.
\end{equation}

Le lemme suivant donne quelques propriétés de la suite $\left(u_n\right)_n$. Il figure d\'ej\`a implicitement dans un article ant\'erieur de Kimberling \cite{kimb}. Nous donnons ici sa version explicite établie par Bliss et al. \cite{Bliss}.

\begin{lemme}[Bliss et al. \cite{Bliss}]\label{pourc11}
Les nombres $u_1,u_2,\dots$ sont tous des entiers strictement positifs. De plus, si $n\in\mathbb{N^*}$ possède la factorisation primaire $n={q_1}^{\alpha_1}{q_2}^{\alpha_2}\cdots {q_m}^{\alpha_m}$, on a:
\begin{equation}\label{pourc12}
u_n=\frac{a_n}{\ppcm\left(a_{n/q_1},a_{n/q_2},\dots,a_{n/q_m}\right)}.
\end{equation}
\end{lemme}
\begin{proof}
Pour tout $\ell\in\{0,1,\dots,m\}$, désignons par $Q_{\ell}$ le produit de tous les $a_d$ tels que $\frac{n}{d}$ soit le produit de $\ell$ nombres premiers distincts. Puisque $\mu\left(\frac{n}{d}\right)=0$ si et seulement si $\frac{n}{d}$ possède un facteur carré $>1$ (i.e., il existe un nombre premier $p$ tel que $p^2\mid\frac{n}{d}$), alors on a:
\begin{equation}\label{bn}
u_n=\prod_{d\mid n}a_{d}^{\mu\left(n/d\right)}=\prod_{\ell=0}^{m}Q_{\ell}^{(-1)^{\ell}}.
\end{equation}
Par ailleurs, d'après le lemme \ref{pourc10} qu'on applique pour $x_i=a_{n/q_i}$ ($\forall i\in\{1,2,\dots,m\}$) et $S=\left\lbrace x_1,x_2,\dots,x_m\right\rbrace$, on a:
\begin{equation}\label{Q1}
Q_1 =\prod_{i=1}^{m}x_i=\ppcm\left(S\right)\prod_{\begin{subarray}{c} A\subset S \\ \left| A\right| \geq 2\end{subarray}}\pgcd\left(A\right)^{(-1)^{\left| A\right|}}=\ppcm\left(S\right)\prod_{k=2}^{m}\prod_{\begin{subarray}{c} A\subset S \\ \left| A\right| =k\end{subarray}}\pgcd\left(A\right)^{(-1)^{k}},
\end{equation}
où la première égalité vient de la définition de $Q_1$. Par suite, comme $\left(a_k\right)_{k\geq 1}$ est à forte divisibilité, on a pour tout $k\in\{2,\dots,m\}$ et pour tous $1\leq i_1<i_2<\dots <i_k\leq m$:
\[\pgcd\left(x_{i_1},x_{i_2},\dots,x_{i_k}\right)=a_{\pgcd\left(n/q_{i_1},n/q_{i_2},\dots,n/q_{i_k}\right)}=a_{n/\left(q_{i_1}q_{i_2}\cdots q_{i_k}\right)}.\]
Il s'ensuit de cette dernière et de la définition de $Q_{k}$ $(k\in\mathbb{N})$ que:
\begin{equation}\label{elemh}
Q_k=\prod_{1\leq i_1<i_2<\dots <i_k\leq m}a_{n/\left(q_{i_1}q_{i_2}\cdots q_{i_k}\right)}=\prod_{\begin{subarray}{c} A\subset S \\ \left| A\right| =k\end{subarray}}\pgcd\left(A\right)~~~~(\forall k\in\{2,\dots,m\}).
\end{equation}
En combinant \eqref{Q1} et \eqref{elemh}, on obtient:
\[Q_1=\ppcm\left(S\right)\prod_{k=2}^{m}Q_{k}^{(-1)^k}.\]
D'où l'on a (en vertu de \eqref{bn}):
\begin{align*}
u_n &=\frac{Q_{0}\prod_{k=2}^{m}Q_{k}^{(-1)^k}}{Q_1}=\frac{Q_{0}\prod_{k=2}^{m}Q_{k}^{(-1)^k}}{\ppcm\left(S\right)\prod_{k=2}^{m}Q_{k}^{(-1)^k}}=\frac{Q_{0}}{\ppcm\left(S\right)}\\ &=\frac{a_n}{\ppcm\left(a_{n/q_1},a_{n/q_2},\dots,a_{n/q_m}\right)},
\end{align*}
ce qui confirme l'identité \eqref{pourc12}. Maintenant, puisque $a_n$ est multiple de chacun des $a_{n/q_i}$ $(\forall i\in\{1,\dots,m\})$ (car $\left(a_k\right)_{k\geq 1}$ est à forte divisibilité), alors $a_n$ est multiple de $\ppcm\left(a_{n/q_1},a_{n/q_2},\dots,a_{n/q_m}\right)$, ce qui entraîne que les nombres $u_1,u_2,\dots$ sont tous des entiers strictement positifs. Le lemme est ainsi démontré.    
\end{proof}

\begin{defi}
Soient $m\in \mathbb{N^*}$ et $p$ un nombre premier. On dit que $m$ est ``champion pour $p$" si $(m=1$ et $\vartheta_{p}(a_{1})>0)$ ou $(m>1$ et $\vartheta_{p}(a_{m})>\vartheta_{p}(a_{j})$ pour tout $j<m)$.
\end{defi}

\begin{rmq}\label{rmq1s}
Soient $m_{1},m_{2}\in \mathbb{N^*}$ et $p$ un nombre premier. Si $m_{1}<m_{2}$ sont tous les deux champions pour $p$, alors $m_{1}$ divise $m_{2}$. En effet, si $d=\pgcd(m_{1},m_{2})$ alors: $\vartheta_{p}(a_{d})=\min (\vartheta_{p}(a_{m_{1}}),\vartheta_{p}(a_{m_{2}}))=\vartheta_{p}(a_{m_{1}})$, ce qui implique que $d=m_{1}$ (car: $d\leq m_1$ et $m_{1}$ est champion pour $p$).  
\end{rmq} 

\begin{lemme}[Myerson \cite{myer}]\label{lc}
Pour tout $m \in \mathbb{N^*}$ et tout nombre premier $p$, on a: $\vartheta_{p}(u_{m})>0$ si et seulement si $m$ est champion pour $p$.
\end{lemme}
\begin{proof}
Soient $m\in\mathbb{N^*}$ et $p$ un nombre premier.\\
$\bullet (\Rightarrow):$ Procédons par l'absurde et supposons que $\vartheta_{p}(u_{m})>0$ et $\vartheta_{p}(a_{m})\leq \vartheta_{p}(a_{j})$ pour un certain $j<m$. D'une part, on a (en vertu de \eqref{eqqq}) pour tout diviseur propre $d'$ de $m$:
\[\vartheta_{p}(a_{m})=\sum_{d\mid m}\vartheta_{p}\left(u_d\right)=\vartheta_{p}\left(u_m\right)+\sum_{d\mid d'}\vartheta_{p}\left(u_d\right)+\sum_{\begin{subarray}{c} d\mid m \\ d\neq m \\ d\nmid d'\end{subarray}}\vartheta_{p}\left(u_d\right)>\sum_{d\mid d'}\vartheta_{p}\left(u_d\right)=\vartheta_{p}(a_{d'}),\]
(où la deuxième égalité vient du fait que tout diviseur de $d'$ est aussi un diviseur de $m$), et d'autre part, pour $d':=\pgcd(m,j)$ on a: $\vartheta_{p}(a_{d'})=\min\left(\vartheta_{p}(a_{m}),\vartheta_{p}(a_{j})\right)=\vartheta_{p}(a_{m})$, ce qui contredit le fait que $d'=\pgcd(m,j)<m$ est un diviseur propre de $m$. Ceci confirme l'implication directe du lemme.\\  $\bullet (\Leftarrow):$ Supposons que $m$ est champion pour $p$. D'après le lemme \ref{pourc11}, on a: $\vartheta_{p}(u_{m})=\vartheta_{p}(a_{m})-\vartheta_{p}(a_{m/q})$ pour un certain facteur premier $q$ de $m$. Puisque $m/q<m$ et $m$ est champion pour $p$, alors $\vartheta_{p}(a_{m})>\vartheta_{p}(a_{m/q})$ et donc $\vartheta_{p}(u_{m})>0$, comme il fallait le prouver. Le lemme est ainsi démontré.
\end{proof} 

\begin{lemme}[Myerson \cite{myer}]\label{cat}
Pour tout entier strictement positif $n$, on a:
\begin{equation}\label{eqttttt1}
\ppcm(a_{1},a_{2},\dots ,a_{n})=\prod_{m=1}^{n}u_{m}.
\end{equation}
\end{lemme}
\begin{proof}
Il s'agit de montrer que pour tout nombre premier $p$, on a:
\[\vartheta_{p}\left(\ppcm(a_{1},a_{2},\dots ,a_{n})\right)=\vartheta_{p}\left(\prod_{m=1}^{n}u_{m}\right).\]
Soit donc $p$ un nombre premier. Sans perte de généralité, nous supposons que $p$ divise au moins l'un des nombres $a_{m}$, avec $1\leq m\leq n$ (le résultat est évident dans le cas contraire). D'une part, on a:
\[\vartheta_{p}\left(\ppcm(a_{1},a_{2},\dots ,a_{n})\right)=\max_{1\leq i\leq n}\vartheta_{p}(a_{i})=\vartheta_{p}(a_{r}),\]
où $r$ est le plus grand champion pour $p$ qui est inférieur ou égal à $n$. D'autre part, d'après les lemmes \ref{pourc11} et \ref{lc} et la remarque \ref{rmq1s}, on a:
\[\vartheta_{p}\left(\prod_{m=1}^{n}u_{m}\right)=\sum_{m=1}^{n}\vartheta_{p}(u_{m})={\sum}^{'}\vartheta_{p}(u_{m})=\sum_{m\mid r}\vartheta_{p}(u_{m})=\vartheta_{p}\left(\prod_{m\mid r}u_{m}\right)=\vartheta_{p}(a_{r}),\]
où ${\sum}^{'}$ est la somme étendue sur tous les champions $m\leq n$ pour $p$. En comparant les deux résultats, on en déduit que pour tout nombre premier $p$, on a:
\[\vartheta_{p}\left(\ppcm(a_{1},a_{2},\dots ,a_{n})\right)=\vartheta_{p}\left(\prod_{m=1}^{n}u_{m}\right).\]
Ce qui conclut au résultat du lemme et achève cette démonstration. 
\end{proof}

\noindent\begin{proof}[Démonstration du théorème \ref{ts1}]
Pour tout $m \in \mathbb{N^*}$, on écrit (en vertu de \eqref{eqqq}):
\[\prod_{1\leq j\leq m}a_{j}=\prod_{1\leq j\leq m}\prod_{d\mid j}u_{d}=\prod_{1\leq j\leq m}u_{j}^{\lfloor m/j \rfloor}.\]
Le membre de droite de \eqref{ds} s'écrit donc:
\[\prod_{1\leq j\leq n}u_{j}^{\lfloor n/j \rfloor - \lfloor n/jb_{1} \rfloor - \lfloor n/jb_{2}\rfloor - \lfloor n/jb_{3} \rfloor -\dots}.\]
Il suffit donc de montrer (en vertu du lemme \ref{cat}) que pour tout entier strictement positif $k$, on a:
\[1 \leq k -\lfloor k/b_{1} \rfloor -\lfloor k/b_{2}\rfloor -\lfloor k/b_{3}\rfloor - \dots, \]
(remarquer que $\lfloor \lfloor x\rfloor /n \rfloor=\lfloor x/n\rfloor$, $\forall x \in \mathbb{R},\forall n\in \mathbb{N^*}$). Étant donné $k\in \mathbb{N^*}$, il existe $r\in \mathbb{N^*}$ tel que $k< b_{r+1}$. On a par suite:
\[\lfloor k/b_{1} \rfloor +\lfloor k/b_{2} \rfloor + \lfloor  k/b_{3}\rfloor + \dots \leq k/b_{1}+k/b_{2}+\dots +k/b_{r}=k\sum_{l=1}^{r}\frac{1}{b_{l}} < k.\]
Ceci entraîne que le nombre $k -\lfloor k/b_{1} \rfloor -\lfloor k/b_{2}\rfloor -\lfloor k/b_{3}\rfloor - \dots$ est un entier strictement positif, comme il fallait le prouver. Le théorème est démontré. 
\end{proof}

\begin{rmq}
Une étude plus profonde des suites à forte divisibilité est donnée dans le chapitre \ref{ch3}.
\end{rmq}

\section{Majoration effective du nombre $\ppcm(1 , 2 , \dots , n)$}\label{prec}
D'après la section  \textsection\ref{for}, on comprend bien que l'étude des propriétés arithmétiques de certains coefficients multinomiaux permet d'obtenir des majorations effectives pour le nombre $\ppcm(1,2,\dots,n)$. Dans cette section, nous présentons l'estimation de Hanson \cite{han} qui utilise la relation de l'exemple \ref{exemple1} et aboutit au théorème suivant:

\begin{thm}[Hanson \cite{han}]\label{thh}
Pour tout entier $n\geq 1$, on a:
\[\ppcm(1,2,\dots,n)\leq 3^{n}.\]
\end{thm}

\noindent D'après l'exemple \ref{exemple1}, on a:
\begin{equation}\label{eqH1}
\ppcm (1,2,\dots,n)\leq C(n):=\frac{n!}{\lfloor n/2 \rfloor ! \lfloor n/3 \rfloor ! \lfloor n/7\rfloor !\lfloor n/43\rfloor !\cdots},
\end{equation}
où $2,3,7,43,\dots$ est la suite de Sylvester (déjà définie dans la partie \textsection\ref{for}). Notons que le choix de cette suite étant heuristique; il s'appuie sur le fait que la somme des inverses de ses termes converge vers $1$ plus vite que toute autre série de la forme $\sum_{\ell\geq 1}\frac{1}{b_{\ell}}$, où $(b_{\ell})_{\ell}$ est une suite d'entiers satisfaisant aux conditions du théorème \ref{ts1}. Cela permet d'exploiter le théorème \ref{ts1} d'une façon optimale. Dans tout ce qui suit, on désigne par $(b_{\ell})_{\ell}$ la suite de Sylvester.

\begin{lemme}[Hanson \cite{han}]\label{avde}
Soit $n\in\mathbb{N^*}$ et désignons par $k$ l'unique entier strictement positif vérifiant $b_{k}\leq n<b_{k+1}$. Alors, on a:
\[C(n)\leq\frac{n^n}{{\lfloor n/b_{1} \rfloor}^{\lfloor n/b_{1} \rfloor}{\lfloor n/b_{2} \rfloor}^{\lfloor n/b_{2} \rfloor}\cdots {\lfloor n/b_{k} \rfloor}^{\lfloor n/b_{k} \rfloor}}.\]
\end{lemme} 
\begin{proof}
Si un entier strictement positif $t$ possède une partition sous la forme $t=t_{1}+t_{2}+\cdots + t_{k}$, avec $t_{\ell}\in\mathbb{N^*}$ $(\forall \ell\in\{1,\dots,k\})$, alors on a (en vertu de la formule multinomiale):
\[{t}^{t}=(t_{1}+t_{2}+\cdots + t_{k})^t=\sum_{i_1+i_2+\dots +i_k=t}\frac{t!}{i_1!i_2!\cdots i_k!}{t_1}^{i_1}{t_2}^{i_2}\cdots {t_k}^{i_k}\geq\frac{t!}{t_{1}!t_{2}!\cdots t_{k}!} {t_{1}}^{t_{1}} {t_{2}}^{t_{2}} \cdots {t_{k}}^{t_{k}}.\]
Par suite, en appliquant cette dernière inégalité pour $t:=\sum_{\ell=1}^{k}\lfloor n/b_{\ell} \rfloor\leq n\sum_{\ell\geq 1}1/b_{\ell}=n$, $t_{\ell}:=\lfloor n/b_{\ell} \rfloor$ $(\forall \ell\in\{1,\dots,k\})$ et en tenant compte du fait que $\lfloor n/b_{\ell}\rfloor=0$ $(\forall \ell\geq k+1)$ (car: $n<b_{k+1}<b_{k+2}<\dots$), on aboutit à: 
\[C(n)=\frac{n(n-1)\cdots (t+1)t!}{{\lfloor n/b_{1} \rfloor}!{\lfloor n/b_{2} \rfloor}!\cdots {\lfloor n/b_{k} \rfloor}!}\leq \frac{n^{n-t}t^{t}}{{\lfloor n/b_{1} \rfloor}^{\lfloor n/b_{1} \rfloor}{\lfloor n/b_{2} \rfloor}^{\lfloor n/b_{2} \rfloor}\cdots {\lfloor n/b_{k} \rfloor}^{\lfloor n/b_{k} \rfloor}}.\]
Ce qui conclut (via l'inégalité: $n^{n-t}t^{t}\leq n^{n-t}n^{t}=n^n$) à l'estimation requise par le lemme et achève cette démonstration.  
\end{proof}
 
\begin{lemme}[Hanson \cite{han}]\label{H1}
Soient $\ell$ et $n$ deux entiers strictement positifs tels que $b_{\ell}\leq n$. Alors, on a:
\begin{equation}\label{eqH2}
\frac{(n/b_{\ell})^{n/b_{\ell}}}{{\lfloor n/b_{\ell}\rfloor}^{\lfloor n/b_{\ell}\rfloor}}\leq \left(\frac{en}{b_{\ell}}\right)^{(b_{\ell}-1)/b_{\ell}}.
\end{equation}
\end{lemme}
\begin{proof}
Pour $n=b_{\ell}$, le résultat est trivial. Supposons pour la suite que $n>b_{\ell}$ et montrons préalablement que l'on a:
\begin{equation}\label{ron}
\left((n-b_{\ell}+1)/b_{\ell}\right)^{(n-b_{\ell}+1)/b_{\ell}}\leq {\lfloor n/b_{\ell}\rfloor}^{\lfloor n/b_{\ell}\rfloor}.
\end{equation}
Puisque $\lfloor n/b_{\ell}\rfloor>n/b_{\ell}-1$, on a: $b_{\ell}\lfloor n/b_{\ell}\rfloor\geq n-b_{\ell}+1$, ce qui entraîne que $\lfloor n/b_{\ell}\rfloor\geq (n-b_{\ell}+1)/b_{\ell}$. Maintenant, si $(n-b_{\ell}+1)/b_{\ell}\geq 1$, l'inégalité \eqref{ron} découle de la croissance de la fonction $x\mapsto x^x$ sur l'intervalle $\left[1,+\infty\right[$. Par contre, si $(n-b_{\ell}+1)/b_{\ell}< 1$, alors l'inégalité \eqref{ron} est évidente (puisque $n>b_{\ell}\Rightarrow {\lfloor n/b_{\ell}\rfloor}^{\lfloor n/b_{\ell}\rfloor}\geq 1$), comme il fallait le prouver. Il s'ensuit (en vertu de \eqref{ron} et de l'inégalité $\left(1+\frac{1}{x}\right)^x\leq e$ $(\forall x>0)$) que: 
\begin{align*}
\frac{(n/b_{\ell})^{n/b_{\ell}}}{{\lfloor n/b_{\ell}\rfloor}^{\lfloor n/b_{\ell}\rfloor}} &\leq \frac{(n/b_{\ell})^{n/b_{\ell}}}{((n-b_{\ell}+1)/b_{\ell})^{(n-b_{\ell}+1)/b_{\ell}}}\\ 
& = \left(1+\frac{1}{(n-b_{\ell}+1)/(b_{\ell}-1)}\right)^{((n-b_{\ell}+1)/(b_{\ell}-1))\times ((b_{\ell}-1)/b_{\ell})}\left(\frac{n}{b_{\ell}}\right)^{(b_{\ell}-1)/b_{\ell}} \\ & \leq \left(\frac{en}{b_{\ell}}\right)^{(b_{\ell}-1)/b_{\ell}}.
\end{align*}
Ce qui termine cette démonstration.
\end{proof}

\begin{lemme}[Hanson \cite{han}]\label{dernierL}
Soit $n$ un entier $\geq 7$ et $k$ l'unique entier strictement positif vérifiant $b_{k}\leq n <b_{k+1}$. Alors, on a:
\[k\leq\log_{2}\left(\log_{2}n\right)+2,\]
où $\log_{2}$ désigne le logarithme de base $2$ (i.e., $\log_{2}x=\frac{\log x}{\log 2}$ $(\forall x>0)$).
\end{lemme}
\begin{proof}
Puisque $b_{3}=7\leq n$, alors on a nécessairement $k\geq 3$. Par ailleurs, en se servant de \eqref{sylv1}, on montre facilement par récurrence que pour tout entier $\ell\geq 3$, on a: $b_{\ell}\geq 2^{2^{\ell-2}}+1$. Il s'ensuit de cette dernière inégalité (appliquée à $\ell=k$) que:
\[k \leq \log_{2}\log_{2}(b_{k}-1)+2 \leq \log_{2}\log_{2}n+2.\]
Le lemme est démontré.
\end{proof}

\begin{proof}[Démonstration du théorème \ref{thh}]
En utilisant un logiciel de calcul (Maple ou Mathematica par exemple), on vérifie que le résultat du théorème est valable pour tout entier $n < 4500$. Supposons pour la suite que $n\geq 4500$ et désignons par $k$ l'unique entier strictement positif vérifiant $b_{k} \leq n < b_{k+1}$. D'après les lemmes \ref{avde} et \ref{H1}, on a:
\begin{equation}\label{finn}
C(n)\leq \frac{n^n}{{\lfloor n/b_{1} \rfloor}^{\lfloor n/b_{1} \rfloor}\cdots {\lfloor n/b_{k} \rfloor}^{\lfloor n/b_{k} \rfloor}}\leq\frac{n^{n}(en/b_{1})^{(b_{1}-1)/b_{1}}\cdots (en/b_{k})^{(b_{k}-1)/b_{k}}}{(n/b_{1})^{n/b_{1}}\cdots (n/b_{k})^{n/b_{k}}}.
\end{equation}
Considérons la suite $\left(u_{\ell}\right)_{\ell\geq 1}$ donnée par: $u_{\ell}:=b_{1}^{1/b_{1}}b_{2}^{1/b_{2}}\cdots b_{\ell}^{1/b_{\ell}}$ $(\forall\ell\geq 1)$ et montrons préalablement que:
\begin{equation}\label{665}
u_{\ell}\leq 2,952~~~~(\forall \ell\in\mathbb{N^*}).
\end{equation}
Il est immédiat que $\left(u_{\ell}\right)_{\ell\geq 1}$ est strictement croissante. D'autre part, pour $\ell\in\mathbb{N^*}$, on a (en vertu de \eqref{sylv1}): ${b_{\ell}}^2=b_{\ell+1}+b_{\ell}-1\geq b_{\ell+1}+1>b_{\ell+1}$ et $b_{\ell+1}={b_{\ell}}^2-b_{\ell}+1>{b_{\ell}}^2-2b_{\ell}+1=\left(b_{\ell}-1\right)^2$, soit:
\[b_{\ell}^{2} > b_{\ell+1} > (b_{\ell}-1)^{2}~~~~(\forall \ell\in\mathbb{N^*}).\]
Cette double-inégalité entraîne que pour tout entier $\ell\geq 3$, on a:
\[\frac{\log\left(b_{\ell+1}^{1/b_{\ell+1}}\right)}{\log\left(b_{\ell}^{1/b_{\ell}}\right)}=\frac{b_{\ell}\log\left(b_{\ell+1}\right)}{b_{\ell+1}\log\left(b_{\ell}\right)} < \frac{2b_{\ell}}{b_{\ell+1}}<\frac{2b_{\ell}}{(b_{\ell}-1)^{2}}<\frac{1}{2}.\]
En combinant cela avec l'inégalité $\log\left({b_{6}}^{1/b_{6}}\right)\leq 5\cdot 10^{-6}$, on obtient:
\begin{align*}
\sum_{\ell=1}^{+\infty}\log\left(b_{\ell}^{1/b_{\ell}}\right)&=\sum_{\ell=1}^{5}\log\left(b_{\ell}^{1/b_{\ell}}\right)+\sum_{\ell=6}^{+\infty}\log\left(b_{\ell}^{1/b_{\ell}}\right)\\&\leq 1,0824 + \log\left(b_{6}^{1/b_{6}}\right)\sum_{\ell\geq 0}\left(\frac{1}{2}\right)^{\ell}\\&\leq 1,0824 +10^{-5}.
\end{align*}
D'où: $u_{m}\leq\lim\limits_{\ell\rightarrow+\infty}b_{1}^{1/b_{1}}b_{2}^{1/b_{2}}\cdots b_{\ell}^{1/b_{\ell}}\leq e^{1,0824 +10^{-5}}\leq 2,952$ $(\forall m\in\mathbb{N^*})$, comme nous l'avons prétendu. Par ailleurs, on a (en vertu de \eqref{sylv2}): 
\begin{align}
\frac{b_{1}-1}{b_{1}}+\frac{b_{2}-1}{b_{2}}+\dots +\frac{b_{k}-1}{b_{k}}&=k-1+ \frac{1}{b_{1}b_{2}\cdots b_{k}}=k-1+\frac{1}{b_{k+1}-1}\label{666}
\intertext{et}
\frac{1}{b_{1}}+\frac{1}{b_{2}}+\dots +\frac{1}{b_{k}}&=1-\frac{1}{b_{1}b_{2}\cdots b_{k}}=1-\frac{1}{b_{k+1}-1}.\label{667}
\end{align}
En combinant les estimations \eqref{finn}, \eqref{665}, \eqref{666}, \eqref{667} et le lemme \ref{dernierL}, on aboutit à:
\begin{align*}
C(n)&\leq\frac{n^n}{\left(n^{1/b_{1}+1/b_{2}+\dots +1/b_{k}}\right)^n}\cdot \frac{(en)^{(b_{1}-1)/b_{1}+(b_{2}-1)/b_{2}+\dots +(b_{k}-1)/b_{k}}\left({b_{1}}^{1/b_{1}}{b_{2}}^{1/b_{2}}\cdots {b_{k}}^{1/b_{k}}\right)^n}{b_{1}^{(b_{1}-1)/b_{1}}b_{2}^{(b_{2}-1)/b_{2}}\cdots b_{k}^{(b_{k}-1)/b_{k}}} \\ &\leq n^{n/(b_{k+1}-1)}(en)^{k-1+1/(b_{k+1}-1)}{(2,952)}^{n}\\ &\leq e^{k}n^{k+1}{(2,952)}^{n}\leq (en)^{\log_{2}(\log_{2}n)+3}{(2,952)}^{n}.
\end{align*}
Puisque la fonction $x\mapsto (ex)^{\frac{\log_{2}(\log_{2}x)+3}{x}}$ est décroissante sur l'intervalle $\left[4500,+\infty\right[$, il s'ensuit que:
\[(en)^{\frac{\log_{2}(\log_{2}n)+3}{n}} \leq (4500e)^{\frac{\log_{2}(\log_{2}4500)+3}{4500}}\leq 1.014.\]
D'où: $C(n)\leq (1.014)^{n}(2.952)^{n}\leq 3^n$. Ce qui complète cette démonstration.
\end{proof}

\noindent Le corollaire suivant est une conséquence immédiate du théorème \ref{thh} et de la définition de la fonction $\psi$ de Chebyshev.

\begin{coll}[Hanson \cite{han}]\label{chebfin}
Pour tout réel $x\geq 1$, on a:
\[\psi(x)\leq x\log 3.\]
\end{coll}

\begin{coll}[Hanson \cite{han}]
Pour tout réel $x\geq 2$, on a:
\[\pi(x)\leq 1,25506\frac{x}{\log x},\]
où $\pi$ désigne la fonction de comptage des nombres premiers. 
\end{coll}
\begin{proof}
En se servant d'un logiciel de calcul (Maple ou Mathematica par exemple), on vérifie que le résultat du théorème est valable pour tout $x < 350$. Supposons pour la suite que $x\geq 350$. En utilisant successivement la formule \eqref{abel1}, l'inégalité triviale $\theta(t)\leq \psi(t)$ $(\forall t\geq 1)$, le corollaire \ref{chebfin} et l'estimation \eqref{psuite}, on obtient:
\begin{align*}
\pi(x)&=\frac{\theta(x)}{\log x}+\int_{2}^{x}\frac{\theta(t)}{t\log^2 t}\mathrm{d}t\\&\leq \frac{\psi(x)}{\log x}+\int_{2}^{x}\frac{\psi(t)}{t\log^2 t}\mathrm{d}t\\&\leq \frac{x\log 3}{\log x}+\log 3\int_{2}^{x}\frac{\mathrm{d}t}{\log^2 t}\\&\leq \frac{x\log 3}{\log x}\left(\frac{1}{\sqrt{x}{\log}^{2}x}+\frac{4}{{\log}^{2}x}\right)\\&\leq 1,25506\frac{x}{\log x}.
\end{align*}
Ce qui complète cette démonstration.
\end{proof}

\section{Minoration effective du nombre $\ppcm(1 , 2 , \dots , n)$}
Nous présentons ici la démonstration de Nair \cite{nair} prouvant l'estimation:
\[\ppcm(1,2,\dots,n)\geq 2^n~~~~(\forall n\geq 7).\]
En fait, il existe dans la littérature mathématique des minorations bien meilleures, néan\-moins celles-ci ne s'obtiennent pas de façon élémentaire. Notre préférence à cette méthode réside dans sa simplicité (on exploite l'intégrale $\int_{0}^{1}x^{n}(1-x)^{n}\mathrm{d}x$ $(n\in\mathbb{N})$). 

\begin{thm}[Nair \cite{nair}]\label{thnai}
Pour tout entier $n\geq 7$, on a:
\begin{equation}\label{minnair}
d_{n}:=\ppcm(1,2,\dots,n)\geq 2^{n}.
\end{equation}
\end{thm}  
\begin{proof}
Soient $k\in\mathbb{N^*}$, $\ell\in\{1,2,\dots,k\}$ et considérons l'intégrale suivante:
\[I(k,\ell):=\int_{0}^{1}x^{\ell-1}(1-x)^{k-\ell}\mathrm{d}x.\]
On a en vertu de la formule du binôme: 
\[I(k,\ell)=\sum_{r=0}^{k-\ell}(-1)^{r}\binom{k-\ell}{r}\frac{1}{\ell+r}.\]
Ce qui entraîne que $I(k,\ell)d_{k}\in \mathbb{N^{*}}$ $(\forall k,\ell\in\mathbb{N^*};~1\leq \ell\leq k)$. D'autre part, en intégrant successivement par parties, on obtient:
\[I(k,\ell)=\frac{(k-\ell)}{\ell}\int_{0}^{1}x^{\ell}(1-x)^{k-\ell-1}\mathrm{d}x=\dots=\frac{(k-\ell)!}{\ell(\ell+1)\cdots (k-1)}\int_{0}^{1}x^{k-1}\mathrm{d}x=\frac{1}{\ell\binom{k}{\ell}}.\]
Ce qui implique que:
\begin{equation}\label{601}
\ell\binom{k}{\ell}~\text{divise}~d_{k}~~~~(\forall k,\ell\in\mathbb{N^*};~1\leq \ell\leq k).
\end{equation}
Fixons $m\in\mathbb{N^*}$. D'après la relation \eqref{601}, on a:
\[m \binom{2m}{m}\mid d_{2m} \mid d_{2m+1}~~\text{et}~~(2m+1)\binom{2m}{m}=(m+1)\binom{2m+1}{m+1}\mid   d_{2m+1}.\]
Comme $\pgcd(m,2m+1)=1$, il s'ensuit que:
\begin{equation}\label{600}
\ppcm\left(m\binom{2m}{m},(2m+1)\binom{2m}{m}\right)= m(2m+1)\binom{2m}{m}~~\text{divise}~~d_{2m+1}.
\end{equation}
Par ailleurs, on a en vertu de la formule du binôme: 
\[4^m=(1+1)^{2m}=\sum_{\ell=0}^{2m}\binom{2m}{\ell}\leq (2m+1)\max\left\lbrace\binom{2m}{\ell};~0\leq \ell\leq 2m\right\rbrace=(2m+1)\binom{2m}{m}.\]
En combinant cette dernière inégalité avec \eqref{600}, on en déduit que:
\begin{align*}
d_{2m+1}&\geq m(2m+1)\binom{2m}{m} \geq m4^{m}\geq 2^{2m+1}~~~~(\forall m\geq 2)\intertext{et}
d_{2m+2}&\geq d_{2m+1}\geq m4^m\geq 2^{2m+2}~~~~(\forall m\geq 4).
\end{align*}
L'estimation requise par le théorème découle de ces deux dernière inégalités et de l'inégalité triviale $d_8\geq 2^8$. Le théorème est ainsi démontré.
\end{proof}

\begin{coll}[Nair \cite{nair}]
Pour tout entier $n\geq 7$, on a:
\[\pi(n)\geq \frac{n\log 2}{\log n}.\]
\end{coll}
\begin{proof}
Le résultat découle immédiatement de \eqref{ustt} et du fait que l'on a (en vertu du théorème \ref{thnai}): $\psi(n)\geq n\log 2$ $(\forall n\geq 7)$.
\end{proof}

\noindent Un autre résultat très important dû à Nair est le suivant.
\begin{thm}[Nair \cite{nair}]\label{thnair2}
Pour tout entier strictement positif $n$, on a:
\[d_{n}:=\ppcm(1,2,\dots,n)=\ppcm\left(\binom{n}{1},2\binom{n}{2},\dots,n\binom{n}{n}\right).\]
\end{thm}
\begin{proof}
D'après la relation \eqref{601}, on a pour tout $m\in\{1,2,\dots,n\}$: $m \binom{n}{m}$ divise $d_{n}$. Ce qui entraîne que:
\[\ppcm\left(\binom{n}{1},2\binom{n}{2},\dots,m\binom{n}{m}\right)~\text{divise}~d_{n}.\]
D'autre part, puisqu'on a visiblement $m$ divise $m\binom{n}{m}$ pour tout $m\in\{1,2,\dots,n\}$, il s'ensuit que:
\[d_{n}~\text{divise}~\ppcm\left(\binom{n}{1},2\binom{n}{2},\dots,n\binom{n}{n}\right).\]
Ce qui conclut à l'identité requise par le théorème et complète cette démonstration.  
\end{proof}

\begin{coll}[Nair \cite{nair}]
Pour tout entier strictement positif $n$, on a:
\begin{equation}\label{binfpp}
d_n:=\ppcm\left(1,2,\dots,n\right)\leq 4^n.
\end{equation}
\end{coll}
\begin{proof}
On procède par récurrence sur $n$. Pour $n=1$, l'inégalité \eqref{binfpp} est triviale. Fixons $m\in\mathbb{N^*}$ et supposons que \eqref{binfpp} est vraie pour tout entier strictement positif $n<2m$. Nous montrons que \eqref{binfpp} reste vraie pour $n=2m$ et pour $n=2m+1$, ce qui conclura au résultat requis. Premièrement, on vérifie facilement que l'on a l'identité:
\begin{equation}\label{idnair}
k\binom{2m}{k}\binom{2m-k}{m-k}=k\binom{m}{k}\binom{2m}{m}~~~~(\forall k\in\{1,2,\dots,m\}).
\end{equation}
Cette dernière entraîne que $k\binom{2m}{k}$ divise $k\binom{m}{k}\binom{2m}{m}$ $(\forall k\in\{1,2,\dots,m\})$. D'autre part, le membre de droite de \eqref{idnair} est (en vertu du théorème \ref{thnair2}) un diviseur de $d_m\binom{2m}{m}$. Par conséquent, on a pour tout $k\in\{1,2,\dots,m\}$:
\begin{equation}\label{idnair1}
k\binom{2m}{k}~~\text{divise}~~d_m\binom{2m}{m}.
\end{equation}
Puisque pour tout $k\in\{m+1,\dots,2m\}$, on a: $k\binom{2m}{k}=(2m-k+1)\binom{2m}{2m-k+1}$ et $1\leq 2m-k+1\leq m$, il s'ensuit que la relation \eqref{idnair1} est valable pour tout $k\in\{1,2,\dots,2m\}$. En combinant cela avec le théorème \ref{thnair2}, on obtient que: $d_{2m}$ divise $d_m\binom{2m}{m}$. En procédant comme ci-dessus et en utilisant l'identité suivante:
\[k\binom{2m+1}{k}\binom{2m+1-k}{m+1-k}=k\binom{m+1}{k}\binom{2m+1}{m+1}~~~~(\forall k\in\{1,2,\dots,m+1\}),\]
au lieu de \eqref{idnair}, on obtient que: $d_{2m+1}$ divise $d_{m+1}\binom{2m+1}{m+1}$. On a par conséquent:
\begin{align*}
d_{2m}&\leq d_m\binom{2m}{m}\leq 4^m(1+1)^{2m}=4^{2m}
\intertext{et}
d_{2m+1}&\leq d_{m+1}\binom{2m+1}{m+1}\leq 4^{m+1}\frac{(1+1)^{2m+1}}{2}=4^{2m+1}.
\end{align*}
Ce qui achève cette récurrence et complète la démonstration du corollaire.
\end{proof}

\section{Minorations effectives du $\ppcm$ d'une suite arith\-métique}

Dans la continuation, Farhi \cite{far} a présenté une méthode permettant d'obtenir des minorations effectives et non triviales du $\ppcm$ d'une suite arithmétique et d'une certaine classe de suites quadratiques. Cette partie est consacrée à la démonstration de quelques résultats portant sur le $\ppcm$ d'une suite arithmétique. On parlera aussi d'une certaine estimation conjecturée dans \cite{far} et démontrée par Hong \cite{hong1}. 

\begin{thm}[Farhi \cite{far}]\label{Flcm}
Soit $(u_{k})_{k\in \mathbb{N}}$ une suite arithmétique d'entiers, de raison $r$ et de premier terme $u_{0}$ strictement positifs et premiers entre eux. Alors, pour tout entier naturel $n$, on a:
\begin{equation}\label{ff1}
\ppcm(u_{0},u_{1},\dots,u_{n}) \geq u_{0}(1+r)^{n-1}.
\end{equation}
Si de plus $n$ est multiple de $(r+1)$, on a:
\begin{equation}\label{ff2}
\ppcm(u_{0},u_{1},\dots,u_{n}) \geq u_{0}(1+r)^{n}.
\end{equation}
\end{thm}

\noindent La démonstration de ce théorème est partagée en plusieurs lemmes.

\begin{lemme}[Farhi \cite{far}]\label{lF1}
Soit $(u_{k})_{k\in \mathbb{N}}$ une suite d'entiers non nuls, strictement croissante. Alors, pour tout entier naturel $n$, le produit $u_{0}u_{1}\cdots u_{n}$ divise le nombre:
\[M_n:=\ppcm\left(u_{0},u_{1},\dots,u_{n}\right)\cdot\ppcm\left\{\prod_{\begin{subarray}{c}0\leq i\leq n\\i\neq j\end{subarray}}(u_{i}-u_{j});~j=0,1,\dots,n\right\}.\]
\end{lemme}
\begin{proof}
En substituant $x=0$ dans la décomposition en éléments simples de la fraction rationnelle $1/(x+u_{0})(x+u_{1})\dots (x+u_{n})$, on obtient:
\[\sum_{j=0}^{n}\frac{1}{u_{j}}\cdot\frac{1}{\prod_{\begin{subarray}{c}0\leq i\leq n\\i\neq j\end{subarray}}(u_{i}-u_{j})}=\frac{1}{u_{0}u_{1}\cdots u_{n}}.\] 
Le résultat requis en découle en multipliant les deux membres de cette dernière identité par $M_n$. 
\end{proof}

\begin{lemme}[Farhi \cite{far}]\label{xyxy}
Soit $(u_{k})_{k\in \mathbb{N}}$ une progression arithmétique strictement croissante d'entiers non nuls. Alors, pour tout entier naturel $n$, on a:
\[\ppcm\left(u_{0},u_{1},\dots,u_{n}\right)~\text{est multiple de}~\frac{u_{0}u_{1}\cdots u_{n}}{n!\left(\pgcd\left(u_{0},u_{1}\right)\right)^{n}}.\]
\end{lemme}
\begin{proof}
On peut supposer que $\pgcd\left(u_{0},u_{1}\right)=1$ (quitte à remplacer la suite $(u_{k})_k$ par la suite $(v_{k})_k$ de terme général $v_{k}:=u_{k}/\pgcd\left(u_{0},u_{1}\right)$ $(\forall k\in\mathbb{N})$). D'après le lemme \ref{lF1}, le nombre $\ppcm\left(u_{0},u_{1},\dots,u_{n}\right)$ est multiple du nombre rationnel:
\[\frac{u_{0}u_{1}\cdots u_{n}}{\ppcm\left\{\prod_{\begin{subarray}{c} 0\leq i\leq n\\i\neq j\end{subarray}}(u_{i}-u_{j});~0\leq j\leq n\right\}}.\]
En désignant par $r\in\mathbb{N^*}$ la raison de la suite arithmétique $(u_{k})_{k\in \mathbb{N}}$, on a pour tout $j\in\{0,1,\dots,n\}$:
\[\prod_{\begin{subarray}{c} 0\leq i\leq n \\ i\neq j\end{subarray}}(u_{i}-u_{j})=(-1)^{j}j!(n-j)!r^{n}.\]
Par conséquent, on a:
\[\ppcm\left\{\prod_{\begin{subarray}{c} 0\leq i\leq n \\ i\neq j\end{subarray}}(u_{i}-u_{j});~0\leq j\leq n\right\}=r^{n}n!.\]
Ce qui entraîne que le nombre $\ppcm\left(u_{0},u_{1},\dots,u_{n}\right)$ est multiple du nombre rationnel $\frac{u_{0}u_{1}\cdots u_{n}}{r^{n}n!}$. Par ailleurs, puisque les entiers $u_{0}$ et $u_{1}$ sont premiers entre eux, alors $r$ est premier avec tous les termes de la suite $(u_{k})_{k \in\mathbb{N}}$, il est donc premier avec le produit $u_{0}u_{1}\cdots u_{n}$. Enfin, le lemme de Gauss permet de déduire que le nombre $\ppcm\left(u_{0},u_{1},\dots,u_{n}\right)$ est multiple du nombre rationnel $\frac{u_{0}u_{1}\cdots u_{n}}{n!}$. Ce qui complète cette démonstration.
\end{proof}

Dans les deux lemmes suivants, nous fixons une suite arithmétique $\left(u_k\right)_{k\in\mathbb{N}}$, ayant pour raison $r$ et pour premier terme $u_0$, avec $r , u_0 \in \mathbb{N}^*$. Pour tout $n \in \mathbb{N}$, nous désignons par $(v_{n,k})_{0\leq k \leq n}$ la suite finie de nombres rationnels définie par $v_{n,k}:=\frac{u_{k}u_{k+1}\cdots u_{n}}{(n-k)!}$ \linebreak $(\forall k\in\{0,1,\dots,n\})$.

\begin{lemme}[Farhi \cite{far}]\label{k0}
Pour tout entier naturel $n$, la suite $(v_{n,k})_{0\leq k \leq n}$ atteint son maximum en $k_{n}\in\{0,1,\dots,n\}$, défini par:
\begin{equation}\label{k0n}
k_{n}:=\max  \left\lbrace 0,\left\lfloor\frac{n-u_{0}}{r+1} \right\rfloor +1\right\rbrace.
\end{equation}
\end{lemme}
\begin{proof}
Pour tout $k\in \lbrace 0,1,\dots,n-1\rbrace$, on a:
\[\frac{v_{n,k+1}}{v_{n,k}}=\frac{n-k}{u_{k}}=\frac{n-k}{u_{0}+kr};\]
d'où:
\[v_{n,k+1}\geq v_{n,k} \Longleftrightarrow k\leq \frac{n-u_{0}}{r+1}\Longleftrightarrow k\leq \left\lfloor \frac{n-u_{0}}{r+1} \right\rfloor.\]
Deux cas peuvent se présenter. Si $n< u_{0}$, alors la suite $(v_{n,k})_{0\leq k \leq n}$ est décroissante et elle atteint donc son maximum en $k=0$. Si on a au contraire $n\geq u_{0}$, alors la suite $(v_{n,k})_{0\leq k \leq n}$ croit jusqu'au terme d'indice $\left\lfloor\frac{n-u_{0}}{r+1} \right\rfloor +1$ et elle décroit au-delà de ce dernier; elle atteint donc sont maximum en $k=\left\lfloor\frac{n-u_{0}}{r+1} \right\rfloor +1$. Ce qui conclut au résultat du lemme et achève cette démonstration.   
\end{proof}

\begin{lemme}[Farhi \cite{far}]\label{beta}
Pour tout $k\in \lbrace 0,1,\dots,n\rbrace$, on a:
\begin{equation}\label{eqbeta}
v_{n,k}=\frac{r^{n-k+1}}{\int_{0}^{1}x^{k+\frac{u_{0}}{r}-1}(1-x)^{n-k}\mathrm{d}x}.
\end{equation}
\end{lemme}
\begin{proof}
D'après les propriétés usuelles des fonctions $\Gamma$ et $\beta$ d'Euler, on a:
\begin{align*}\label{vbn}
v_{n,k}=\frac{u_{k}\cdots u_{n}}{(n-k)!}&=\frac{u_{k}(u_{k}+r)\cdots (u_{k}+(n-k)r)}{(n-k)!} \\ &=r^{n-k+1}\cdot \frac{\frac{u_{k}}{r}(\frac{u_{k}}{r}+1)\cdots (\frac{u_{k}}{r}+n-k)}{(n-k)!} \\ &=r^{n-k+1}\frac{\Gamma(\frac{u_{k}}{r}+n-k+1)}{\Gamma(\frac{u_{k}}{r})\cdot \Gamma(n-k+1)}\\ &=\frac{r^{n-k+1}}{\beta(\frac{u_{k}}{r},n-k+1)},
\end{align*}
L'identité \eqref{eqbeta} découle de la formule intégrale de la fonction $\beta$. Ce qui complète cette démonstration.
\end{proof}

\begin{proof}[Démonstration du théorème \ref{Flcm}]
Fixons $n\in\mathbb{N}$ et montrons l'inégalité \eqref{ff1} du théorème. Étant donné $k\in \lbrace 0,1,\dots,n\rbrace$, le nombre $\ppcm\left(u_{0},u_{1},\dots,u_{n}\right)$ est visiblement un multiple du nombre $\ppcm\left(u_{k},u_{k+1},\dots,u_{n}\right)$, qui est lui même (d'après le lemme \ref{xyxy}) multiple de $v_{n,k}:=\frac{u_{k}u_{k+1}\cdots u_{n}}{(n-k)!}$. En combinant cela avec le lemme \ref{k0}, on obtient que:
\begin{equation}\label{cl}
\ppcm\left(u_{0},u_{1},\dots,u_{n}\right)\geq \max \left\lbrace v_{n,k};~0\leq k\leq n\right\rbrace=v_{n,k_n}.
\end{equation}
Il ne reste plus qu'à trouver une bonne minoration pour le nombre $v_{n,k_n}$. Pour surmonter certaines difficultés de calcul, nous allons plutôt minorer le nombre $v_{n,k^*}$, où $k^{*}:=\left\lfloor \frac{n-1}{r+1}+1\right\rfloor\in\{0,1,\dots,n\}$ et on conclura au résultat requis via l'inégalité $v_{n,k_n}\geq v_{n,k^*}$. D'après le lemme \ref{beta}, on a:
\begin{equation}\label{rama}
v_{n,k^*}=\frac{r^{n-k^{*}+1}}{\int_{0}^{1}x^{k^{*}+\frac{u_{0}}{r}-1}(1-x)^{n-k^{*}}\mathrm{d}x}.
\end{equation}
Par ailleurs, comme $\frac{n-1}{r+1}<k^{*}\leq \frac{n+r}{r+1}$, il s'ensuit que:
\begin{equation}\label{rrr}
r^{n-k^{*}+1}\geq r^{\frac{(n-1)r}{r+1}+1}.
\end{equation}
De même, on a pour tout $x\in \left[0,1\right]$:
\[x^{k^{*}+\frac{u_{0}}{r}-1}(1-x)^{n-k^{*}}\leq x^{\frac{n-1}{r+1}+\frac{u_{0}}{r}-1}(1-x)^{\frac{(n-1)r}{r+1}}.\]
En intégrant les deux membres de cette dernière inégalité sur l'intervalle $\left[0,1\right]$, on obtient:
\begin{equation}\label{uu}
\int_{0}^{1}x^{k^{*}+\frac{u_{0}}{r}-1}(1-x)^{n-k^{*}}\mathrm{d}x \leq \int_{0}^{1}\left\lbrace x(1-x)^{r} \right\rbrace^{\frac{n-1}{r+1}}x^{\frac{u_{0}}{r}-1}\mathrm{d}x. 
\end{equation}
L'estimation \eqref{uu} avec l'inégalité $x(1-x)^{r}\leq \frac{r^{r}}{(r+1)^{r+1}}$ $(\forall x\in \left[0,1\right])$, entraînent que:
\begin{equation}\label{ccc}
\int_{0}^{1}x^{k^{*}+\frac{u_{0}}{r}-1}(1-x)^{n-k^{*}}\mathrm{d}x \leq \frac{r^{\frac{(n-1)r}{r+1}}}{(r+1)^{n-1}}\cdot\frac{r}{u_{0}}.
\end{equation} 
En combinant les estimations \eqref{cl}, \eqref{rama}, \eqref{rrr} et \eqref{ccc}, on aboutit à:
\[\ppcm\left(u_{0},u_{1},\dots,u_{n}\right)\geq v_{n,k_n}\geq v_{n,k^*}\geq u_{0}(r+1)^{n-1},\] 
comme il fallait le prouver. La minoration \eqref{ff2} du théorème se démontre de la même façon en travaillons cette fois-ci avec l'entier naturel $k^{**}:=\frac{n}{r+1}$ au lieu de $k^{*}$. Ce qui complète cette démonstration. 
\end{proof}

Farhi \cite{far} a conjecturé que l'estimation \eqref{ff2} est vraie pour tout $n\in \mathbb{N}$. Une démonstration de cette conjecture, établie par Hong \cite{hong1}, est fournie dans le théorème suivant:

\begin{thm}[Hong \cite{hong1}]
Soit $(u_{k})_{k\in \mathbb{N}}$ une suite arithmétique d'entiers, de raison $r\in \mathbb{N^{*}}$ et dont le premier terme $u_{0}$ est strictement positif et premier avec $r$. Alors, pour tout entier naturel $n$, on a:
\begin{equation}
\ppcm\left(u_{0},u_{1},\dots,u_{n}\right)\geq u_{0}(1+r)^{n}.
\end{equation}  
\end{thm}
\begin{proof}
On procède par récurrence pour montrer que $v_{n,k_{n}}\geq u_{0}(1+r)^{n}$ $(\forall n\in \mathbb{N})$, où $k_{n}\in\{0,1,\dots,n\}$ est déjà défini dans le lemme \ref{k0}, ce qui conclura au résultat requis (via \eqref{cl}). En vertu du lemme \ref{k0}, on a pour tout entier naturel $n\leq u_{0}$:
\begin{equation}\label{HH1}
v_{n,k_{n}}\geq v_{n,0}=\frac{u_{0}u_{1}\cdots u_{n}}{n!}=u_{0}(u_{0}+r)\left(\frac{u_{0}}{2}+r\right)\cdots \left(\frac{u_{0}}{n}+r\right)\geq u_{0}(1+r)^{n}.
\end{equation}
Ce qui montre en particulier que la propriété requise est vraie pour $n\in\{0,1\}$ (car $u_{0}\geq 1$). Soit $n$ un entier strictement positif. Supposons que $v_{n,k_n}\geq u_{0}(1+r)^{n}$ et montrons que $v_{n+1,k_{n+1}}\geq u_{0}(1+r)^{n+1}$. Si $n\leq u_{0}$, le résultat découle de \eqref{HH1}. Supposons pour la suite que $n>u_{0}$. On vérifie facilement (à partir de la définition de $k_n$) que $k_{n}\leq k_{n+1}\leq k_{n}+1$; ce qui nous amène à distinguer les deux cas suivants:\\
\textbullet{} \underline{\textbf{1\textsuperscript{er}cas:}} (si $k_{n+1}=k_{n}$). Dans ce cas, on a:
\begin{equation}\label{eee2}
v_{n+1,k_{n+1}}=v_{n+1,k_{n}}=\frac{u_{k_{n}}u_{k_{n}+1}\cdots u_{n+1}}{(n+1-k_{n})!}=v_{n,k_{n}}\cdot \frac{u_{n+1}}{n+1-k_{n}}.
\end{equation}
Par ailleurs, on vérifie facilement (en se servant de l'égalité $k_n=k_{n+1}$) que:
\[\frac{n+1-u_{0}}{r+1}<k_{n}.\]
Cette dernière entraîne que:
\[u_{n+1}-(r+1)(n+1-k_{n})=u_{0}-(n+1)+k_{n}(r+1)>0.\]
D'où l'on tire:
\begin{equation}\label{eee3}
\frac{u_{n+1}}{n+1-k_{n}}>(r+1).
\end{equation}
En combinant \eqref{eee2}, \eqref{eee3} et l'hypothèse de récurrence, on aboutit à:
\[v_{n+1,k_{n+1}}\geq v_{n,k_{n}}(r+1)\geq u_0(1+r)^{n}(1+r)=u_{0}(1+r)^{n+1},\]
comme il fallait le prouver.\\
\textbullet{} \underline{\textbf{2\textsuperscript{nd}cas:}} (si $k_{n+1}=k_{n}+1$). Dans ce cas, on a:
\begin{equation}\label{eee5}
v_{n+1,k_{n+1}}=v_{n+1,k_{n}+1}=\frac{u_{k_{n}+1}\cdots u_{n+1}}{(n-k_{n})!}=v_{n,k_{n}}\cdot \frac{u_{n+1}}{u_{k_{n}}}.
\end{equation}
Par ailleurs, on vérifie facilement (en se servant de l'égalité $k_{n+1}=k_{n}+1$) que:
\[k_{n}\leq \frac{n+1-u_{0}}{r+1}.\]
Cette dernière entraîne que:
\begin{align*}
u_{n+1}-(r+1)u_{k_{n}}&=u_{0}+(n+1)r-(r+1)u_{0}-k_{n}r(r+1)\\ &\geq nr+r-u_{0}r-r(n+1-u_{0})=0.
\end{align*}
D'où l'on tire:
\begin{equation}\label{eee6}
\frac{u_{n+1}}{u_{k_{n}}}\geq (r+1).
\end{equation}
En combinant \eqref{eee5} et \eqref{eee6} et l'hypothèse de récurrence, on a abouti à:
\[v_{n+1,k_{n+1}}\geq v_{n,k_{n}}(r+1)\geq u_0(1+r)^{n}(1+r)=u_{0}(1+r)^{n+1}.\]
Ce qui achève cette récurrence et complète la preuve du théorème.
\end{proof}

\section{Minorations effectives du $\ppcm$ de certaines suites quadratiques}\label{far-oon}
Dans cette section, nous présentons les résultats de Farhi \cite{far} portant sur l'estimation du plus petit commun multiple d'une certaine classe de progressions quadratiques. On parlera ensuite d'une amélioration remarquable établie par Oon \cite{oon} en 2013. Notons que cette amélioration est adaptable à une généralisation  importante qu'on discutera dans la section suivante.
 
\begin{thm}[Farhi \cite{far}]\label{JJJ}
Soient $n\in\mathbb{N}$ et $(u_{k})_{k\in \mathbb{N}}$ la suite d'entiers de terme général donné par:
\[u_{k}:=ak(k+t)+b~~~~(\forall k\in \mathbb{N}),\]
avec $a,b\in \mathbb{N^*}$, $t\in\mathbb{N}$ et $\pgcd(a,b)=1$. Alors, on a:
\begin{equation}
\ppcm \left(u_{0},u_{1},\dots,u_{n}\right)\geq \begin{cases} 2b\left(\frac{a}{4}\right)^n &\text{si}~t=0, \\ \frac{b}{t2^{t}}\left(\frac{a}{4}\right)^n &\text{si}~t\geq 1. \end{cases}.
\end{equation} 
\end{thm}
\begin{proof}
D'après le lemme \ref{lF1}, le nombre $\ppcm \left(u_{0},u_{1},\dots,u_{n}\right)$ est multiple du nombre rationnel:
\[R:=\frac{u_{0}u_{1}\cdots u_{n}}{\ppcm\left\{\prod_{\begin{subarray}{c} 0\leq i\leq n \\ i\neq j \end{subarray}}(u_{i}-u_{j});~j=0,1,\dots,n\right\}}.\]
D'autre part, on vérifie facilement que pour tout $j\in\lbrace 0,1,\dots,n\rbrace$, on a:
\begin{equation}
\prod_{\begin{subarray}{c} 0\leq i\leq n \\ i\neq j \end{subarray}}(u_{i}-u_{j})=\begin{cases} (-1)^{j}a^{n}\frac{(n-j)!(n+j)!}{2} &\text{si}~t=0,\\ (-1)^{j}a^{n}\frac{(n-j)!(n+j+t)!}{\Phi(j,t)}\frac{1}{2j+t} &\text{si}~t\geq 1. \end{cases},
\end{equation}
où $\Phi(j,t):=1$ si $t=1$ et $\Phi(j,t):=(j+1)\cdots (j+t-1)$ si $t\geq 2$. Puisque $(n-j)!(n+j+t)!$ divise $(2n+t)!$ (car $\frac{(2n+t)!}{(n-j)!(n+j+t)!}=\binom{2n+t}{n-j}\in\mathbb{N^*}$) et pour tout $t\geq 1$, $\Phi(j,t)$ est multiple de $(t-1)!$ (car $\frac{\Phi(j,t)}{(t-1)!}=\binom{j+t-1}{t-1}\in\mathbb{N^*}$), on en déduit que le produit $\prod_{\begin{subarray}{c} 0\leq i\leq n \\ i\neq j \end{subarray}}(u_{i}-u_{j})$ divise l'entier strictement positif $f(t,n)$ défini par:
\begin{equation}\label{bbbb}
f(t,n):=\begin{cases} a^{n}\frac{(2n)!}{2}&\text{si}~t=0,\\ a^{n}\frac{(2n+t)!}{(t-1)!}&\text{si}~t\geq 1.\end{cases}.
\end{equation} 
Comme $f(t,n)$ est indépendant de $j\in \lbrace 0,1,\dots,n\rbrace$, il s'ensuit que:
\[\ppcm\left\{\prod_{\begin{subarray}{c} 0\leq i\leq n \\ i\neq j \end{subarray}}(u_{i}-u_{j});~j=0,1,\dots,n\right\}~~\text{divise}~~f(t,n).\]
Cette dernière relation entraîne que $R$ est multiple de $\frac{u_{0}u_{1}\cdots u_{n}}{f(t,n)}$. Par conséquent, l'entier $\ppcm\left(u_{0},u_{1},\dots,u_{n}\right)$ est aussi multiple de $\frac{u_{0}u_{1}\cdots u_{n}}{f(t,n)}$. Par ailleurs, puisque $a^{n}$ est premier avec tous les termes de la suite $(u_{k})_{k\in \mathbb{N}}$ (car $\pgcd(a,b)=1$) et $f(t,n)$ est multiple de $a^{n}$ (par définition), alors (en vertu du lemme de Gauss) l'entier $\ppcm\left(u_{0},u_{1},\dots,u_{n}\right)$ est multiple de $a^{n}\frac{u_{0}u_{1}\cdots u_{n}}{f(t,n)}$. D'où l'on a: $\ppcm\left(u_{0},u_{1},\dots,u_{n}\right)\geq a^{n}\frac{u_{0}u_{1}\cdots u_{n}}{f(t,n)}$. L'estimation requise découle alors de cette dernière, des inégalités $\binom{2n}{n}\leq 2^{2n}=4^{n}$, $\binom{2n+t}{n}\leq 2^{2n+t}=2^{t}4^{n}$ et de l'estimation:
\[u_{0}u_{1}\cdots u_{n}\geq b\left(a(1+t)\right)\left(2a(2+t)\right)\cdots\left(na(n+t)\right)=ba^{n}\frac{n!(n+t)!}{t!}.\]
Ce qui complète cette démonstration.
\end{proof}

\begin{rmq}
La minoration du théorème \ref{JJJ} est non triviale dès que $a\geq 5$. \`A priori, en appliquant ce résultat pour la suite $(n^{2}+1)_{n\in\mathbb{N^*}}$, on obtient une minoration triviale sans importance. Par contre, si $r$ est un entier $\geq 3$, le théorème \ref{JJJ} donne une minoration intéressante du $\ppcm$ de la suite $(r^{2}n^{2}+1)_{n\geq 1}$, qui est une sous-suite de $(n^{2}+1)_{n\in\mathbb{N^*}}$. En fait, on a:
\[\ppcm\left(1^2+1,2^2+1,\cdots,n^2+1\right)\geq \ppcm\left(r^2+1,r^{2}2^2+1,\cdots,r^{2}k^2+1\right),\]
où $k:=\lfloor\frac{n}{r}\rfloor$. Ensuite, le théorème \ref{JJJ} appliqué à la suite $(r^{2}n^{2}+1)_{n\geq 1}$ entraîne:
\begin{equation}\label{ffff}
\ppcm\left(1^2+1,2^2+1,\cdots,n^2+1\right)\geq 2\left(\frac{r^{2}}{4}\right)^{k}>2\left(\frac{r^{2}}{4}\right)^{\frac{n}{r}-1}=\frac{8}{r^{2}}\left\{\left(\frac{r}{2}\right)^{\frac{2}{r}}\right\}^{n}.
\end{equation}
En prenant enfin dans \eqref{ffff} $r=5$ (qui est la valeur de $r$ qui rend la quantité $\frac{8}{r^{2}}\left\{\left(\frac{r}{2}\right)^{\frac{2}{r}}\right\}^{n}$ maximale), on obtient le corollaire suivant: 
\end{rmq}

\begin{coll}[Farhi \cite{far}]\label{ffff1}
pour tout entier $n\geq 1$, on a:
\[\ppcm\left(1^2+1,2^2+1,\cdots,n^2+1\right)\geq 0,32(1,442)^{n}.\]
\end{coll}
\noindent Le corollaire \ref{ffff1} a été amélioré par Oon \cite{oon}, qui avait établi le résultat plus fort suivant:
 
\begin{thm}[Oon \cite{oon}]\label{oon'}
Soient $c,m,n\in \mathbb{N^*}$ tels que $m\leq \lceil \frac{n}{2}\rceil$. Alors, on a:
\[L_{c,m,n}:=\ppcm\left(m^{2}+c,(m+1)^{2}+c,\dots,n^{2}+c\right)\geq 2^{n}.\]
\end{thm}
\begin{proof}
Définissons le nombre $I_{c,m,n}$ comme suit:
\[I_{c,m,n}:=\int_{0}^{1}x^{m-1+\sqrt{c}i}(1-x)^{n-m}\mathrm{d}x,\]
(où $i$ désigne le nombre complexe $\sqrt{-1}$). Nous allons évaluer $I_{c,m,n}$ par deux méthodes différentes.\\ 
\textbullet{} \underline{\textbf{1\textsuperscript{ère}méthode:}} En développant par la formule du binôme l'expression $(1-x)^{n-m}$, dans $I_{c,m,n}$, on obtient:
\begin{align*}
I_{c,m,n}&=\sum_{k=0}^{n-m}(-1)^{k}\binom{n-m}{k}\int_{0}^{1}x^{m-1+k+\sqrt{c}i}\mathrm{d}x=\sum_{k=0}^{n-m}\frac{(-1)^{k}\binom{n-m}{k}}{m+k+\sqrt{c}i}\\ 
&=\sum_{k=0}^{n-m}\frac{(-1)^{k}\binom{n-m}{k}(m+k-\sqrt{c}i)}{(m+k)^{2}+c}.
\end{align*} 
Cette dernière expression montre que $I_{c,m,n}L_{c,m,n}$ est un nombre complexe de la forme $x+i\sqrt{c}y$, avec $x,y\in \mathbb{Z}$. Puisque $I_{c,m,n}L_{c,m,n}\neq 0$, il s'ensuit que $\left|I_{c,m,n}L_{c,m,n}\right|=x^{2}+cy^{2}\in \mathbb{N}^{*}$. Par conséquent, on a:
\begin{equation}\label{ffff3}
L_{c,m,n}\geq \frac{1}{\left|I_{c,m,n}\right|}.
\end{equation}
\textbullet{} \underline{\textbf{2\textsuperscript{nde}méthode}:} D'après les propriétés usuelles des fonctions $\Gamma$ et $\beta$ d'Euler, on a:
\begin{align*}
I_{c,m,n}&=\beta(m+\sqrt{c}i,n-m+1)=\frac{\Gamma(m+\sqrt{c}i)\Gamma(n-m+1)}{\Gamma(n+1+\sqrt{c}i)}=\frac{(n-m)!}{\frac{\Gamma(n+1+\sqrt{c}i)}{\Gamma(m+\sqrt{c}i)}}.
\end{align*}
D'où l'on a:
\begin{equation}\label{ffff4}
I_{c,m,n}=\frac{(n-m)!}{\prod_{k=m}^{n}(k+\sqrt{c}i)}.
\end{equation}
En substituant \eqref{ffff4} dans \eqref{ffff3}, on aboutit à:
\begin{equation}\label{ffff5}
L_{c,m,n}\geq \frac{\prod_{k=m}^{n}\sqrt{k^{2}+c}}{(n-m)!}\geq \frac{\prod_{k=m}^{n}k}{(n-m)!}=m\binom{n}{m}.
\end{equation}
Maintenant, l'estimation \eqref{ffff5} est vraie pour tous $n,m\in \mathbb{N}^{*}$ tels que $m\leq n$; en particulier, on a pour $m=\left\lceil \frac{n}{2}\right\rceil$:
\[L_{c,\left\lceil \frac{n}{2}\right\rceil,n}\geq \left\lceil \frac{n}{2}\right\rceil\binom{n}{\left\lceil \frac{n}{2}\right\rceil}.\]
Par ailleurs, on montre facilement par récurrence que pour tout entier $n\geq 7$, on a:
\begin{equation}\label{nsup7}
\left\lceil \frac{n}{2}\right\rceil\binom{n}{\left\lceil \frac{n}{2}\right\rceil}\geq 2^{n}.
\end{equation}
Ce qui entraîne que pour tous $n,m\in \mathbb{N}^{*}$ tels que $n\geq 7$ et $m\leq \left\lceil \frac{n}{2}\right\rceil$, on a:
\[L_{c,m,n}\geq L_{c,\left\lceil \frac{n}{2}\right\rceil,n}\geq 2^{n}.\] 
Pour les petites valeurs de $n$ (i.e., $n\leq 6$), on vérifie le résultat de façon calculatoire (cas par cas). Ce qui complète la démonstration du théorème.
\end{proof}

\section{Minorations effectives du $\ppcm$ de suites polynomiales}

Comme nous l'avons signalé dans la section \textsection\ref{far-oon}, la méthode de Oon est adaptable à une généralisation importante. En 2013, Hong et al. \cite{hong2} ont généralisé le résultat de Oon (i.e., le théorème \ref{oon'}) en montrant qu'il reste vrai en remplaçant la suite $(n^{2}+c)_{n}$ par n'importe quelle suite polynomiale (non constante), ayant des coefficients entiers positifs. La méthode de Hong et al. utilise quelques arguments algébriques, ainsi que certaines identités binomiales que nous présenterons dans ce qui va suivre.
   
\begin{thm}[Hong et al. \cite{hong2}]\label{HH1}
Soient $n$ un entier $\geq 7$ et $f\in \mathbb{Z}[X]$ un polynôme non constant, à coefficients positifs. Alors, pour tout entier $0<m\leq \left\lceil \frac{n}{2}\right\rceil$, on a:
\[\ppcm\left(f(m),f(m+1),\dots,f(n)\right) \geq 2^{n}.\]
\end{thm}

\noindent La preuve de ce théorème est partagée en plusieurs lemmes. Rappelons d'abord la définition d'un entier algébrique.

\begin{defi}
Un nombre complexe $\alpha$ est appelé un ``entier algébrique", s'il est racine d'un polynôme unitaire $f\in\mathbb{Z}[X]$.
\end{defi}

\begin{lemme}\label{ppta}
On a les propriétés suivantes:
\begin{enumerate}
\item Les entiers algébriques rationnels sont simplement les entiers rationnels.\label{ppp1}
\item Si $\alpha\in\mathbb{C}$ est un entier algébrique, alors il en est de même des nombres $k\alpha$ $(k\in\mathbb{Z})$.\label{ppp2}
\end{enumerate} 
\end{lemme}
\begin{proof}~\\
\textbullet{} Montrons le point \ref{ppp1} du lemme. Il est immédiat que tout entier rationnel est un entier algébrique (car: si $m\in\mathbb{Z}$, alors $m$ est racine du polynôme $\left(X-m\right)\in\mathbb{Z}[X]$). Inversement, supposons que $m\in\mathbb{Q}$ est racine d'un polynôme unitaire $f(X)=X^n+a_{n-1}X^{n-1}+\dots+a_0\in\mathbb{Z}[X]$. Il existe $a,b\in\mathbb{Z}$ tels que $b\neq 0$, $\pgcd(a,b)=1$ et $m=\frac{a}{b}$. On a par suite:
\[\left(\frac{a}{b}\right)^n+a_{n-1}\left(\frac{a}{b}\right)^{n-1}+\dots+a_0=0.\]
En multipliant les deux membres de cette dernière par $b^{n-1}$, on obtient:
\[\frac{a^{n-1}}{b}=-a_{n-1}a^{n-1}-a_{n-2}a^{n-1}b-\dots-a_0b^{n-1}\in\mathbb{Z},\]
ce qui montre que $b\mid a^{n-1}$. Comme $\pgcd(a,b)=1$, alors on a forcément $b=\pm 1$ et donc $m=\pm a\in\mathbb{Z}$, comme il fallait le prouver. \\
\textbullet{} Montrons maintenant le point \ref{ppp2} du lemme. Supposons que $\alpha$ est racine d'un polynôme unitaire $f(X)=X^n+a_{n-1}X^{n-1}+\dots+a_0\in\mathbb{Z}[X]$. Puisque $f(\alpha)=0$, alors $k^nf(\alpha)=0$; cela entraîne que $k\alpha$ est racine du polynôme $g\in\mathbb{Z}[X]$, donné par:
\[g(X):=X^n+ka_{n-1}X^{n-1}+k^2a_{n-2}X^{n-2}+\dots+k^na_0,\]
comme il fallait le prouver. Le lemme est ainsi démontré.
\end{proof}

Dans ce qui suit, nous utilisons librement les résultats du lemme \ref{ppta}, sans se référer à ce dernier.

\begin{lemme}[Hong et al. \cite{hong2}]\label{alg1}
Soient $s$ un entier strictement positif et $f(X)=\sum_{i=0}^{s}a_{i}X^{i}$ $\in \mathbb{Z}[X]$ un polynôme de degré $s$. Désignons par $\alpha_{1},\alpha_2,\dots,\alpha_{s}$ les racines complexes de $f$ comptées avec leurs multiplicités. Alors, pour tout $j\in\{1,2,\dots,s\}$, le nombre $\beta_j:=a_{s}\left(\prod_{\begin{subarray}{c} 1\leq i\leq s \\ i\neq j\end{subarray}}\alpha_{i}\right)$ est un entier algébrique.
\end{lemme}
\begin{proof}
Pour $s=1$, le résultat du lemme est trivial. Supposons pour la suite que $s\geq 2$. Si au moins deux racines de $f$ sont nulles, alors $\beta_{j}=0$ $(\forall j\in\{1,2,\dots,s\})$, ce qui conclut au résultat requis. Si une seule racine $\alpha_{t}$ (pour un certain $t\in\{1,2,\dots,s\}$) est nulle, alors $\beta_{t}=(-1)^{s-1}a_{1}\in \mathbb{Z}$ et $\beta_{j}=0$ $(\forall j\neq t)$, ce qui permet de conclure. Maintenant, si toutes les racines $\alpha_{1},\alpha_2,\dots,\alpha_{s}$ sont non nulles, alors $\beta_{j}=\frac{(-1)^{s}a_{0}}{\alpha_{j}}$ $(\forall j\in\{1,2,\dots,s\})$. Il suffit donc de montrer que pour tout $j\in\{1,2,\dots,s\}$, le nombre $\frac{a_{0}}{\alpha_{j}}$ est un entier algébrique. Fixons $j\in\{1,2,\dots,s\}$. Puisque $f(\alpha_{j})=0$, il s'ensuit que:
\[\frac{a_{0}^{s-1}}{\alpha_{j}^{s}}f(\alpha_{i})=\left(\frac{a_{0}}{\alpha_{j}}\right)^{s}+a_{1}\left(\frac{a_{0}}{\alpha_{j}}\right)^{s-1}+\dots+a_{s-1}a_{0}^{s-2}\left(\frac{a_{0}}{\alpha_{j}}\right)+a_{s}a_{0}^{s-1}=0.\]
Par conséquent, le nombre $\frac{a_{0}}{\alpha_{j}}$ est racine du polynôme unitaire $g(X)=X^{s}+a_{1}X^{s-1}+\dots+a_{s-1}a_{0}^{s-2}X+a_{s}a_{0}^{s-1}\in \mathbb{Z}[X]$, comme il fallait le prouver. Le lemme est démontré. 
\end{proof}

\begin{lemme}[Hong et al. \cite{hong2}]\label{binom1}
Soient $m$ et $n$ deux entiers strictement positifs tels que $m\leq n$. Alors, pour tout $z\in\mathbb{C}\setminus\lbrace m,m+1,\dots,n\rbrace$, on a:
\begin{equation}\label{hongr}
\frac{(n-m)!}{\prod_{k=m}^{n}(k-z)}=\sum_{k=m}^{n}(-1)^{k-m}\binom{n-m}{k-m}\frac{1}{k-z}.
\end{equation}
\end{lemme}
\begin{proof}
Considérons le polynôme de Lagrange $L$, donné par:
\[L(z):=\sum_{k=m}^{n}\prod_{\begin{subarray}{c} m\leq j\leq n \\ j\neq k \end{subarray}}\left(\frac{j-z}{j-k}\right)=\sum_{k=m}^{n}\frac{(-1)^{k-m}}{(k-m)!(n-k)!}\prod_{\begin{subarray}{c} m\leq j\leq n \\ j\neq k \end{subarray}}(j-z).\]
On a visiblement $\deg(L-1)\leq n-m$. Puisque le polynôme $(L-1)$ s'annule en chacun des nombres $m,m+1,\dots,n$, alors le nombre de racines de $(L-1)$ est strictement plus grand que son degré. Cela entraîne que $(L-1)$ n'est rien d'autre que le polynôme nul. Il s'ensuit que pour tout $z\in\mathbb{C}\setminus\lbrace m,m+1,\dots,n\rbrace$, on a:
\begin{align*}
(n-m)!&=(n-m)!L(z)=\sum_{k=m}^{n}(-1)^{k-m}\binom{n-m}{k-m}\prod_{\begin{subarray}{c} m\leq j\leq n \\ j\neq k \end{subarray}}(j-z)\\&=\left\lbrace\prod_{k=m}^{n}(k-z)\right\rbrace\sum_{k=m}^{n}(-1)^{k-m}\binom{n-m}{k-m}\frac{1}{k-z},
\end{align*}
confirmant ainsi l'identité \eqref{hongr}. La démonstration du lemme est achevée.
\end{proof}

\begin{lemme}[Hong et al. \cite{hong2}]\label{HH0}
Soient $m,n,s\in\mathbb{N}^{*}$ tels que $m\leq n$ et soit $f(x)=\sum_{i=0}^{s}a_{i}X^{i}\in \mathbb{Z}[X]$ un polynôme de degré $s$. Si $f$ ne s'annule en aucun entier de l'intervalle $[m,n]$, alors on a:
\[\prod_{k=m}^{n}f(k)~~\text{divise}~~{a_{s}}^{n-m+1}\left((n-m)!\right)^{s}\left(\ppcm\left(f(m),f(m+1),\dots,f(n)\right)\right)^s.\] 
\end{lemme}
\begin{proof}
Désignons par $\alpha_{1},\alpha_{2},\dots,\alpha_{s}$ les racines complexes de $f$ comptées avec leurs multiplicités. Donc, on peut écrire $f(X)=a_{s}\left(X-\alpha_{1}\right)\left(X-\alpha_{2}\right)\cdots\left(X-\alpha_{s}\right)$. Soit $k\in\{m,m+1,\dots,n\}$. Il est immédiat que le polynôme $h_{k}(X):=(-1)^{s}f(k-X)$ est dans $\mathbb{Z}[X]$, son degré est égal à $s$, son coefficient dominant est égal à $a_{s}$ et ses racines (comptées avec leurs multiplicités) sont $k-\alpha_{1},k-\alpha_{2},\dots,k-\alpha_{s}$. En appliquant le lemme \ref{alg1} sur le polynôme $h_{k}$, on en déduit que $\frac{f(k)}{k-\alpha_{j}}=a_{s}\prod_{\begin{subarray}{c}1\leq i\leq s \\ i\neq j \end{subarray}}(k-\alpha_{i})$ est un entier algébrique $(\forall j\in\{1,2,\dots,s\})$. Cela est aussi vrai pour tout $k\in\{m,m+1,\dots,n\}$ (puisque $k$ est arbitraire). Par ailleurs, on a (en vertu du lemme \ref{binom1}):
\begin{equation}\label{binom2}
\frac{(n-m)!}{\prod_{k=m}^{n}(k-\alpha_{j})}=\sum_{k=m}^{n}(-1)^{k-m}\binom{n-m}{k-m}\frac{1}{k-\alpha_{j}}~~~~(\forall j\in\{1,2,\dots,s\}).
\end{equation}
En multipliant les deux membres de \eqref{binom2} par $\ppcm\left(f(m),f(m+1),\dots,f(n)\right)$ et en se servant de ce qui précède, il s'ensuit que les nombres $Q_1,Q_2,\dots,Q_s$ donnés par:
\begin{equation}
Q_{j}:=\frac{(n-m)!\ppcm\left(f(m),f(m+1),\dots,f(n)\right)}{\prod_{k=m}^{n}(k-\alpha_{j})}~~~~(\forall j\in\{1,2,\dots,s\}),
\end{equation}
sont tous des entiers algébriques. Donc, leur produit $Q:=\prod_{j=1}^{s}Q_{j}$ l'est aussi. Le résultat requis par le lemme découle du fait que tout entier algébrique rationnel est un entier rationnel (voir le point \ref{ppp1} du lemme \ref{ppta}) et de l'expression suivante montrant en particulier que $Q$ est un nombre rationnel:
\[Q:=\frac{{a_{s}}^{n-m+1}\left((n-m)!)\right)^{s}\left(\ppcm\left(f(m),f(m+1),\dots,f(n)\right)\right)^{s}}{\prod_{k=m}^{n}f(k)}\in\mathbb{Q}.\]
Ce qui complète cette démonstration.       
\end{proof}



\begin{proof}[Démonstration du théorème \ref{HH1}]
Le lemme \ref{HH0} entraîne que:
\begin{align*}
\ppcm\left(f(m),f(m+1),\dots,f(n)\right)&\geq \ppcm\left(f(\left\lceil n/2\right\rceil),f(\left\lceil n/2\right\rceil+1),\dots,f(n)\right)\\&\geq \frac{\prod_{k=\left\lceil n/2\right\rceil}^{n}\left|f(k)/a_{s}\right|^{\frac{1}{s}}}{\left(n-\left\lceil n/2\right\rceil\right)!}\geq \frac{\prod_{k=\left\lceil n/2\right\rceil}^{n}k}{\left(n-\left\lceil n/2\right\rceil\right)!}=\left\lceil \frac{n}{2}\right\rceil\binom{n}{\left\lceil \frac{n}{2}\right\rceil}.
\end{align*}
En combinant cette dernière avec \eqref{nsup7}, on a:
\[\ppcm\left(f(m),f(m+1),\dots,f(n)\right)\geq 2^{n}~~~~(\forall n\geq 7).\]
Ce qui complète cette démonstration.
\end{proof}

\section{Estimations asymptotiques} 
Dans cette section on s'intéresse à l'étude du comportement asymptotique du nombre: 
\[\log \ppcm (f(1),f(2),\dots,f(n)),\]
où $f\in\mathbb{Z}[X]$ et $\deg f\in\{1,2\}$. On commencera par le résultat de Bateman et al. \cite{bat} dont on présentera une démonstration; nous énonçons ensuite le résultat de Cilleruelo \cite{cil} sans démonstration.   

\begin{thm}[Bateman et al. \cite{bat}]
Soient $a,b\in \mathbb{Z}$ tels que $b>0$, $a+b>0$ et $\pgcd(a,b)=1$. Alors, on a:
\begin{equation}\label{batem}
\log\ppcm\left(a+b,a+2b,\dots,a+nb\right)\sim_{+\infty}\left(\frac{b}{\varphi(b)}\sum_{\begin{subarray}{c} 1\leq m\leq b \\ \pgcd(m,b)=1\end{subarray}}\frac{1}{m}\right)n,
\end{equation}
où $\varphi$ désigne la fonction indicatrice d'Euler.
\end{thm}
\begin{proof}
Fixons $n\in\mathbb{N^*}$ et définissons:
\begin{align*}
T(b)&:=\left\lbrace m\in\mathbb{N^*};~m\leq b~\text{et}~\pgcd(m,b)=1\right\rbrace,\\
L_{a,b,n}&:=\ppcm\left(a+b,a+2b,\dots,a+nb\right),\\
S(n)&:=\lbrace a+b,a+2b,\dots,a+nb\rbrace.
\end{align*}
Désignons par $P(n)$ l'ensemble des facteurs premiers de $L_{a,b,n}$ et posons $M(n):=\prod_{p\in P(n)}p$. En écrivant $L_{a,b,n}=\prod_{p\in P(n)}p^{a_{p}}$, avec $a_{p}=a_{p}(n):=\vartheta_{p}(L_{a,b,n})$ $(\forall p\in P(n))$, on a: $\frac{L_{a,b,n}}{M(n)}=\prod p^{a_{p}-1}$, où le produit porte sur tous les nombres premiers $p\in P(n)$ tels que $p^{2}$ divise $L_{a,b,n}$. Puisque $\max S(n)=a+nb$, il s'ensuit que: $p^{a_{p}-1}\leq a+nb$ $(\forall p\in P(n))$ et $\#\lbrace p\in P(n);~p^{2}~\text{divise}~L_{a,b,n}\rbrace\leq \sqrt{a+nb}$. On a par conséquent:
\[0\leq \log L_{a,b,n}-\log M(n)\leq \sqrt{a+nb}\log (a+nb)~~~~(\forall n\in\mathbb{N^*}).\] 
Ce qui entraîne que:
\[\lim_{n\rightarrow +\infty}\frac{\log L_{a,b,n}-\log M(n)}{n}=0.\]
Il suffit donc d'étudier $\lim_{n\rightarrow +\infty}\frac{\log M(n)}{n}$. Pour ce faire, nous caractérisons d'abord les nombres premiers de l'ensemble $P(n)$. Soit $p$ un nombre premier tel que $\pgcd(p,b)=1$. Si $m$ désigne le reste de la division euclidienne de $p$ par $b$, alors $m\in T(b)$ et $p\equiv m\pmod b$. Puisque $\pgcd(a,b)=1$, pour un tel $m$ il existe un unique $m'\in T(b)$ vérifiant $mm'\equiv a\pmod b$, on a en particulier $m'p\equiv a\pmod b$. Nous constatons que $m'p$ est le plus petit multiple positif de $p$ tel que $m'p\equiv a\pmod b$. On en déduit que $m'p\in S(n)$ si et seulement si $m'p\leq a+nb$; ce qui revient à dire que $p\in P(n)$ si et seulement si $p\leq \frac{a+nb}{m'}$. En désignant par $U(m)$ ($\forall m\in T(b)$) l'ensemble des nombres premiers $p$ tels que $p\equiv m\pmod b$ et $p\leq \frac{a+nb}{m'}$, il s'ensuit que:
\[\log M(n)=\sum_{m\in T(b)}\sum_{p\in U(m)}\log p=\sum_{m\in T(b)}\theta\left(\frac{a+nb}{m'};m,b\right),\]
où $\theta(x;m,b)$ désigne la fonction de Chebyshev généralisée. Par ailleurs, le théorème des nombres premiers pour les progressions arithmétiques (voir par exemple \cite{bla}, p. $72$) montre que l'on a:
\[\theta(x;m,b)=\frac{x}{\varphi(b)}+o(x).\]
Ce qui donne:
\[\log M(n)=\sum_{m\in T(b)}\left(\frac{a+nb}{m'\varphi(b)}+o(n)\right).\]     
D'où:
\[\lim_{n\rightarrow +\infty}\frac{\log M(n)}{n}=\sum_{m\in T(b)}\frac{b}{m'\varphi(b)}=\sum_{m\in T(b)}\frac{b}{m\varphi(b)},\]
où la dernière égalité est due au fait que $m'$ parcourt l'ensemble $T(b)$ tout comme $m$.\\ Ce qui termine cette démonstration.  
\end{proof}

\begin{thm}[Cilleruelo \cite{cil}]
Pour tout polynôme irréductible $f(X)=aX^{2}+bX+c\in\mathbb{Z}[X]$, on a:
\[\log \ppcm\left(f(1),f(2),\dots,f(n)\right)= n\log n +Bn +o(n),\]
où $B=B_{f}$ est la constante définie par:
\begin{align*}
B_{f}:&=\gamma -1-2\log 2 -\sum_{p}\left(\frac{d}{p}\right)\frac{\log p}{p-1}+\frac{1}{\varphi(q)}\sum_{\begin{subarray}{c}1\leq r\leq q \\ \pgcd(r,q)=1 \end{subarray}}\log\left(1+\frac{r}{q}\right)\\ & +\log a +\sum_{p|2aD}\log p \left(\frac{1+\left(\frac{d}{p}\right)}{p-1} -\sum_{k\geq 1}\frac{s\left(f,p^{k}\right)}{p^{k}}\right),
\end{align*}
avec $\gamma$ désigne la constante d'Euler, $D:=b^{2}-4ac$, $d$ est la partie sans facteur carré de $D$, $\left(\frac{\cdot}{\cdot}\right)$ est le symbole de Legendre, $q:=\frac{a}{\pgcd(a,b)}$ et $s\left(f,p^{k}\right)$ est le nombre de solutions de la congruence $f(X)\equiv 0\pmod {p^{k}}$.
\end{thm}
\noindent La démonstration de ce théorème est un peu longue; elle se sert de quelques résultats concernant la répartition des racines d'un polynôme quadratique modulo des puissances de nombres premiers. Pour une preuve bien détaillée, le lecteur est invité à consulter l'article \cite{cil} de Cilleruelo. 

\bigskip

\begin{center}
\includegraphics[scale=1]{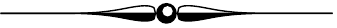}
\end{center}

\chapter{Minorations non triviales du $\ppcm$ de la suite $(n^2+c)_n$}\label{ch2}

\section{Introduction}

Ce chapitre (dont les résultats sont publiés dans \cite{Bousfar1}) est consacré à l'étude des nombres: 
\[L_{c,m,n}:=\ppcm\{m^2+c,(m+1)^2+c,\dots,n^2+c\},\]
où $c,m,n$ sont des entiers strictement positifs tels que $m\leq n$. Plus précisément, nous utilisons des arguments d'algèbre commutative et d'analyse complexe pour établir de nouvelles minorations non triviales de $L_{c,m,n}$. Le reste de ce chapitre est organisé en quatre parties. Dans la première partie, nous donnons un lemme algébrique qui nous permet, d'une part, de redémontrer le théorème \ref{oon'} de Oon par une méthode facile et purement algébrique, et d'autre part de reformuler le problème de minoration du nombre $L_{c,m,n}$. Dans cette reformulation, nous sommes amenés à introduire une fonction arithmétique, notée $h_c$, dont un multiple fournit un diviseur pour $L_{c,m,n}$. Dans les deux parties suivantes, nous étudions la fonction arithmétique $h_c$ et nous lui trouvons un multiple simple et non trivial. Dans la dernière partie, nous utilisons le multiple obtenu de $h_c$ pour déduire un diviseur non trivial pour $L_{c,m,n}$. Notre nouvelle minoration non triviale pour $L_{c,m,n}$ découle alors de ce diviseur.

Étant donné un polynôme $P\in\mathbb{C}[X]$, on désigne par $\overline{P}$ le polynôme conjugué de $P$ dans $\mathbb{C}[X]$ (i.e., le polynôme que nous obtenons en remplaçant chaque coefficient de $P$ par son conjugué complexe). Il est connu que la conjugaison des polynômes dans $\mathbb{C}[X]$ est compatible avec l'addition et la multiplication, c'est-à-dire que pour tous $P,Q\in\mathbb{C}[X]$, on a: $\overline{P+Q}=\overline{P}+\overline{Q}$ et $\overline{P\cdot Q}=\overline{P}\cdot\overline{Q}$. Par ailleurs, on désigne par $I$, $E_h$ $(h\in\mathbb{R})$ et $\Delta$ les opérateurs linéaires de $\mathbb{C}[X]$ qui représentent respectivement l'identité, l'opérateur de translation de pas $h$ ($E_hP\left(X\right)=P\left(X+h\right)$, $\forall P\in\mathbb{C}[X]$) et l'opérateur de différence avant ($\Delta P\left(X\right)=P\left(X+1\right)-P\left(X\right)$, $\forall P\in\mathbb{C}[X]$). Pour $n\in\mathbb{N}$, l'expression de $\Delta^n$ en fonction de $E_h$ s'obtient facilement à partir de la formule du binôme, comme suit:
\begin{equation}\label{dece}
\Delta^n=\left(E_1-I\right)^n=\sum_{m=0}^{n}(-1)^{n-m}\binom{n}{m}E_{1}^m=\sum_{m=0}^{n}(-1)^{n-m}\binom{n}{m}E_{m}.
\end{equation} 
Enfin, nous utilisons la notation de Knuth pour la factorielle décroissante:
\[X^{\underline{n}}:=X\left(X-1\right)\left(X-2\right)\cdots\left(X-n+1\right)~~~~(\forall n\in\mathbb{N}).\] 

\section{La méthode algébrique}

Bien que la méthode d'obtention du résultat de Oon \cite{oon} (i.e, le théorème \ref{oon'}) est d'apparence analytique, les ingrédients qui font son succès sont, en profondeur, algébriques! Nous allons montrer ce fait à travers le lemme algébrique fondamental suivant:

\begin{lemme}[fondamental]\label{lf'}
Soit $\mathcal{A}$ un anneau commutatif, unitaire et intègre et soient $n$ un entier strictement positif et $u_{0},u_{1},\dots,u_{n}$, $a,b$ des éléments de $\mathcal{A}$. Supposons que $a$ et $b$ vérifient les conditions suivantes: 
\begin{enumerate}
\item Chacun des éléments $u_{0},u_{1},\dots,u_{n}$ de $\mathcal{A}$ est un diviseur de $a$.
\item Chacun des éléments $\prod_{\begin{subarray}{c} 0\leq j\leq n \\ j\neq i\end{subarray}}\left(u_{i}-u_{j}\right)$ $(i=0,1,\dots,n)$ de $\mathcal{A}$ est un diviseur de $b$.
\end{enumerate}
Alors, le produit $ab$ est multiple du produit $u_{0}u_{1}\cdots u_{n}$.
\end{lemme}
\begin{proof}
Lorsque les éléments $u_{0},u_{1},\dots,u_{n}$ de $\mathcal{A}$ ne sont pas deux à deux distincts (i.e., il existe $i,j\in\{0,1,\dots,n\}$, avec $i\neq j$, tels que $u_i=u_j$), le résultat du lemme est immédiat puisqu'on aura $b=0_{\mathcal{A}}$. Supposons donc, pour toute la suite que les $u_{i}~(i=0,1,\dots,n)$ sont deux-à-deux distincts.\\
Étant donné que $\mathcal{A}$ est commutatif, unitaire et intègre, tout polynôme non identiquement nul de $\mathcal{A}[X]$, d'un certain degré $d\in \mathbb{N}$, possède au plus $d$ racines dans $\mathcal{A}$. C'est sur ce résultat bien connu que nous appuyons pour prouver le lemme. Comme $a$ est multiple de chacun des éléments $u_{0},u_{1},\dots,u_{n}$ de $\mathcal{A}$, il existe $k_{0},k_{1},\dots,k_{n}\in \mathcal{A}$ tels que:
\begin{equation}\label{res1}
a=k_{0}u_{0}=k_{1}u_{1}=\dots=k_{n}u_{n}.
\end{equation}
De même, comme $b$ est multiple de chacun des éléments $\prod_{\begin{subarray}{c} 0\leq j\leq n \\ j\neq i \end{subarray}}(u_{i}-u_{j})$ $(i=0,1,\dots,n)$ de $\mathcal{A}$, alors il existe $\ell_{0},\ell_{1},\dots,\ell_{n}\in \mathcal{A}$ tels que:
\begin{equation}\label{res1'}
b=\ell_{i}\prod_{\begin{subarray}{c} 0\leq j\leq n \\ j\neq i \end{subarray}}(u_{i}-u_{j})~~(\forall i\in \lbrace 0,1,\dots,n\rbrace).
\end{equation}
Considérons le polynôme $P\in\mathcal{A}[X]$ suivant:
\[P(X):=\sum_{i=0}^{n}\left[\ell_{i}\prod_{\begin{subarray}{c} 0\leq j\leq n \\ j\neq i \end{subarray}}(X-u_{j}) \right]-b.\]
Il est immédiat que $\deg P \leq n$. D'autre part, on a (d'après \eqref{res1'}): 
\[P(u_{i})=0~~~~(\forall i\in \lbrace 0,1,\dots,n\rbrace).\]
Ce qui montre que $P$ possède au moins $(n+1)$ racines distinctes dans $\mathcal{A}$. Il s'ensuit (en vertu du résultat signalé au début de cette démonstration) que $P$ est identiquement nul. D'où (en particulier) $P(0)=0$; ce qui donne:
\[b=(-1)^{n}\sum_{i=0}^{n}\ell_{i}\left(\prod_{\begin{subarray}{c} 0\leq j\leq n \\ j\neq i\end{subarray}}u_{j}\right).\]
En multipliant les deux membres de cette dernière égalité par $a$, il en résulte que:
\begin{align*}
ab &=(-1)^{n}\sum_{i=0}^{n}\ell_{i}a\left(\prod_{\begin{subarray}{c} 0\leq j\leq n \\ j\neq i\end{subarray}}u_{j}\right)\\&=(-1)^{n}\sum_{i=0}^{n}\ell_{i}k_{i}u_{i}\left(\prod_{\begin{subarray}{c} 0\leq j\leq n \\ j\neq i\end{subarray}}u_{j}\right)~~~~\text{(en vertu de \eqref{res1})}\\ &=(-1)^{n}\left(\sum_{i=0}^{n}k_{i}\ell_{i}\right)u_{0}u_{1}\cdots u_{n}. 
\end{align*}
Ce qui montre bien que $ab$ est multiple de $u_{0}u_{1}\cdots u_{n}$. Le lemme est ainsi démontré.
\end{proof}

\begin{rmq}
Le lemme \ref{lf'} est inspiré du lemme \ref{lF1} de Farhi \cite{far} qui en devient un cas particulier lorsqu'on prend $\mathcal{A}=\mathbb{Z}$, $a= \ppcm(u_0,u_1,\dots,u_n)$ et
\[b=\ppcm\left\lbrace\prod_{\begin{subarray}{c}0 \leq j \leq n \\ j \neq i
\end{subarray}} (u_i - u_j);~i = 0 , 1 , \dots , n\right\rbrace.\]
C'est précisément ce cas particulier qui a conduit Farhi \cite{far} à établir les premières minorations non triviales du $\ppcm$ d'une suite arithmétique et d'un certain type de suites quadratiques (voir les théorèmes \ref{Flcm} et \ref{JJJ}).
\end{rmq}

Maintenant, nous utilisons le lemme \ref{lf'} pour établir une nouvelle démonstration (purement algébrique) du théorème \ref{oon'} de Oon.

\begin{proof}[Démonstration algébrique du théorème \ref{oon'}]
Puisque $L_{c,m,n}$ est clairement croissant par rapport à $m$, alors il suffit de prouver le résultat du théorème pour $m=\left\lceil \frac{n}{2}\right\rceil$. Pour simplifier, posons $m_0:=\left\lceil \frac{n}{2}\right\rceil$. Nous devons donc démontrer que $L_{c,m_0,n}\geq 2^n$. Pour $n\in\{1,2,\dots,6\}$, cela peut être facilement vérifié à la main (comme le fait Oon). Supposons pour la suite que $n\geq 7$. Il est connu que pour tout entier $r\geq 7$, on a: $\left\lceil \frac{r}{2}\right\rceil\binom{r}{\left\lceil \frac{r}{2}\right\rceil}\geq 2^{r}$. En vertu de cette dernière inégalité pour $r=n$, il suffit de montrer que $L_{c,m_0,n}\geq m_0\binom{n}{m_0}$. Plus généralement, nous allons montrer que:
\begin{equation}\label{cor2}
L_{c,m',n}\geq m'\binom{n}{m'}~~~~(\forall m'\in\mathbb{N^*},~m'\leq n).
\end{equation}
Soit $m'\in\mathbb{N^*}$ tel que $m'\leq n$. Pour prouver \eqref{cor2}, nous appliquons le lemme \ref{lf'} pour $\mathcal{A}=\mathbb{Z}[\sqrt{-c}]$ en prenant à la place des $u_i$ les éléments $m'+\sqrt{-c},m'+1+\sqrt{-c},\dots,n+\sqrt{-c}$ de $\mathcal{A}$ et pour $a$ et $b$ les entiers $a=L_{c,m',n}$ et $b=(n-m')!$. Pour tout $k\in\left\lbrace m',m'+1,\dots,n\right\rbrace$, comme $L_{c,m',n}$ est clairement multiple (dans $\mathbb{Z}$, donc aussi dans $\mathcal{A}=\mathbb{Z}[\sqrt{-c}]$) de $(k^2+c)$ et $(k^2+c)=\left(k+\sqrt{-c}\right)\left(k-\sqrt{-c}\right)$ est multiple (dans $\mathbb{Z}[\sqrt{-c}]$) de $\left(k+\sqrt{-c}\right)$, alors $L_{c,m',n}$ est multiple (dans $\mathbb{Z}[\sqrt{-c}]$) de $\left(k+\sqrt{-c}\right)$. Cela montre que la première condition du lemme \ref{lf'} est satisfaite. D'autre part, on a pour tout $k\in\left\lbrace m',m'+1,\dots,n\right\rbrace$: 
\[\prod_{\begin{subarray}{c}m'\leq \ell\leq n \\ \ell\neq k\end{subarray}}\left\lbrace\left(k+\sqrt{-c}\right)-\left(\ell+\sqrt{-c}\right)\right\rbrace=\prod_{\begin{subarray}{c}m'\leq \ell\leq n\\ \ell\neq k\end{subarray}}(k-\ell)=(-1)^{n-k}(k-m')!(n-k)!,\]
qui divise (dans $\mathbb{Z}$, donc aussi dans $\mathbb{Z}[\sqrt{-c}]$) l'entier $(n-m')!$ (car: $\frac{(n-m')!}{(k-m')!(n-k)!}=\binom{n-m'}{k-m'}\in\mathbb{Z}$). Cela montre que la deuxième condition du lemme \ref{lf'} est aussi satisfaite. Nous en déduisons donc (en appliquant le lemme \ref{lf'}) que $L_{c,m',n}(n-m')!$ est multiple (dans $\mathbb{Z}[\sqrt{-c}]$) de $\prod_{k=m'}^{n}\left(k+\sqrt{-c}\right)$. Il existe par conséquent $x,y\in\mathbb{Z}$ tels que:
\begin{equation}\label{**}
L_{c,m',n}(n-m')!=\left(x+y\sqrt{-c}\right)\prod_{k=m'}^{n}\left(k+\sqrt{-c}\right).\end{equation}
Ainsi, en prenant les modules dans $\mathbb{C}$ des deux côtés, on obtient:
\[L_{c,m',n}(n-m')!=\sqrt{x^2+cy^{2}}\prod_{k=m'}^{n}\sqrt{k^2+c}.\]
Par suite, comme $x^2+cy^2\in\mathbb{N}$ et $x^2+cy^2\neq 0$ (car: $x^2+cy^2=0$ $\Longrightarrow$ $L_{c,m',n}=0$, ce qui est faux), alors $x^2+cy^2\geq 1$. D'où:
\[L_{c,m',n}=\frac{\sqrt{x^2+cy^2}\prod_{k=m'}^{n}\sqrt{k^2+c}}{(n-m')!}\geq \frac{\prod_{k=m'}^{n}\sqrt{k^2+c}}{(n-m')!}\geq \frac{\prod_{k=m'}^{n}k}{(n-m')!}=m'\binom{n}{m'},\]
comme il fallait le prouver. Ce qui achève la démonstration du théorème.
\end{proof}

\noindent Maintenant, il est naturel de soulever la question suivante:
\begin{quote}
\textit{Comment pourrait-on améliorer la minoration $L_{c,m,n}\geq \frac{\prod_{k=m}^{n}\sqrt{k^2+c}}{(n-m)!}$, obtenue lors de la preuve du théorème \ref{oon'} et établie initialement par Oon \cite{oon}?}
\end{quote}

\noindent Pour discuter cette question, nous aurons besoin de la définition suivante du $\pgcd$ et du $\ppcm$ dans un anneau commutatif unitaire.

\begin{defi}
Soit $\mathcal{A}$ un anneau commutatif unitaire et soient $a,b\in\mathcal{A}$. Un élément $d$ de $\mathcal{A}$ est appelé un plus grand commun diviseur de $a$ et $b$ (et est désigné par ${\pgcd}_{\mathcal{A}}(a,b)$) si $d$ divise à la fois $a$ et $b$ et si tout autre élément $d'$ de $\mathcal{A}$, qui divise à la fois $a$ et $b$, divise également $d$. De même, un élément $m$ de $\mathcal{A}$ est appelé un plus petit commun multiple de $a$ et $b$ (et est désigné par ${\ppcm}_{\mathcal{A}}(a,b)$) si $m$ est un multiple de $a$ et $b$ et si tout autre élément $m'$ de $\mathcal{A}$, qui est un multiple de $a$ et $b$ à la fois, est également un multiple de $m$. Noter que ${\pgcd}_{\mathcal{A}}(a,b)$ et ${\ppcm}_{\mathcal{A}}(a,b)$ existent au moins lorsque $\mathcal{A}$ est un anneau factoriel (ce qui est le cas de l'anneau des entiers de Gauss $\mathbb{Z}[i]$) et ils sont uniques à une multiplication près par une unité.
\end{defi}

Maintenant, pour simplifier, supposons que $c=1$ et soient $m,n\in\mathbb{N^*}$ tels que $m\leq n$. En vertu de la formule \eqref{**}, l'entier strictement positif $L_{1,m,n}(n-m)!$ est multiple (dans $\mathbb{Z}[i]$) de l'entier de Gauss $\prod_{k=m}^{n}(k+i)$. Par suite, en prenant les conjugués (dans $\mathbb{C}$) des deux membres de \eqref{**}, on obtient que $L_{1,m,n}(n-m)!$ est aussi multiple (dans $\mathbb{Z}[i]$) de l'entier de Gauss $\prod_{k=m}^{n}(k-i)$. Il résulte de ces deux faits que $L_{1,m,n}(n-m)!$ est multiple (dans $\mathbb{Z}[i]$) de:
\begin{align*}
{\ppcm}_{\mathbb{Z}[i]}\left\lbrace \prod_{k=m}^{n}(k+i),\prod_{k=m}^{n}(k-i)\right\rbrace&=\frac{\prod_{k=m}^{n}(k+i)\cdot\prod_{k=m}^{n}(k-i)}{{\pgcd}_{\mathbb{Z}[i]}\left\lbrace \prod_{k=m}^{n}(k+i),\prod_{k=m}^{n}(k-i)\right\rbrace}\\&=\frac{\prod_{k=m}^{n}(k^2+1)}{{\pgcd}_{\mathbb{Z}[i]}\left\lbrace \prod_{k=m}^{n}(k+i),\prod_{k=m}^{n}(k-i)\right\rbrace}.
\end{align*}
Par conséquent, on a:
\begin{equation}\label{eq4}
L_{1,m,n}\geq \frac{\prod_{k=m}^{n}(k^2+1)}{(n-m)!\left|{\pgcd}_{\mathbb{Z}[i]}\left\lbrace \prod_{k=m}^{n}(k+i),\prod_{k=m}^{n}(k-i)\right\rbrace\right|}.
\end{equation}
Nous remarquons que la majoration triviale: 
\[\left|{\pgcd}_{\mathbb{Z}[i]}\left\lbrace \prod_{k=m}^{n}(k+i),\prod_{k=m}^{n}(k-i)\right\rbrace\right|\leq \left|\prod_{k=m}^{n}(k+i)\right|\leq \prod_{k=m}^{n}\sqrt{k^2+1}\]
suffit pour établir la minoration de Oon $L_{1,m,n}\geq \frac{\prod_{k=m}^{n}\sqrt{k^2+1}}{(n-m)!}$. Par conséquent, toute majoration non triviale pour le nombre $\left|{\pgcd}_{\mathbb{Z}[i]}\left\lbrace \prod_{k=m}^{n}(k+i),\prod_{k=m}^{n}(k-i)\right\rbrace\right|$ entraîne immédiatement une amélioration du théorème \ref{oon'} de Oon. D'autre part, pour $a,b\in\mathbb{Z}$ tels que $(a,b)\neq (0,0)$, on peut facilement vérifier que ${\pgcd}_{\mathbb{Z}[i]}\left(a+bi,a-bi\right)$ n'est pas trop loin de ${\pgcd}_{\mathbb{Z}}(a,b)$. Plus précisément, on a:
\[{\pgcd}_{\mathbb{Z}[i]}\left(a+bi,a-bi\right)=\left(\sigma+i\tau\right){\pgcd}_{\mathbb{Z}}(a,b),\]
où $\sigma,\tau\in\left\lbrace -1,0,1\right\rbrace$ et $(\sigma,\tau)\neq (0,0)$. Donc, pour le cas $c=1$, on est amené à étudier la fonction arithmétique:
\[\begin{array}{rcl}h :~\mathbb{Z}[i]\setminus\left\lbrace 0\right\rbrace &\longrightarrow &\mathbb{N^{*}}\\a+bi ~~& \longmapsto &\pgcd (a,b)\end{array};\]
plus précisément, à trouver une majoration non triviale pour la quantité $h\left(\prod_{k=m}^{n}(k+i)\right)$ $(m,n\in\mathbb{N^*},~m\leq n)$. Pour le cas général $(c\in\mathbb{N^*})$, la fonction arithmétique que nous devons étudier est clairement donnée par:
\[\begin{array}{rcl}h_c :~\mathbb{Z}[\sqrt{-c}]\setminus\left\lbrace 0\right\rbrace &\longrightarrow &\mathbb{N^{*}}\\a+b\sqrt{-c}~~& \longmapsto &\pgcd (a,b)\end{array}\]
et la quantité que nous devons majorer est $h_c\left(\prod_{k=m}^{n}(k+\sqrt{-c})\right)$ $(m,n\in\mathbb{N^*},$ $m \leq n)$.

La proposition suivante a pour objectif de remplacer un langage arithmétique spécifique de l'anneau $\mathbb{Z}[\sqrt{-c}]$ par son analogue (plus simple) dans $\mathbb{Z}$. 
   
\begin{prop}\label{propr}
Soient $c\in \mathbb{N^*}$ et $N,a,b \in \mathbb{Z}$, avec $(a,b)\neq (0,0)$. Alors, $N$ est multiple (dans $\mathbb{Z}[\sqrt{-c}]$) de $(a+b\sqrt{-c})$ si et seulement si $N$ est multiple (dans $\mathbb{Z}$) de $\frac{a^2+cb^{2}}{\pgcd(a,b)}$.  
\end{prop}
\begin{proof}
Le résultat de la proposition est trivial pour $b=0$. Supposons pour la suite que $b\neq 0$. On a deux implications à démontrer.\\
$\bullet(\Rightarrow):$ Supposons que $N$ est multiple (dans $\mathbb{Z}[\sqrt{-c}]$) de $(a+b\sqrt{-c})$. Il existe donc $x,y\in \mathbb{Z}$ tels que:
\[N=(x+y\sqrt{-c})(a+b\sqrt{-c}).\] 
En identifiant les parties réelles et imaginaires des deux côtés de cette égalité, on obtient:
\begin{align}
N&=ax-byc, \label{dor1}\\
0&=bx+ay. \label{dor2}
\end{align} 
Posons maintenant $d:=\pgcd (a,b)$. Il existe donc $a',b'\in\mathbb{Z}$, avec $b'\neq 0$ et $\pgcd(a',b')=1$, tels que $a=da'$ et $b=db'$. En substituant cela dans \eqref{dor2}, on obtient (après simplification):
\begin{equation}\label{dor3}
b'x=-a'y.
\end{equation}
Cette dernière égalité montre que $b'$ divise $a'y$. Puisque $\pgcd (a',b')=1$, alors (d'après le lemme de Gauss) $b'$ divise $y$. Il existe donc $k\in\mathbb{Z}$ tel que $y=kb'$. En remplaçant cela dans \eqref{dor3}, on obtient que $x=-ka'$. D'où, en substituant $x=-ka'=-k\frac{a}{d}$ et $y=kb'=k\frac{b}{d}$ dans \eqref{dor1}, on obtient enfin:  
\[N=-k\frac{a^2+cb^2}{d}=-k\frac{a^2+cb^2}{\pgcd(a,b)},\]
ce qui montre que $N$ est multiple (dans $\mathbb{Z}$) de $\frac{a^2+cb^2}{\pgcd(a,b)}$, comme il fallait le prouver.\\
$\bullet(\Leftarrow):$ Inversement, supposons que $N$ est multiple (dans $\mathbb{Z}$) de $\frac{a^2+cb^{2}}{\pgcd(a,b)}$. Il existe donc $k\in \mathbb{Z}$ tel que:
\[N=k\frac{a^2+cb^{2}}{\pgcd(a,b)}=k\frac{a-b\sqrt{-c}}{\pgcd(a,b)}(a+b\sqrt{-c})=\left(\frac{ka}{\pgcd(a,b)}-\frac{kb}{\pgcd(a,b)}\sqrt{-c}\right)(a+b\sqrt{-c}).\]
Puisque $\left(\frac{ka}{\pgcd(a,b)}-\frac{kb}{\pgcd(a,b)}\sqrt{-c}\right)\in \mathbb{Z}[\sqrt{-c}]$, cette dernière égalité montre que $N$ est multiple (dans $\mathbb{Z}[\sqrt{-c}]$) de $\left(a+b\sqrt{-c}\right)$, comme il fallait le prouver.\\ Ce qui complète la démonstration de la proposition.
\end{proof}

De la proposition \ref{propr}, nous tirons le corollaire suivant, qui est la première étape clé pour obtenir les résultats de ce chapitre.

\begin{coll}\label{p2}
Soient $c,m,n\in\mathbb{N^*}$ tels que $m\leq n$. Alors, le nombre $L_{c,m,n}(n-m)!$ est multiple (dans $\mathbb{Z}$) de l'entier strictement positif:
\[\frac{\prod_{k=m}^{n}\left(k^2+c\right)}{h_c\left(\prod_{k=m}^{n}\left(k+\sqrt{-c}\right)\right)}.\] 
\end{coll}
\begin{proof}
La formule \eqref{**} (obtenue lors de notre nouvelle preuve du théorème \ref{oon'}) montre que $L_{c,m,n}(n-m)!$ est multiple (dans $\mathbb{Z}[\sqrt{-c}]$) de $\prod_{k=m}^{n}\left(k+\sqrt{-c}\right)$. Selon la proposition \ref{propr}, cette dernière propriété est équivalente à l'énoncé du corollaire. 
\end{proof}

En vertu du corollaire \ref{p2}, pour minorer le nombre $L_{c,m,n}$ $(c,m,n\in\mathbb{N^*},~m\leq n)$, il suffit de majorer la quantité $h_c\left(\prod_{k=m}^{n}(k+\sqrt{-c})\right)$. De même, pour trouver un diviseur (rationnel) non trivial de $L_{c,m,n}$, il suffit de trouver un multiple non trivial pour $h_c\left(\prod_{k=m}^{n}(k+\sqrt{-c})\right)$. C'est ce que nous allons faire dans la suite.
 
\section{Une identité de Bézout explicite}\label{sub22}

\noindent Dans toute la suite, fixons $c\in\mathbb{N^*}$ et $k\in\mathbb{N}$ et Posons:
\begin{align*}
P_{k}(X)&:=(X+\sqrt{-c})(X-1+\sqrt{-c})\cdots (X-k+\sqrt{-c})=A_{k}(X)+B_{k}(X)\sqrt{-c}, \\ \overline{P_{k}}(X)&:=(X-\sqrt{-c})(X-1-\sqrt{-c})\cdots (X-k-\sqrt{-c})=A_{k}(X)-B_{k}(X)\sqrt{-c},
\end{align*}
où l'on sous-entend que $A_{k},B_{k}\in \mathbb{Z}[X]$. Dans ce qui suit, nous trouvons un multiple non trivial pour l'entier strictement positif $h_c\left(P_{k}(n)\right)=\pgcd\left(A_{k}(n),B_{k}(n)\right)$. Pour ce faire, nous allons plutôt chercher deux suites polynomiales $\left(a_{k}(n)\right)_n$ et $\left(b_{k}(n)\right)_n$ de sorte que la suite polynomiale $\left(a_{k}(n)A_{k}(n)+b_{k}(n)B_{k}(n)\right)_n$ soit indépendante de $n$. Il est clair que cela conduit à la recherche de deux polynômes $U_{k},V_{k}\in \mathbb{Q}[X]$ qui satisfont l'identité de Bézout:
\[U_{k}(X)A_{k}(X)+V_{k}(X)B_{k}(X)=1.\]
Par suite, comme $A_{k}=\frac{P_{k}+\overline{P_{k}}}{2}$ et $B_{k}=\frac{P_{k}-\overline{P_{k}}}{2\sqrt{-c}}$, cela revient donc à chercher $\sigma_{k},\tau_{k}\in \mathbb{Q}\left(\sqrt{-c}\right)[X]$ tels que:
\[\sigma_{k}\left(X\right)P_{k}\left(X\right)+\tau_{k}\left(X\right)\overline{P_{k}}\left(X\right)=1.\]
Justifions d'abord l'existence de $\sigma_k$ et $\tau_k$. En désignant par $Z(P)$ l'ensemble de toutes les racines complexes d'un polynôme $P\in\mathbb{C}[X]$, on a clairement:
\[Z\left(P_{k}\right)=\lbrace -\sqrt{-c},1-\sqrt{-c},\dots,k-\sqrt{-c}\rbrace~\text{et}~Z\left(\overline{P_{k}}\right)=\lbrace \sqrt{-c},1+\sqrt{-c},\dots,k+\sqrt{-c}\rbrace,\]
donc $Z\left(P_{k}\right)\cap Z\left(\overline{P_{k}}\right)=\emptyset$; c'est-à-dire que $P_{k}$ et $\overline{P_{k}}$ n'ont pas de racines communes dans $\mathbb{C}$. Cela implique que $P_{k}$ et $\overline{P_{k}}$ sont premiers entre eux dans $\mathbb{C}[X]$; donc ils sont aussi premiers entre eux dans $\mathbb{Q}\left(\sqrt{-c}\right)[X]$. Il s'ensuit (en vertu du théorème de Bézout) qu'il existe $\sigma_{k},\tau_{k}\in \mathbb{Q}\left(\sqrt{-c}\right)[X]$ tels que: $\sigma_{k}P_{k}+\tau_{k}\overline{P_{k}}=1$, comme il fallait le prouver. 

Maintenant, pour trouver explicitement de tels $\sigma_k$ et $\tau_k$, nous avons besoin de la version plus précise du théorème de Bézout suivante:

\begin{thm}\label{p3}
Soient $\mathbb{K}$ un corps et $P$ et $Q$ deux polynômes non constants de $\mathbb{K}[X]$ tels que ${\pgcd}_{\mathbb{K}[X]}\left(P,Q\right)=1$. Alors, il existe un couple unique $\left(U,V\right)$ de polynômes de $\mathbb{K}[X]$, avec $\deg U<\deg Q$ et $\deg V<\deg P$, tel que:
\[PU+QV=1.\]
\end{thm}
\begin{proof}
Puisque ${\pgcd}_{\mathbb{K}[X]}\left(P,Q\right)=1$, alors (d'après le théorème de Bézout) il existe $U_0,V_0\in\mathbb{K}[X]$ tels que:
\[PU_0+QV_0=1.\]
Considérons la division euclidienne de $U_0$ par $Q$ et la division euclidienne de $V_0$ par $\left(-P\right)$ dans $\mathbb{K}[X]$:
\begin{align*}
U_0 &= U_1Q+U , \\
V_0 &= V_1\left(-P\right)+V,
\end{align*}
où $U_1,V_1,U,V\in\mathbb{K}[X]$, $\deg U<\deg Q$ et $\deg V<\deg \left(-P\right)=\deg P$. Donc, on a:
\[PU+QV = P\left(U_0-U_1Q\right)+Q\left(V_0+V_1P\right) = PQ\left(V_1-U_1\right)+PU_0+QV_0 = PQ\left(V_1-U_1\right)+1.\]
Si $V_1-U_1\neq 0$, alors la dernière égalité implique que $\deg\left(PU+QV\right)\geq \deg\left(PQ\right)$, ce qui est impossible, car $\deg U<\deg Q$ et $\deg V<\deg P$. D'où $V_1-U_1=0$, ce qui donne $PU+QV=1$. L'existence du couple $\left(U,V\right)$ requis par le théorème est prouvée. Il reste à prouver l'unicité de $\left(U,V\right)$. Soit $\left(U_*,V_*\right)$ un autre couple de polynômes de $\mathbb{K}[X]$, avec $\deg U_* <\deg Q$, $\deg V_* <\deg P$ et $PU_*+QV_*=1$ et montrons que $\left(U_*,V_*\right)=\left(U,V\right)$. On a:
\begin{align*}
P\left(UV_*-U_*V\right)=\left(PU\right)V_*-\left(PU_*\right)V=\left(1-QV\right)V_*-\left(1-QV_*\right)V=V_*-V,
\end{align*}
ce qui montre que le polynôme $\left(V_*-V\right)$ est multiple de $P$ dans $\mathbb{K}[X]$. Comme \linebreak $\deg\left(V_*-V\right)<\deg P$ (car: $\deg V<\deg P$ et $\deg V_*<\deg P$), on a nécessairement $V_*-V=0$; donc $V_*=V$. Il s'ensuit de cela que $PU_*=1-QV_*=1-QV=PU$. D'où $U_*=U$. On a par conséquent $\left(U_*,V_*\right)=\left(U,V\right)$, comme il fallait le prouver.\\ Ce qui complète la démonstration du théorème. 
\end{proof}

Dans notre contexte, l'application du théorème \ref{p3} donne le corollaire suivant:

\begin{coll}\label{ly}
Il existe un unique polynôme $\sigma_{k}\in\mathbb{C}[X]$, de degré $\leq k$, tel que: 
\[\sigma_{k}P_{k}+\overline{\sigma_{k}}\overline{P_{k}}=1.\]
\end{coll}
\begin{proof}
D'après le théorème \ref{p3} (appliqué à $\mathbb{K}=\mathbb{C}$ et $\left(P,Q\right)=\left(P_k,\overline{P_k}\right)$), il existe un unique couple $\left(\sigma_k,\tau_k\right)$ de polynômes de $\mathbb{C}[X]$, avec $\deg \sigma_k<\deg \overline{P_k}=k+1$ et $\deg \tau_k<\deg P_k=k+1$, tel que $\sigma_{k}P_{k}+\tau_{k}\overline{P_{k}}=1$. En prenant les conjugués dans $\mathbb{C}[X]$ des deux côtés de cette dernière égalité, on obtient: $\overline{\sigma_{k}}\overline{P_{k}}+\overline{\tau_{k}}P_{k}=1$, c'est-à-dire que $\overline{\tau_{k}}P_{k}+\overline{\sigma_{k}}\overline{P_{k}}=1$. Comme $\deg \overline{\tau_k}=\deg \tau_k<k+1$ et $\deg \overline{\sigma_k}=\deg \sigma_k<k+1$, cela montre que le couple $\left(\overline{\tau_k},\overline{\sigma_k}\right)$ satisfait à la propriété caractéristique du couple $\left(\sigma_k,\tau_k\right)$. D'où $\left(\overline{\tau_k},\overline{\sigma_k}\right)=\left(\sigma_k,\tau_k\right)$, ce qui revient à dire que $\tau_k=\overline{\sigma_k}$. On a par conséquent $\sigma_{k}P_{k}+\overline{\sigma_{k}}\overline{P_{k}}=1$. Ce qui complète la démonstration du corollaire.   
\end{proof}

Maintenant, nous allons déterminer l'expression explicite du polynôme $\sigma_k$ annoncé dans le corollaire \ref{ly}. En remplaçant dans l'identité $\sigma_{k}\left(X\right)P_{k}\left(X\right)+\overline{\sigma_{k}}\left(X\right)\overline{P_{k}}\left(X\right)=1$, l'indéterminée $X$ par les nombres $s+\sqrt{-c}$ $(s=0,1,\dots,k)$, on obtient:  
\begin{equation}\label{dec1}
\sigma_k\left(s+\sqrt{-c}\right)=\frac{1}{P_k\left(s+\sqrt{-c}\right)}~~(\forall s\in\left\lbrace 0,1,\dots,k\right\rbrace).
\end{equation}
(car: $\overline{P_{k}}\left(s+\sqrt{-c}\right)=0$ pour $s=0,1,\dots,k$). Ainsi, les valeurs de $\sigma_{k}$ sont connues en $(k+1)$ points équidistants de distance $1$. Comme $\deg\sigma_k\leq k$, cela suffit pour déterminer l'expression de $\sigma_{k}\left(X\right)$ en utilisant par exemple la formule d'interpolation de Newton. En procédant ainsi, on obtient: 
\[\sigma_k\left(X\right)=\sum_{\ell=0}^{k}\frac{\left(\Delta^{\ell}\sigma_k\right)\left(\sqrt{-c}\right)}{\ell!}\left(X-\sqrt{-c}\right)^{\underline{\ell}}.\]
Puis, en utilisant \eqref{dece}, on obtient:
\begin{align*}
\sigma_{k}\left(X\right)&=\sum_{\ell=0}^{k}\sum_{j=0}^{\ell}\frac{(-1)^{\ell-j}}{\ell!}\binom{\ell}{j}\sigma_{k}\left(j+\sqrt{-c}\right)\left(X-\sqrt{-c}\right)^{\underline{\ell}}\\
&=\sum_{\ell=0}^{k}\left\lbrace\frac{1}{\ell!}\sum_{j=0}^{\ell}(-1)^{\ell-j}\binom{\ell}{j}\sigma_{k}\left(j+\sqrt{-c}\right)\right\rbrace\left(X-\sqrt{-c}\right)^{\underline{\ell}}\\
&=\sum_{\ell=0}^{k}\left\lbrace\frac{1}{\ell!}\sum_{j=0}^{\ell}(-1)^{\ell-j}\binom{\ell}{j}\frac{1}{P_{k}\left(j+\sqrt{-c}\right)}\right\rbrace\left(X-\sqrt{-c}\right)^{\underline{\ell}}
\end{align*}  
(en vertu de \eqref{dec1}). Donc, en posant pour tout $\ell \in\left\lbrace 0,1,\dots,k\right\rbrace$:
\begin{equation}\label{dec2}
\Theta_{k,\ell}:=\frac{1}{\ell!}\sum_{j=0}^{\ell}(-1)^{\ell-j}\binom{\ell}{j}\frac{1}{P_{k}(j+\sqrt{-c})},
\end{equation}
il vient que:
\begin{equation}\label{dec3}
\sigma_{k}\left(X\right)=\sum_{\ell=0}^{k}\Theta_{k,\ell}\left(X-\sqrt{-c}\right)^{\underline{\ell}}.
\end{equation}
Il reste à simplifier les expressions des nombres $\Theta_{k,\ell}$ $(0\leq\ell\leq k)$. Pour ce faire, nous introduisons les fonctions rationnelles $R_{k,\ell}$ $(0\leq\ell\leq k)$, définies par:
\begin{equation}\label{dec4}
R_{k,\ell}(z):=\frac{1}{\ell!}\sum_{j=0}^{\ell}(-1)^{\ell-j}\binom{\ell}{j}\frac{1}{P_{k}(z+j+\sqrt{-c})},
\end{equation}
de sorte que l'on ait:
\begin{equation}\label{dec5}
\Theta_{k,\ell}=R_{k,\ell}(0)~~~~(\forall \ell\in\left\lbrace 0,1,\dots,k\right\rbrace).
\end{equation}
Le domaine commun d'holomorphie des fonctions $R_{k,\ell}$ $(0\leq\ell\leq k)$ est visiblement la région connexe et ouverte $D$ de $\mathbb{C}$, donnée par:
\[D:=\mathbb{C}\setminus\lbrace j - 2 \sqrt{-c}~;~j\in\mathbb{Z}~\text{et} -k\leq j\leq k\rbrace.\]
En utilisant le principe du prolongement analytique (voir \cite{rud}) ainsi que la théorie des fonctions gamma et bêta (que l'on peut trouver dans \cite{artin}), nous pouvons trouver une autre expression de $R_{k,\ell}$ $(0\leq\ell\leq k)$, qui est plus simple que celle de ci-dessus. On a la proposition suivante:

\begin{prop}\label{pp}
Pour tout $\ell\in\mathbb{N}$, avec $\ell\leq k$, et tout $z\in D$, on a:
\begin{equation}\label{dec6}
R_{k,\ell}(z)=\frac{(-1)^{k+\ell}}{z+2\sqrt{-c}}\binom{k+\ell}{\ell}\frac{1}{\left(k-2\sqrt{-c}-z\right)^{\underline{k}}\left(\ell+2\sqrt{-c}+z\right)^{\underline{\ell}}}.
\end{equation}
\end{prop}
\begin{proof}
Soit $\ell \in \mathbb{N}$ tel que $\ell \leq k$. D'après le principe du prolongement analytique, il suffit de prouver la formule \eqref{dec6} pour $z\in\mathbb{C}$, tel que $\Re(z)>k$. Pour un tel $z$, on a:
\begin{align*}
R_{k,\ell}(z)&:=\frac{1}{\ell!}\sum_{j=0}^{\ell}(-1)^{\ell-j}\binom{\ell}{j}\frac{1}{P_{k}(z+j+\sqrt{-c})}\\
&=\frac{1}{\ell!}\sum_{j=0}^{\ell}(-1)^{\ell-j}\binom{\ell}{j}\frac{1}{\left(z+j+2\sqrt{-c}\right)\left(z+j-1+2\sqrt{-c}\right)\cdots\left(z+j-k+2\sqrt{-c}\right)}\\
&=\frac{1}{\ell!}\sum_{j=0}^{\ell}(-1)^{\ell-j}\binom{\ell}{j}\frac{\Gamma\left(z+j-k+2\sqrt{-c}\right)}{\Gamma\left(z+j+1+2\sqrt{-c}\right)}\\
&=\frac{1}{\ell!}\sum_{j=0}^{\ell}(-1)^{\ell-j}\binom{\ell}{j}\frac{1}{k!}\beta\left(z+j-k+2\sqrt{-c},k+1\right)\\
&=\frac{1}{k!\ell!}\sum_{j=0}^{\ell}\left[(-1)^{\ell-j}\binom{\ell}{j}\int_{0}^{1}t^{z+j-k-1+2\sqrt{-c}}(1-t)^{k}\mathrm{d}t\right] \\
&=\frac{1}{k!\ell!}\int_{0}^{1} t^{z-k-1+2\sqrt{-c}}(1-t)^{k}\left\lbrace \sum_{j=0}^{\ell}(-1)^{\ell-j}\binom{\ell}{j}t^{j}\right\rbrace \mathrm{d}t \\
&=\frac{1}{k!\ell!}\int_{0}^{1} t^{z-k-1+2\sqrt{-c}}(1-t)^{k}\left(t-1\right)^{\ell} \mathrm{d}t \\
&=\frac{(-1)^{\ell}}{k!\ell!}\int_{0}^{1} t^{z-k-1+2\sqrt{-c}}(1-t)^{k+\ell}\mathrm{d}t \\
&=\frac{(-1)^\ell}{k!\ell!}\beta\left(z-k+2\sqrt{-c},k+\ell+1\right)\\
&=\frac{(-1)^\ell}{k!\ell!}\frac{\Gamma\left(z-k+2\sqrt{-c}\right)\Gamma\left(k+\ell+1\right)}{\Gamma\left(z+\ell+1+2\sqrt{-c}\right)}\\
&=(-1)^\ell\binom{k+\ell}{\ell}\frac{1}{\left(z+\ell+2\sqrt{-c}\right)\left(z+\ell-1+2\sqrt{-c}\right)\cdots\left(z-k+2\sqrt{-c}\right)}\\
&=\frac{(-1)^{k+\ell}}{z+2\sqrt{-c}}\binom{k+\ell}{\ell}\frac{1}{\left(k-2\sqrt{-c}-z\right)^{\underline{k}}\left(\ell+2\sqrt{-c}+z\right)^{\underline{\ell}}},
\end{align*}
comme il fallait le prouver. Ce qui complète cette démonstration.
\end{proof}

De la proposition \ref{pp}, découle immédiatement une expression explicite plus simple de $\sigma_k\left(X\right)$. On a le corollaire suivant:

\begin{coll}\label{jan1}
On a:
\[\sigma_{k}\left(X\right)=\frac{1}{2\sqrt{-c}\left(k-2\sqrt{-c}\right)^{\underline{k}}}\sum_{\ell=0}^{k}\frac{(-1)^{k+\ell}\binom{k+\ell}{\ell}}{\left(\ell+2\sqrt{-c}\right)^{\underline{\ell}}}\left(X-\sqrt{-c}\right)^{\underline{\ell}}.\]
\end{coll}
\begin{proof}
Cela découle immédiatement des formules \eqref{dec3}, \eqref{dec5} et \eqref{dec6}. 
\end{proof}

\section{Multiples non triviaux de certaines valeurs de $h_c$}

Dans cette section, nous conservons les notations de la section \textsection\ref{sub22}. Du corollaire \ref{jan1}, nous déduisons le théorème suivant:

\begin{thm}\label{jan2}
Pour tous $c,n,m\in\mathbb{N^*}$, avec $m\leq n$, on a:
\[h_c\left(\prod_{\ell=m}^{n}\left(\ell+\sqrt{-c}\right)\right)~~\text{divise}~~c\prod_{\ell=1}^{n-m}(\ell^2+4c).\] 
\end{thm} 
\begin{proof}
Soient $c,n,m\in\mathbb{N^*}$, avec $m\leq n$ et posons $k:=n-m\in\mathbb{N}$ et $d:=c\prod_{\ell=1}^{n-m}(\ell^2+4c)\in\mathbb{N^*}$. On a $\prod_{\ell=m}^{n}\left(\ell+\sqrt{-c}\right)=P_k(n)$; nous devons donc démontrer que $h_c\left(P_k(n)\right)$ divise $d$. En constatant que $2d=\sqrt{-c}\cdot 2\sqrt{-c}\left(k-2\sqrt{-c}\right)^{\underline{k}}\left(k+2\sqrt{-c}\right)^{\underline{k}}$, on obtient (en vertu du corollaire \ref{jan1}) que $2d\sigma_k\in\mathbb{Z}[\sqrt{-c}][X]$. Donc, il existe $r_k,s_k\in\mathbb{Z}[X]$ tels que:
\[2d\sigma_k\left(X\right)=r_k\left(X\right)+s_k\left(X\right)\sqrt{-c}.\] 
Par suite, l'identité polynomiale $\sigma_kP_k+\overline{\sigma_k}\overline{P_k}=1$ (donnée par le corollaire \ref{ly}) implique que $2d\sigma_k\cdot P_k+\overline{2d\sigma_k}\cdot \overline{P_k}=2d$. En substituant dans cette dernière égalité $P_k$ par $\left(A_k+B_k\sqrt{-c}\right)$ et $2d\sigma_k$ par $\left(r_k+s_k\sqrt{-c}\right)$, on obtient (en particulier) que:
\[r_kA_k-cs_kB_k=d,\] 
ce qui implique que ${\pgcd}_{\mathbb{Z}[X]}\left(A_k,B_k\right)$ divise $d$. On en déduit alors que $h_c\left(P_k(n)\right)={\pgcd}_{\mathbb{Z}}\left(A_k(n),B_k(n)\right)$ divise $d$, comme il fallait le prouver.
\end{proof}

\section{Nouvelles estimations pour le nombre $L_{c,m,n}$}

On a le théorème suivant:
\begin{thm}\label{t7}
Soient $c,m,n\in\mathbb{N^*}$ tels que $m\leq n$. Alors:
\begin{enumerate}
\item L'entier strictement positif $L_{c,m,n}$ est multiple du nombre rationnel:
\[\frac{\displaystyle\prod_{k=m}^{n}\left(k^2+c\right)}{c\cdot (n-m)!\displaystyle\prod_{k=1}^{n-m}\left(k^2+4c\right)}.\]\label{p1t7}
\item On a:
\[L_{c,m,n}\geq \lambda_1(c) \cdot m^2\frac{n!^2}{m!^2(n-m)!^3},\]
où $\lambda_1(c):=e^{-\frac{2\pi^{2}}{3} c}/c$.
\end{enumerate}
\end{thm}
\begin{proof}
Le premier point du théorème est une conséquence immédiate du corollaire \ref{p2} et du théorème \ref{jan2}. Par suite, en utilisant l'inégalité $1+x\leq e^x$ $(\forall x\in\mathbb{R})$, on a:
\begin{align*}
\prod_{k=1}^{n-m}\left(k^2+4c\right)&=\prod_{k=1}^{n-m}k^2\left(1+\frac{4c}{k^2}\right)\\&={(n-m)!}^2\prod_{k=1}^{n-m}\left(1+\frac{4c}{k^2}\right)\\&\leq {(n-m)!}^2\prod_{k=1}^{n-m}e^{\frac{4c}{k^2}}\\&\leq {(n-m)!}^2\prod_{k=1}^{+\infty}e^{\frac{4c}{k^2}}\\&={(n-m)!}^2e^{\sum_{k=1}^{+\infty}\frac{4c}{k^2}}\\&={(n-m)!}^2e^{\frac{2\pi^2}{3}c}.
\end{align*}
D'où l'on a:
\begin{align*}
\frac{\prod_{k=m}^{n}\left(k^2+c\right)}{c\cdot (n-m)!\prod_{k=1}^{n-m}\left(k^2+4c\right)}&\geq \frac{\prod_{k=m}^{n}k^2}{c\cdot (n-m)!\cdot{(n-m)!}^2e^{\frac{2\pi^2}{3}c}}\\&= \frac{m^2\left(\frac{n!}{m!}\right)^2}{c\cdot {(n-m)!}^3e^{\frac{2\pi^2}{3}c}}\\&=\frac{e^{-\frac{2\pi^{2}}{3} c}}{c}\cdot m^2\frac{{n!}^2}{{m!}^2{(n-m)!}^3}.
\end{align*}
Le second point du théorème découle alors du premier et de cette dernière minoration. Le théorème est démontré.
\end{proof}

Nous allons maintenant imposer des conditions sur $m$ (en fonction de $n$) afin d'optimiser (resp. de simplifier) l'estimation du second point du théorème \ref{t7}. Pour ce faire, nous devons d'abord nous débarrasser des factoriels figurant dans cette estimation. On a le corollaire suivant:

\begin{coll}\label{t9}
Soient $c,n,m\in\mathbb{N^*}$ tels que $m<n$. Alors, on a:
\begin{equation}\label{jan3}
L_{c,m,n}\geq \lambda_{2}(c)\cdot\frac{nm}{\left(n-m\right)^{3/2}}\left( \frac{m^2}{(n-m)^3}\right)^{n-m} e^{3(n-m)},
\end{equation}
où $\lambda_2(c):=\frac{e^{-\frac{2\pi^{2}}{3} c -\frac{5}{12}}}{\left(2\pi\right)^{3/2}c}$.
\end{coll}
\begin{proof}
En partant de la minoration du second point du théorème \ref{t7} pour $L_{c,m,n}$ et en estimant chacun des factoriels qui y figurent en se servant de la double inégalité bien connue:
\[k^{k}e^{-k}\sqrt{2\pi k}\leq k!\leq k^{k}e^{-k}\sqrt{2\pi k} e^{\frac{1}{12k}}~~~~(\forall k\in\mathbb{N^*})\]
(que l'on peut trouver dans \cite[Problem 1.15]{kon}), on obtient:
\[L_{c,m,n} \geq \lambda_1(c) (2 \pi)^{- 3/2} \cdot \frac{n m}{(n - m)^{3 / 2}} \cdot \left(\frac{n}{m}\right)^{2 n} \cdot \left(\frac{m^2}{(n - m)^3}\right)^{n - m} e^{n - m} \cdot e^{- \frac{1}{6 m} - \frac{1}{4 (n - m)}}.\]
Par suite, comme $e^{- \frac{1}{6 m} - \frac{1}{4 (n - m)}} \geq e^{- \frac{1}{6} - \frac{1}{4}} = e^{- \frac{5}{12}}$ et $\left(\frac{n}{m}\right)^{2 n} = e^{- 2 n \log(\frac{m}{n})} \geq e^{- 2 n (\frac{m}{n} - 1)} = e^{2 (n - m)}$, on en déduit que:
\[L_{c , m , n} \geq \lambda_1(c) (2 \pi)^{- 3/2} e^{- 5 / 12} \cdot \frac{n m}{(n - m)^{3 / 2}} \left(\frac{m^2}{(n - m)^3}\right)^{n - m} e^{3 (n - m)} ,\]
comme il fallait le prouver.
\end{proof}

Dans le contexte du corollaire \ref{t9}, en supposant que $(n-m)$ est d'un ordre de grandeur $n^\alpha$ pour $n$ assez grand (où $0<\alpha<1$), alors la partie dominante de la minoration \eqref{jan3} de $L_{c,m,n}$ est $\left(\frac{m^2}{(n-m)^3}\right)^{n-m}$, qui est d'ordre de grandeur $n^{(2-3\alpha)n^{\alpha}}$. Ainsi, pour avoir une estimation optimale, nous devons prendre $\alpha$ inférieurement proche de $\frac{2}{3}$ (une étude de la fonction $\alpha\longmapsto (2-3\alpha)n^{\alpha}$ montre que la meilleure valeur de $\alpha$ est $\alpha=\frac{2}{3}-\frac{1}{\log n}$). Un résultat concret spécifiant ce raisonnement heuristique est donné par le théorème suivant:
 
\begin{thm}\label{c5}
Soient $c,m,n\in\mathbb{N^*}$ tels que $m\leq n - \frac{1}{2}n^{2/3}$. Alors, on a:
\[L_{c,m,n}\geq \lambda_3(c)\cdot\left(n-\frac{1}{2}n^{2/3}\right)\cdot\left(2e^{3}\right)^{\left\lfloor \frac{1}{2}n^{2/3}\right\rfloor},\]
où $\lambda_3(c):=\frac{e^{-\frac{2\pi^2}{3} c -\frac{5}{12}}}{\pi^{3/2} c}$.
\end{thm}
\begin{proof}
Un simple calcul montre que le résultat du théorème est vrai pour $n<3$. Supposons pour la suite que $n\geq 3$ et posons $m_n:=n-\left\lfloor \frac{1}{2}n^{2/3}\right\rfloor < n$; donc $m \leq m_n$. D'après le corollaire \ref{t9}, on a:
\begin{align*}
L_{c,m_n,n}&\geq \lambda_2(c)\frac{n\left(n-\left\lfloor \frac{1}{2}n^{2/3}\right\rfloor\right)}{\left\lfloor \frac{1}{2}n^{2/3}\right\rfloor^{3/2}}\left( \frac{\left(n-\left\lfloor \frac{1}{2}n^{2/3}\right\rfloor\right)^2}{\left\lfloor \frac{1}{2}n^{2/3}\right\rfloor^3}\right)^{\left\lfloor \frac{1}{2}n^{2/3}\right\rfloor} e^{3\left\lfloor \frac{1}{2}n^{2/3}\right\rfloor}\\&\geq \lambda_2(c)\frac{n\left(n-\frac{1}{2}n^{2/3}\right)}{\left(\frac{1}{2}n^{2/3}\right)^{3/2}}\left(\frac{\left(n-\frac{1}{2}n^{2/3}\right)^2}{\left(\frac{1}{2}n^{2/3}\right)^3}\right)^{\left\lfloor \frac{1}{2}n^{2/3}\right\rfloor}e^{3\left\lfloor \frac{1}{2}n^{2/3}\right\rfloor}\\&=2^{3/2}\lambda_2(c)\left(n-\frac{1}{2}n^{2/3}\right)\left[8\left(1-\frac{1}{2n^{1/3}}\right)^2\right]^{\left\lfloor \frac{1}{2}n^{2/3}\right\rfloor}e^{3\left\lfloor \frac{1}{2}n^{2/3}\right\rfloor}.
\end{align*}
Comme $1-\frac{1}{2n^{1/3}}\geq \frac{1}{2}$ (car: $n\geq 1$), on en déduit que:
\[L_{c,m_n,n}\geq 2^{3/2}\lambda_2(c)\left(n-\frac{1}{2}n^{2/3}\right)\left(2e^3\right)^{\left\lfloor \frac{1}{2}n^{2/3}\right\rfloor}.\] 
Le résultat requis découle du fait que $L_{c,m,n}\geq L_{c,m_n,n}$ (car: $m\leq m_n$).
\end{proof}

Par une autre façon, nous tirons du corollaire \ref{t9} le théorème suivant, qui complète (d'une certaine manière) le théorème \ref{c5} ci-dessus.

\begin{thm}\label{c6++}
Soient $c,m,n\in\mathbb{N^*}$ tels que $n-\frac{1}{2}n^{2/3} \leq m \leq n$. Alors, on a:
\[L_{c,m,n}\geq \lambda_2(c)\cdot ne^{3(n-m)},\]
où $\lambda_2(c)$ est déjà défini dans le corollaire \ref{t9}. 
\end{thm}
\begin{proof}
Le résultat du théorème est trivial pour $m=n$. Supposons pour la suite que $m<n$; donc $n \geq 2$. Maintenant, soit $f:\left[0,n\right]\longrightarrow \mathbb{R}$ la fonction définie par $f(x)=x^2-(n-x)^3$ $(\forall x\in [0,n])$. Il est clair que $f$ est strictement croissante. Par suite, on a:
\begin{align*}
f\left(n-\frac{1}{2}n^{2/3}\right)&=\left(n-\frac{1}{2}n^{2/3}\right)^2-\left(\frac{1}{2}n^{2/3}\right)^3\\&=n^2-n^{5/3}+\frac{1}{4}n^{4/3}-\frac{1}{8}n^2\\&=\frac{7}{8}n^2-n^{5/3}+\frac{1}{4}n^{4/3}.
\end{align*}
Puisque $n^2\geq \frac{8}{7}n^{5/3}$ (car $n\geq 2$), il s'ensuit que $f\left(n-\frac{1}{2}n^{2/3}\right)\geq \frac{1}{4}n^{4/3}>0$. Par suite, la croissance de $f$ assure que $f(m)>0$ (car $m\geq n-\frac{1}{2}n^{2/3}$ par hypothèse). D'où $\frac{m^2}{(n-m)^3}>1$ et $\frac{m}{(n-m)^{3/2}}>1$. En combinant cela avec \eqref{jan3}, on en déduit que:
\[L_{c,m,n}\geq \lambda_2(c)\cdot ne^{3(n-m)},\]
comme il fallait le prouver. Ce qui complète cette démonstration.  
\end{proof}

\subsection{Comparaison avec la minoration de Oon}

Dans la minoration de Oon (c'est-à-dire le théorème \ref{oon'}), le nombre de termes figurant dans le plus petit commun multiple 
\[L_{c,m,n}=\ppcm(m^2+c,(m+1)^2+c,\dots,n^2+c)\]
est strictement plus grand que $n/2$; alors que lorsque nous mettons ensemble nos théorèmes \ref{c5} et \ref{c6++}, cette contrainte est éliminée. Cependant, si la condition d'application du théorème de Oon est remplie, nous obtenons alors une minoration pour $L_{c ,m,n}$ plus forte que celles de nos théorèmes. Ceci dit, notre résultat clé est plutôt le point \ref{p1t7} du théorème \ref{t7} qui fournit un diviseur rationnel et non trivial de $L_{c,m,n}$. Ici, nous avons exploité ce résultat clé de manière naïve. Il est probable qu'une ``procédure plus intelligente" donnerait de meilleurs résultats.

\bigskip

\begin{center}
\includegraphics[scale=1]{fini}
\end{center}

\chapter{Identités et estimations concernant le $\ppcm$ de suites à forte divisibilité}\label{ch3} 

Il est à noter que les résultats de ce chapitre sont publiés dans \cite{Bousfarcras}.

\section{Introduction}

L'étude des propriétés arithmétiques des coefficients binomiaux est un sujet tr\`es ancien et fascinant. \`A titre d'exemple, il y a plus d'un siècle que Sylvester \cite{syl} prouvait que pour tous $n,k\in\mathbb{N^*}$, tels que $n\geq 2k$, le coefficient binomial $\binom{n}{k}$ possède au moins un diviseur premier strictement supérieur à $k$. Assez récemment, Farhi \cite{Farhi} a montré l'identité $\ppcm\left\lbrace \binom{n}{0},\binom{n}{1},\dots,\binom{n}{n}\right\rbrace=\frac{\ppcm\left(1,2,\dots,n,n+1\right)}{n+1}$ $(\forall n\in\mathbb{N})$, que Guo \cite{Victor} a généralisé aux coefficients $q$-binomiaux. Dans ce chapitre, nous présentons d'abord une démonstration de cette identité; ensuite nous démontrons une identité plus générale relative aux suites à forte divisibilité. Cette identité générale englobe à la fois les identités de Farhi et Guo qui en deviennent des cas particuliers (voir le théorème \ref{Ri1} et la remarque \ref{rmq1}). Nous en déduisons par suite deux autres identités également intéressantes (voir les corollaires \ref{sour} et \ref{sour1}). Comme application, nous utilisons nos identités pour établir des estimations effectives et non triviales du $\ppcm$ des termes consécutifs de certaines suites de Lucas (voir le théorème \ref{aR4}). L'efficacité de nos estimations effectives est garantie par les estimations asymptotiques obtenues par Matiyasevich et Guy \cite{Mat} et Kiss et Matyas \cite{kiss} dans le même contexte.

\section{Une identité concernant le $\ppcm$ des coefficients binomiaux usuels}

Le théorème suivant est le résultat principal de l'article \cite{Farhi}.

\begin{thm}[Farhi \cite{Farhi}]\label{Fip}
Pour tout entier naturel $n$, on a:
\begin{equation}\label{idf}
\ppcm\left\lbrace \binom{n}{0},\binom{n}{1},\dots,\binom{n}{n} \right\rbrace =\frac{\ppcm\left(1,2,\dots,n,n+1\right)}{n+1}.
\end{equation}
\end{thm}
\noindent La preuve présentée dans \cite{Farhi} utilise le théorème de Kummer qui donne la valuation $p$-adique des coefficients binomiaux. Plus précisément, on a:

\begin{thm}[Kummer]\label{kum}
Soient $k,n\in\mathbb{N}$ et $p$ un nombre premier. Alors, la valuation $p$-adique du nombre $\binom{n}{k}$ est égale au nombre d'emprunts effectués lorsque l'on soustrait $k$ de $n$ suivant le système de numération de base $p$.
\end{thm} 
\begin{proof}
Considérons d'abord les représentations en base $p$ des nombres $n$, $k$ et $n-k$:
\begin{align*}
n&=a_{r}p^{r}+a_{r-1}p^{r-1}+\dots +a_{1}p+a_{0},\\ k&=b_{r}p^{r}+b_{r-1}p^{r-1}+\dots +b_{1}p+b_{0},\\n-k&=c_{r}p^{r}+c_{r-1}p^{r-1}+\dots +c_{1}p+c_{0},
\end{align*} 
où $r\in\mathbb{N}$ et $a_0,a_1,\dots,a_r,b_0,b_1,\dots,b_r,c_0,c_1,\dots,c_r\in\left\lbrace 0,1,\dots,p-1\right\rbrace$. Les emprunts requis $\left(\gamma_{i}\right)_{0\leq i\leq r}$ pour soustraire $k$ de $n$ suivant le système de base $p$ sont alors donnés par:
\[\gamma_{0}:=\begin{cases} 1 &\text{si}~a_0<b_0 \\ 0 &\text{sinon}\end{cases},\]
et pour tout $i\in\left\lbrace 1,2,\dots,r\right\rbrace$:
\[\gamma_{i}:=\begin{cases} 1 &\text{si}~a_i-b_i-\gamma_{i-1}<0 \\ 0 &\text{sinon}\end{cases}.\]
Le nombre d'emprunts non nuls (c'est-à-dire de véritables emprunts) est alors égal à $\gamma_0+\gamma_1+\dots+\gamma_r$. On vérifie facilement (à partir de la définition des $\gamma_i$) que l'on a: $\gamma_r=0$, $c_r=a_r-b_r-\gamma_{r-1}$, $c_0=a_0-b_0+p\gamma_0$ et $c_i=a_i-b_i-\gamma_{i-1}+p\gamma_i$ $(\forall i\in\left\lbrace 1,2,\dots,r-1\right\rbrace)$. D'autre part, en désignant par $s_{p}(\ell)$ ($\ell\in\mathbb{N}$) la somme des chiffres de $\ell$ dans sa représentation en base $p$, on obtient (en vertu de la formule de Legendre):
\begin{align*}
\vartheta_{p}\left(\binom{n}{k}\right)&=\vartheta_{p}\left(n!\right)-\vartheta_{p}\left(k!\right)-\vartheta_{p}\left((n-k)!\right)\\&=\frac{n-s_{p}(n)}{p-1}-\frac{k-s_{p}(k)}{p-1}-\frac{(n-k)-s_{p}(n-k)}{p-1}\\&= \frac{s_{p}(k)+s_{p}(n-k)-s_{p}(n)}{p-1} .
\end{align*}
D'où:
\begin{align*}
\vartheta_{p}\left(\binom{n}{k}\right)&=\frac{\left(b_0+c_0-a_0\right)+\left(b_1+c_1-a_1\right)+\dots+\left(b_r+c_r-a_r\right)}{p-1}\\&=\frac{p\gamma_0+\left(p\gamma_1-\gamma_0\right)+\dots +\left(p\gamma_{r-1}-\gamma_{r-2}\right)-\gamma_{r-1}}{p-1}\\&=\frac{(p-1)\gamma_0+(p-1)\gamma_1+\dots +(p-1)\gamma_{r-1}}{p-1}\\&=\gamma_0+\gamma_1+\dots +\gamma_{r-1}.
\end{align*}
Ce qui confirme le résultat requis et complète cette démonstration.  
\end{proof}

\begin{lemme}[Farhi \cite{Farhi}]\label{llmm}
Soient $n\in\mathbb{N}$ et $p$ un nombre premier. Supposons que $n=\sum_{i=0}^{N}c_ip^{i}$ est la représentation de $n$ dans le système de numération de base $p$, avec $N\in\mathbb{N}$, $c_i\in\lbrace 0,\dots,p-1\rbrace$ $(\forall i\in\left\lbrace 0,1,\dots,N\right\rbrace)$ et $c_{N}\neq 0$. Alors, on a:
\[\max_{0\leq k\leq n}\vartheta_{p}\left(\binom{n}{k}\right)=\vartheta_{p}\left(\binom{n}{p^{N}-1}\right)=\begin{cases} 0 &\text{si}~n=p^{N+1}-1 \\ N-\min\lbrace i;~c_i\neq p-1\rbrace &\text{sinon}\end{cases}.\]
\end{lemme}
\begin{proof}
Nous distinguons les deux cas suivants:\\
$\bullet$\underline{\textbf{1\textsuperscript{er} cas:}} (si $n=p^{N+1}-1$). Dans ce cas, on a $c_{i}=p-1$ pour tout $i\in\lbrace 0,1,\dots,N\rbrace$. Par conséquent, la soustraction de tout $k\in\lbrace 0,1,\dots,n\rbrace$ de $n$ suivant le système de numération de base $p$ ne nécessite aucun emprunt. Il s'ensuit (en vertu du théorème \ref{kum}) que: $\vartheta_{p}\left(\binom{n}{k}\right)=0$ $(\forall k\in\left\lbrace 0,1,\dots,n\right\rbrace)$. D'où l'on a:
\[\max_{0\leq k\leq n}\vartheta_{p}\left(\binom{n}{k}\right)=\vartheta_{p}\left(\binom{n}{p^{N}-1}\right)=0,\]
comme il fallait le prouver.\\
$\bullet$\underline{\textbf{2\textsuperscript{nd} cas:}} (si $n\neq p^{N+1}-1$). Dans ce cas, il existe au moins un $i\in\left\lbrace 0,1,\dots,N\right\rbrace$ tel que $c_i\neq p-1$. Posons $i_{0}:=\min\lbrace i;~c_i\neq p-1\rbrace$. Nous allons montrer que $\vartheta_{p}\left(\binom{n}{k}\right)\leq N-i_{0}$ $(\forall k\in\left\lbrace 0,1,\dots,n\right\rbrace)$ et que $\vartheta_{p}\left(\binom{n}{p^{N}-1}\right)=N-i_{0}$, ce qui conclura au résultat requis. Fixons $k\in\lbrace 0,1,\dots,n\rbrace$. Par définition de $i_0$, on a: $c_0=c_1=\dots =c_{i_0-1}=p-1$. Donc lorsqu'on soustrait $k$ de $n$ suivant le système de base $p$, les premières $i_0$ soustractions chiffre-par-chiffre ne nécessitent aucun emprunt. Ce qui montre que le nombre d'emprunts requis dans cette soustraction est au plus égal à $N-i_{0}$. D'où l'on a (en vertu du théorème \ref{kum}):
\[\vartheta_{p}\left(\binom{n}{k}\right)\leq N-i_{0}~~~~(\forall k\in\left\lbrace 0,1,\dots,n\right\rbrace).\]
Par ailleurs, puisque $c_{i_0}<p-1$, alors lorsqu'on soustrait $p^{N}-1=\sum_{i=0}^{N-1}(p-1)p^{i}$ de $n$ suivant le système de base $p$, chacune des soustractions chiffre-par-chiffre du rang $i_0$ au rang $N-1$ nécessite un emprunt. Ce qui entraîne (d'après le théorème \ref{kum}) que:
\[\vartheta_{p}\left(\binom{n}{p^{N}-1}\right)=N-i_{0}\]
et complète cette démonstration.
\end{proof}

\begin{proof}[Démonstration du théorème \ref{Fip}]
Pour $n=0$, l'identité \eqref{idf} est triviale. Supposons pour la suite que $n\geq 1$ et désignons respectivement par $A_n$ et $B_n$ les membres de gauche et de droite de \eqref{idf}. Nous allons montrer que $\vartheta_p\left(A_n\right)=\vartheta_p\left(B_n\right)$ pour tout nombre premier $p$, ce qui conclura à l'identité \eqref{idf}. Fixons donc un nombre premier $p$ et soit $\sum_{i=0}^{N}c_{i}p^{i}$ la représentation de $n$ suivant le système de base $p$ (où $N\in\mathbb{N}$, $c_{i}\in\lbrace 0,1,\dots,p-1\rbrace$ et $c_{N}\neq 0$). D'après le lemme \ref{llmm}, on a:
\begin{equation}\label{ffffff1}
\vartheta_{p}\left(A_n\right)=\begin{cases} 0 &\text{si}~n=p^{N+1}-1 \\ N-\min\lbrace i;~c_i\neq p-1\rbrace &\text{sinon}\end{cases}.
\end{equation}
D'autre part, puisque $\vartheta_{p}\left(\ppcm\left(1,2,\dots,n,n+1\right)\right)$ est le plus grand exposant $\alpha\in\mathbb{N}$ tel que $p^{\alpha}\leq n+1$, alors:
\begin{equation}\label{ffffff2}
\vartheta_{p}\left(\ppcm\left(1,2,\dots,n,n+1\right)\right)=\begin{cases} N+1 &\text{si}~n=p^{N+1}-1 \\ N &\text{sinon}\end{cases}.
\end{equation}
De plus, on constate que lorsque $n\neq p^{N+1}-1$, on a: $(n+1)=\left(c_{i_{0}}+1\right)p^{i_{0}}+c_{i_{0}+1}p^{i_{0}+1}+\dots +c_{N}p^{N}$ (avec $i_{0}:=\min\lbrace i;~c_i\neq p-1\rbrace$). D'où:
\begin{equation}\label{ffffff3}
\vartheta_{p}\left(n+1\right)=\begin{cases} N+1 &\text{si}~n=p^{N+1}-1 \\ \min\lbrace i;~c_i\neq p-1\rbrace &\text{sinon}\end{cases}.
\end{equation}
En soustrayant \eqref{ffffff3} de \eqref{ffffff2} et en comparant ensuite avec \eqref{ffffff1}, on obtient que $\vartheta_p\left(A_n\right)=\vartheta_p\left(B_n\right)$, comme il fallait le prouver. Ce qui complète cette démonstration. 
\end{proof}

\section{Quelques propriétés de suites à forte divisibilité}
On rappelle qu'une suite d'entiers strictement positifs $\boldsymbol{a}=\left(a_n\right)_{n\geq 1}$ est dite \textit{\`a forte divisibilit\'e} lorsqu'elle v\'erifie la propri\'et\'e:
\[\pgcd\left(a_n,a_m\right)=a_{\pgcd\left(n,m\right)}~~~~(\forall n,m\in\mathbb{N^*}).\]
D'après \textsection\ref{for}, une telle suite $\boldsymbol{a}$ lui correspond une unique suite d'entiers strictement positifs $\left(u_{n}\right)_{n\geq 1}$ telle que:
\begin{equation}\label{kimmm}
a_{n}=\prod_{d\mid n}u_d~~~~(\forall n\geq 1).
\end{equation}
Le théorème suivant traite la correspondance inverse. Plus précisément, il établit une condition nécessaire et suffisante sur une suite d'entiers strictement positifs $\left(u_n\right)_{n\geq 1}$ pour que la suite $\left(\prod_{d\mid n}u_{d}\right)_{n\geq 1}$ soit à forte divisibilité.

\begin{thm}[Bliss et al. \cite{Bliss}]\label{divseq}
Soient $\left(u_{n}\right)_{n\geq 1}$ une suite d'entiers strictement positifs et $\left(a_{n}\right)_{n\geq 1}$ la suite de terme général: $a_{n}=\prod_{d\mid n}u_{d}~~(\forall n\in \mathbb{N^*})$. Alors, les propriétés suivantes sont équivalentes:
\begin{enumerate}
\item La suite $\left(a_{n}\right)_{n\geq 1}$ est à forte divisibilité. 
\item Pour tous $n,m\in \mathbb{N^*}$ tels que $n\nmid m$ et $m\nmid n$, on a: $\pgcd\left(u_{n},u_{m}\right)=1$.\label{cond2}
\end{enumerate}  
\end{thm}  
\begin{proof}~\\
\textbullet{} $\left(\Rightarrow\right)$: Supposons que la suite $\left(a_{n}\right)_{n\geq 1}$ est à forte divisibilité. Soient $m,n\in\mathbb{N^*}$ tels que $n\nmid m$ et $m\nmid n$ et posons $\delta:=\pgcd\left(n,m\right)$. Il existe donc $n',m'\in\mathbb{N^*}$ tels que $n=\delta n'$, $m=\delta m'$ et $\pgcd\left(n',m'\right)=1$. Puisque $n\nmid m$ et $m\nmid n$ alors $n',m'\neq 1$. Par suite, on a (puisque $\left(a_{n}\right)_{n\geq 1}$ est à forte divisibilité):
\[\pgcd\left(\prod_{d\mid n, d\nmid\delta}u_{d},\prod_{d\mid m, d\nmid\delta}u_{d}\right)\prod_{d\mid \delta}u_{d}=\pgcd\left(\prod_{d\mid n}u_{d},\prod_{d\mid m}u_{d}\right)=\pgcd\left(a_{n},a_{m}\right)=a_{\delta}=\prod_{d\mid\delta}u_{d}.\]
Ce qui entraîne que:
\[\pgcd\left(\prod_{d\mid n, d\nmid\delta}u_{d},\prod_{d\mid m, d\nmid\delta}u_{d}\right)=1.\]
Comme on a de toute évidence: $u_{n}\mid\prod_{d\mid n, d\nmid\delta}u_{d}$ et $u_{m}\mid\prod_{d\mid m, d\nmid\delta}u_{d}$, il en découle à fortiori que: $\pgcd\left(u_{n},u_{m}\right)=1$, comme il fallait le prouver.\\ 
\textbullet{} $\left(\Leftarrow\right)$: Inversement, supposons que pour tous $n,m\in\mathbb{N^*}$ tels que $n\nmid m$ et $m\nmid n$, on a: $\pgcd\left(u_{n},u_{m}\right)=1$ et montrons que $\pgcd(a_n,a_m)=a_{\pgcd(n,m)}$ $(\forall n,m\in\mathbb{N^*})$. Fixons $n,m\in\mathbb{N^*}$ et posons $\delta:=\pgcd\left(n,m\right)$. Il existe donc $n',m'\in\mathbb{N^*}$ tels que $n=\delta n'$, $m=\delta m'$ et $\pgcd\left(n',m'\right)=1$. Nous distinguons les deux cas suivants:\\
\underline{\textbf{1\textsuperscript{er} cas:}} (si $n'=1$ ou $m'=1$). Dans ce cas, on a visiblement $n=\delta$ ou $m=\delta$. Sans perte de généralité, supposons que $n=\delta$ (le cas où $m=\delta$ se traite de la même façon). Puisque tout diviseur de $\delta$ est également un diviseur de $\delta m'$, alors $\prod_{d\mid\delta}u_d$ divise $\prod_{d\mid\delta m'}u_d$. On a par conséquent:
\[\pgcd\left(a_{n},a_{m}\right)=\pgcd\left(a_{\delta},a_{\delta m'}\right)=\pgcd\left(\prod_{d\mid\delta}u_d,\prod_{d\mid\delta m'}u_d\right)=\prod_{d\mid\delta}u_{d}=a_{\delta}=a_{\pgcd\left(n,m\right)},\]
comme il fallait le prouver.\\
\underline{\textbf{2\textsuperscript{nd} cas:}} (si $n',m'\neq 1$). Dans ce cas, on a $n\nmid m$ et $m\nmid n$, donc $\pgcd\left(u_{n},u_{m}\right)=1$ (par hypothèse). De plus, on a:
\begin{equation}\label{stdiv}
\pgcd\left(a_n,a_m\right)=\pgcd\left(\prod_{d\mid n, d\nmid\delta}u_{d},\prod_{d\mid m, d\nmid\delta}u_{d}\right)a_{\delta}.
\end{equation}
Maintenant, on constate que si $d_1,d_2\in\mathbb{N^*}$ vérifient: $(d_1\mid n, d_1\nmid\delta)$ et $(d_2\mid m, d_2\nmid\delta)$, alors $\pgcd\left(d_{1},d_{2}\right)$ divise $\delta$; ce qui fait que: $\pgcd\left(d_{1},d_{2}\right)\notin\lbrace d_{1},d_{2}\rbrace$ et donc $d_{1}\nmid d_{2}$ et $d_{2}\nmid d_{1}$; d'où (par hypothèse): $\pgcd\left(u_{d_1},u_{d_2}\right)=1$. Il découle de ce fait que:
\[\pgcd\left(\prod_{d\mid n, d\nmid\delta}u_{d},\prod_{d\mid m, d\nmid\delta}u_{d}\right)=1.\]
En substituant cette dernière dans \eqref{stdiv}, on obtient l'égalité requise:
\[\pgcd\left(a_n,a_m\right)=a_{\delta}=a_{\pgcd\left(n,m\right)}.\]
Ce qui complète cette démonstration. 
\end{proof}

Nowicki \cite{Nowicki} a développé la propriété \ref{cond2} du théorème \ref{divseq} de Bliss et al. et a obtenu une autre quelque part plus pratique. On a le théorème suivant:

\begin{thm}[Nowicki \cite{Nowicki}]\label{nowi}
Soient $\left(a_{n}\right)_{n\geq 1}$ une suite d'entiers strictement positifs et $\left(c_n\right)_{n\geq 1}$ la suite d'entiers définie par: $c_1=a_1$ et
\[c_n:=\frac{\ppcm\left(a_{1},\dots,a_{n}\right)}{\ppcm\left(a_{1},\dots,a_{n-1}\right)}~~~~(\forall n\geq 2).\]
Alors $\left(a_{n}\right)_{n}$ est à forte divisibilité si et seulement si l'on a: 
\[a_{n}=\prod_{d\mid n}c_d~~~~(\forall n\geq 1).\]
\end{thm}
\begin{proof}~\\
\textbullet{} $\left(\Rightarrow\right)$: Supposons que $\left(a_{n}\right)_{n\geq 1}$ est à forte divisibilité. En vertu de \eqref{kimmm}, il existe une unique suite d'entiers strictement positifs $\left(u_{n}\right)_{n\geq 1}$ vérifiant: $a_{n}=\prod_{d\mid n}u_d$ $(\forall n\geq 1)$. Considérons la suite $\left(e_{n}\right)_{n\geq 1}$ donnée par $e_{1}=1$ et $e_{n+1}=\ppcm\left(e_n,a_n\right)$ $(\forall n\in\mathbb{N^*})$. Nous procédons par récurrence pour montrer que:
\begin{equation}\label{b1b2}
e_{n+1}=u_{1}u_{2}\cdots u_{n}~~~~(\forall n\in\mathbb{N^*}),
\end{equation}
ce qui entraînera que $\left(u_n\right)_n$ est identique à $\left(c_n\right)_n$ et conclura au résultat requis. Pour $n=1$, le résultat est trivial; en effet, on a: $e_2=\ppcm\left(e_1,a_1\right)=a_1=u_1$. Supposons pour la suite que $n\geq 2$ et que $e_n=u_{1}u_{2}\cdots u_{n-1}$. Il s'ensuit que:
\[e_{n+1}=\ppcm\left(e_n,a_n\right)=\ppcm\left(\prod_{k=1}^{n-1}u_k,\prod_{d\mid n}u_d\right)=\ppcm\left(A\cdot B,A\cdot u_n\right)=A\cdot\ppcm\left(B,u_n\right),\]
avec $A:=\prod_{d\mid n,d<n}u_d$ et $B:=\prod_{d\nmid n,d<n}u_d$. Par ailleurs, on a (en vertu du théorème \ref{divseq}): $\pgcd\left(u_d,u_n\right)=1$ pour tout $1\leq d<n$ tel que $d\nmid n$. On a par conséquent: $\pgcd\left(B,u_n\right)=1$. Ce qui entraîne que:
\[e_{n+1}=ABu_n=\left(\prod_{k=1}^{n-1}u_k\right)\cdot u_n=u_{1}u_{2}\cdots u_n,\]
comme il fallait le prouver.\\
\textbullet{} $\left(\Leftarrow\right)$: Inversement, supposons que $a_n=\prod_{d\mid n}c_d$ $(\forall n\geq 1)$. Fixons $m\in\mathbb{N^*}$ et posons:
\[U:=\prod_{d\mid n, d<m}c_d,~~V:=\prod_{d\nmid n, d<m}c_d.\]
On a:
\begin{align*}
UV c_m &=\prod_{i=1}^{m}c_i=\ppcm\left(a_1,\dots,a_m\right)=\ppcm\left(\ppcm\left(a_1,\dots,a_{m-1}\right),a_m\right)\\&=\ppcm\left(\prod_{i=1}^{m-1}c_i,\prod_{d\mid m}c_d\right)=\ppcm\left(U\cdot V,U\cdot c_m\right)=U\cdot\ppcm\left(V,c_m\right).
\end{align*} 
Ce qui entraîne que $V\cdot c_m=\ppcm\left(V,c_m\right)$, donc $\pgcd\left(V,c_m\right)=1$. Par conséquent, on a pour tout $1\leq d<m$ tel que $d\nmid m$: $\pgcd\left(c_d,c_m\right)=1$; cela revient à dire que pour tous $n,m\in\mathbb{N^*}$ tels que $n\nmid m$ et $m\nmid n$, on a: $\pgcd\left(c_m,c_n\right)=1$. Ce qui montre (en vertu du théorème \ref{divseq}) que $\left(a_n\right)_{n\geq 1}$ est à forte divisibilité. Le théorème est ainsi démontré.
\end{proof}

\section{Identités concernant le $\ppcm$ de suites à forte divisibilité}

Pour généraliser l'identité \eqref{idf}, nous allons d'abord définir les coefficients binomiaux associés à une suite à forte divisibilité. Étant donné une telle suite ${\boldsymbol{a}}:=\left(a_n\right)_{n\geq 1}$, pour tout entier naturel $n$, on désigne par $[n]_{\boldsymbol{a}}!$ l'entier strictement positif défini par:
\[[n]_{\boldsymbol{a}}!:=a_1a_2\cdots a_n,\]
(en convenant que $[0]_{\boldsymbol{a}}!=1$). Pour $n,k\in\mathbb{N}$, avec $n\geq k$, on désigne par $\binom{n}{k}_{\boldsymbol{a}}$ le nombre rationnel strictement positif défini par:
\[\binom{n}{k}_{\boldsymbol{a}}:=\frac{a_{n}a_{n-1}\cdots a_{n-k+1}}{a_{1}a_{2}\dots a_{k}}=\frac{[n]_{\boldsymbol{a}}!}{[k]_{\boldsymbol{a}}![n-k]_{\boldsymbol{a}}!}.\]  
Ces nombres sont appelés les coefficients $\boldsymbol{a}$-binomiaux. Les coefficients binomiaux usuels s'obtiennent en prenant simplement $a_n=n$ ($\forall n\geq 1$). D'après la définition, on vérifie facilement que les coefficients $\boldsymbol{a}$-binomiaux vérifient les identités suivantes:  
\begin{equation}\label{b1}
\binom{n}{k}_{\boldsymbol{a}}=\binom{n}{n-k}_{\boldsymbol{a}}~~~~\left(\forall n,k\in\mathbb{N}, n\geq k\right)
\end{equation} 
\begin{equation}\label{b2}
a_k\binom{n+1}{k}_{\boldsymbol{a}}=a_{n+1}\binom{n}{k-1}_{\boldsymbol{a}}~~~~\left(\forall n,k\in\mathbb{N}, 1\leq k\leq n+1\right)
\end{equation} ~\\[-10mm]
\begin{equation}\label{b3}
\binom{n}{k}_{\boldsymbol{a}}\binom{k}{l}_{\boldsymbol{a}}=\binom{n}{l}_{\boldsymbol{a}}\binom{n-l}{k-l}_{\boldsymbol{a}}~~~~\left(\forall n,k,l\in\mathbb{N}, l\leq k\leq n\right).
\end{equation} 

\bigskip

En utilisant la représentation \eqref{kimmm} pour le cas particulier $a_n=q^n-1$ (où $q\geq 2$ est un nombre entier), Knuth et Wilf ont montré une formule importante pour les coefficients binomiaux de Gauss (voir \cite[Equation (10)]{Knuth}). En fait, cette formule peut être facilement généralisée pour toute suite à forte divisibilité, comme le montre la proposition suivante:

\begin{prop}[Knuth \cite{Knuth}]\label{knuth}
Soient $\left(u_n\right)_{n\geq 1}$ une suite d'entiers strictement positifs et $\left(a_n\right)_{n\geq 1}$ la suite de terme général:
\[a_n=\prod_{d\mid n}u_d~~~~\left(\forall n\in\mathbb{N^*}\right).\]
Alors, pour tous $n,k\in\mathbb{N}$ tels que $n\geq k$, on a:
\[\binom{n}{k}_{\boldsymbol{a}}=\prod_{d}u_d,\]
où le produit porte sur les entiers strictement positifs $d\leq n$ tels que:
\[\left\lfloor \frac{k}{d}\right\rfloor +\left\lfloor \frac{n-k}{d}\right\rfloor <\left\lfloor \frac{n}{d}\right\rfloor.\]
En particulier, les coefficients $\boldsymbol{a}$-binomiaux $\binom{n}{k}_{\boldsymbol{a}}$ $\left(n,k\in\mathbb{N},~n\geq 1\right)$ sont tous des entiers strictement positifs. 
\end{prop}
\begin{proof}
Soient $n,k\in\mathbb{N}$ tels que $n\geq k$. On a (en vertu de \eqref{kimmm}):
\[\binom{n}{k}_{\boldsymbol{a}}:=\frac{\prod_{n-k< m\leq n}a_m}{\prod_{1\leq m\leq k}a_m}=\frac{\prod_{n-k< m\leq n}\prod_{d\mid m}u_d}{\prod_{1\leq m\leq k}\prod_{d\mid m}u_d}.\]
Comme on a:
\[\prod_{n-k< m\leq n}\prod_{d\mid m}u_d=\prod_{1\leq d\leq n}\prod_{\begin{subarray}{c} n-k< m\leq n \\ m\equiv 0\!\!\pmod d\end{subarray}}u_d=\prod_{1\leq d\leq n}{u_d}^{\left\lfloor \frac{n}{d}\right\rfloor-\left\lfloor \frac{n-k}{d}\right\rfloor}=\prod_{d\geq 1}{u_d}^{\left\lfloor \frac{n}{d}\right\rfloor-\left\lfloor \frac{n-k}{d}\right\rfloor}\]
et
\[\prod_{1\leq m\leq k}\prod_{d\mid m}u_d=\prod_{1\leq d\leq k}\prod_{\begin{subarray}{c} 1\leq m\leq k \\ m\equiv 0\!\!\pmod d\end{subarray}}u_d=\prod_{1\leq d\leq k}{u_d}^{\left\lfloor\frac{k}{d}\right\rfloor}=\prod_{d\geq 1}{u_d}^{\left\lfloor\frac{k}{d}\right\rfloor},\]
il s'ensuit que:
\[\binom{n}{k}_{\boldsymbol{a}}=\prod_{d\geq 1}{u_d}^{\left\lfloor \frac{n}{d}\right\rfloor-\left\lfloor\frac{k}{d}\right\rfloor-\left\lfloor \frac{n-k}{d}\right\rfloor}.\]
L'identité requise de la proposition découle du fait que:
\[\left\lfloor \frac{n}{d}\right\rfloor-\left\lfloor\frac{k}{d}\right\rfloor-\left\lfloor \frac{n-k}{d}\right\rfloor\in\left\lbrace 0,1\right\rbrace~~~~(\forall d\geq 1).\]
Ce qui complète cette démonstration.
\end{proof}

Notre résultat principal est le suivant:

\begin{thm}\label{Ri1}
Soit ${\boldsymbol{a}}=\left(a_n\right)_{n\geq 1}$ une suite à forte divisibilité. Alors, pour tout $n\in\mathbb{N}$, on a:
\begin{equation}\label{rr1}
\ppcm\left\lbrace \binom{n}{0}_{\boldsymbol{a}},\binom{n}{1}_{\boldsymbol{a}},\dots,\binom{n}{n}_{\boldsymbol{a}}\right\rbrace=\frac{\ppcm\left(a_1,a_2,\dots,a_n,a_{n+1}\right)}{a_{n+1}}.
\end{equation}
\end{thm}

Pour démontrer le théorème \ref{Ri1}, nous utiliserons le lemme de Guo \cite{Victor} suivant:

\begin{lemme}[Guo \cite{Victor}]\label{guos}
Soient $n$ et $d$ deux entiers strictement positifs tels que $n\geq d$. Alors, les deux propriétés suivantes sont équivalentes:
\begin{enumerate}
\item Il existe $k\in \left\lbrace 0,1,\dots,n\right\rbrace$ tel que: $\left\lfloor \frac{k}{d} \right\rfloor +\left\lfloor \frac{n-k}{d} \right\rfloor < \left\lfloor \frac{n}{d}\right\rfloor$. 
\item Le nombre $d$ ne divise pas $(n+1)$.
\end{enumerate}
\end{lemme}
\begin{proof}
Supposons que la première propriété du lemme est vérifiée et désignons respectivement par $a$ et $b$ les restes des divisions euclidiennes de $k$ et $(n-k)$ sur $d$. Notre hypothèse est alors équivalente à:
\[\left\lfloor \frac{a}{d} \right\rfloor +\left\lfloor \frac{b}{d}\right\rfloor < \left\lfloor \frac{a+b}{d}\right\rfloor.\]
Ce qui entraîne que: $1\leq a,b\leq d-1$ et $d\leq a+b\leq 2d-2$. Puisque $n\equiv a+b\pmod d$, alors $n+1\equiv a+b+1\not\equiv 0\pmod d$, comme il fallait le prouver.\\
Inversement, supposons que $n+1\equiv c\pmod d$ pour un certain $c\in\left\lbrace 1,2,\dots,d-1\right\rbrace$. Il existe donc un $q\in\mathbb{N^*}$ tel que $n+1=c+qd$. On a par conséquent:
\[\left\lfloor \frac{c}{d} \right\rfloor +\left\lfloor \frac{n-c}{d}\right\rfloor=\left\lfloor\frac{qd-1}{d}\right\rfloor=q-1\leq \left\lfloor \frac{n}{d}\right\rfloor-1 < \left\lfloor \frac{n}{d}\right\rfloor.\]
On conclut à la première propriété du lemme en prenant $k=c$. Ce qui complète cette démonstration.
\end{proof}

\begin{proof}[Démonstration du théorème \ref{Ri1}]
Pour $n=0$, l'identité \eqref{rr1} du théorème \ref{Ri1} est triviale. Supposons pour la suite que $n\geq 1$ et désignons respectivement par $A_n$ et $B_n$ le membre de gauche et le membre de droite de \eqref{rr1}. Nous allons montrer que $A_n$ divise $B_n$ puis que $B_n$ divise $A_n$, ce qui conclura que $A_n=B_n$. Puisque la suite $\boldsymbol{a}$ est à forte divisibilité, alors (d'après le théorème \ref{nowi}) on a pour tout entier $m\geq 1$:   
\[a_m=\prod_{d\mid m}u_d,\]
où $\left(u_{d}\right)_{d\geq 1}$ est la suite d'entiers strictement positifs définie par:
\[u_1:=a_1~~\text{et}~~u_d:=\frac{\ppcm\left(a_1,\dots,a_d\right)}{\ppcm\left(a_1,\dots,a_{d-1}\right)}~~~~\left(\forall d\geq 2\right).\]
D'autre part, de la définition de $\left(u_{d}\right)_{d}$, découle immédiatement que:
\[\ppcm\left(a_1,\dots,a_{m}\right)=\prod_{d=1}^{m}u_d~~~~\left(\forall m\geq 1\right).\]
Maintenant, pour tout $k\in\left\lbrace 0,1,\dots,n\right\rbrace$, le produit $\prod_{d}u_d=\binom{n}{k}_{\boldsymbol{a}}$ fourni par la proposition \ref{knuth} porte sur les entiers $d$ qui vérifient tous (d'après le lemme \ref{guos}): $1\leq d\leq n$ et $d\nmid (n+1)$. Cela implique que le produit:
\[\prod_{\begin{subarray}{c}1\leq d\leq n \\ d\nmid (n+1) \end{subarray}}u_d=\frac{\displaystyle\prod_{1\leq d\leq n+1}u_d}{\displaystyle\prod_{d\mid (n+1)}u_d}=\frac{\ppcm\left(a_1,\dots,a_{n},a_{n+1}\right)}{a_{n+1}}=B_n\]
est un multiple de chacun des nombres $\binom{n}{k}_{\boldsymbol{a}}$ $\left(0\leq k\leq n\right)$. D'où $B_n$ est un multiple de:
\[\ppcm\left\lbrace \binom{n}{0}_{\boldsymbol{a}},\binom{n}{1}_{\boldsymbol{a}},\dots,\binom{n}{n}_{\boldsymbol{a}}\right\rbrace=A_n.\]
Ce qui montre que $A_n\mid B_n$. Inversement, il est immédiat que:
\[\ppcm\left(a_{1},a_{2},\dots,a_{n+1}\right)~\text{divise}~\ppcm\left\lbrace a_1\binom{n+1}{1}_{\boldsymbol{a}},a_2\binom{n+1}{2}_{\boldsymbol{a}},\dots,a_{n+1}\binom{n+1}{n+1}_{\boldsymbol{a}}\right\rbrace,\]
qui est égal (en vertu de \eqref{b2}) à:
\begin{align*}
\ppcm\left\lbrace a_{n+1}\binom{n}{0}_{\boldsymbol{a}},a_{n+1}\binom{n}{1}_{\boldsymbol{a}},\dots,a_{n+1}\binom{n}{n}_{\boldsymbol{a}}\right\rbrace &=a_{n+1}\ppcm\left\lbrace\binom{n}{0}_{\boldsymbol{a}},\binom{n}{1}_{\boldsymbol{a}},\dots,\binom{n}{n}_{\boldsymbol{a}}\right\rbrace \\&=a_{n+1}A_{n}.
\end{align*}
D'où l'on déduit que:
\[\frac{\ppcm\left(a_{1},\dots,a_n,a_{n+1}\right)}{a_{n+1}}=B_n~\text{divise}~A_n.\]
Ce qui complète cette démonstration.
\end{proof}

Du théorème \ref{Ri1}, nous tirons les deux corollaires suivants:

\begin{coll}\label{sour}
Soit ${\boldsymbol{a}}=\left(a_n\right)_{n\geq 1}$ une suite à forte divisibilité. Alors, pour tout entier strictement positif $n$, on a:
\[\ppcm\left(a_1,a_2,\dots,a_n\right)=\ppcm\left\lbrace a_1\binom{n}{1}_{\boldsymbol{a}},\dots,a_n\binom{n}{n}_{\boldsymbol{a}}\right\rbrace.\] 
\end{coll}
\begin{proof}
Pour tout entier strictement positif $n$, on a d'après la formule \eqref{b2}:
\begin{align*}
\ppcm\left\lbrace a_1\binom{n}{1}_{\boldsymbol{a}},\dots,a_{n}\binom{n}{n}_{\boldsymbol{a}}\right\rbrace &=\ppcm\left\lbrace a_{n}\binom{n-1}{0}_{\boldsymbol{a}},a_{n}\binom{n-1}{1}_{\boldsymbol{a}},\dots,a_{n}\binom{n-1}{n-1}_{\boldsymbol{a}}\right\rbrace \\ &=a_{n}\ppcm\left\lbrace \binom{n-1}{0}_{\boldsymbol{a}},\binom{n-1}{1}_{\boldsymbol{a}},\dots,\binom{n-1}{n-1}_{\boldsymbol{a}}\right\rbrace \\&=\ppcm\left(a_{1},\dots,a_{n}\right)~~~~\text{(d'après le théorème \ref{Ri1}),}
\end{align*}
comme il fallait le prouver.
\end{proof}

\begin{coll}\label{sour1}
Soit ${\boldsymbol{a}}=\left(a_n\right)_{n\geq 1}$ une suite à forte divisibilité. Alors, pour tout entier strictement positif $n$, on a:
\[\ppcm\left(a_1,a_2,\dots,a_n\right)=\pgcd\left\lbrace \binom{n}{k}_{\boldsymbol{a}}\ppcm\left(a_1,\dots,a_k\right);~n/2\leq k\leq n \right\rbrace.\]
\end{coll} 
Pour présenter une preuve plus propre de ce corollaire, nous faisons intervenir le lemme élémentaire suivant:

\begin{lemme}\label{ccc1}
Soient $n,m\in\mathbb{N^*}$ et $a_{1},\dots,a_{n}$, $b_{1},\dots,b_{m}$ des entiers strictement positifs. Alors, la propriété affirmant que: $a_i$ divise $b_j$ pour tous $i,j$ tels que $1\leq i\leq n$ et $1\leq j\leq m$ est équivalente à la propriété affirmant que $\ppcm\left(a_1,\dots,a_n\right)$ divise $\pgcd\left(b_1,\dots,b_m\right)$.
\end{lemme}

\begin{proof}[Démonstration du corollaire \ref{sour1}]
Soit $n\in\mathbb{N^*}$ fixé. Pour $k,\ell\in\mathbb{N}$ tels que $n/2\leq k\leq n$ et $1\leq \ell\leq k$, on a visiblement $a_{\ell}\binom{n}{\ell}_{\boldsymbol{a}}$ divise $a_{\ell}\binom{n}{\ell}_{\boldsymbol{a}}\binom{n-\ell}{k-\ell}_{\boldsymbol{a}}$, qui est égale (en vertu de \eqref{b3}) à $a_{\ell}\binom{n}{k}_{\boldsymbol{a}}\binom{k}{\ell}_{\boldsymbol{a}}$. Mais ce dernier nombre divise clairement le nombre: 
\[\binom{n}{k}_{\boldsymbol{a}}\ppcm\left\lbrace a_{i}\binom{k}{i}_{\boldsymbol{a}};~i=1,\dots,k\right\rbrace,\]
qui est égale (en vertu du corollaire \ref{sour}) à $\binom{n}{k}_{\boldsymbol{a}}\ppcm\left(a_{1},\dots,a_{k}\right)$. Par conséquent, pour tous $k,\ell\in\mathbb{N}$, tels que $n/2\leq k\leq n$ et $1\leq \ell\leq k$, on a:
\begin{equation}\label{ccc2}
a_{\ell}\binom{n}{\ell}_{\boldsymbol{a}}~\text{divise}~\binom{n}{k}_{\boldsymbol{a}}\ppcm\left(a_{1},\dots,a_{k}\right).
\end{equation}
Nous affirmons que \eqref{ccc2} reste vraie pour $n/2\leq k\leq n$ et $k<\ell\leq n$. En effet, si $k$ et $\ell$ sont des entiers tels que $n/2\leq k\leq n$ et $k< \ell\leq n$, alors on a $1\leq n-\ell+1\leq n-k\leq k$ et $a_{n-\ell+1}\binom{n}{n-\ell+1}_{\boldsymbol{a}}=a_{\ell}\binom{n}{\ell}_{\boldsymbol{a}}$. L'application de \eqref{ccc2} pour $\ell'=n-\ell+1$, au lieu de $\ell$, confirme donc notre affirmation. Ainsi, \eqref{ccc2} est vraie pour tous $k,\ell\in\mathbb{N}$ tels que $n/2\leq k\leq n$ et $1\leq \ell\leq n$. Par suite, en appliquant le lemme \ref{ccc1} pour toutes les relations de divisibilité \eqref{ccc2}, où $1\leq \ell\leq n$ et $n/2\leq k\leq n$, on en déduit que:
\[\ppcm\left\lbrace a_{\ell}\binom{n}{\ell}_{\boldsymbol{a}};~\ell=1,\dots,n\right\rbrace~\text{divise}~\pgcd\left\lbrace \binom{n}{k}_{\boldsymbol{a}}\ppcm\left(a_{1},\dots,a_{k}\right);~n/2\leq k\leq n\right\rbrace;\]
ce qui est équivalent (en vertu du corollaire \ref{sour}) à:
\[\ppcm\left(a_{1},a_{2},\dots,a_{n}\right)~\text{divise}~\pgcd\left\lbrace\binom{n}{k}_{\boldsymbol{a}}\ppcm\left(a_{1},\dots,a_{k}\right);~n/2\leq k\leq n\right\rbrace.\] 
L'identité du corollaire \ref{sour1} se déduit en observant que:
\[\ppcm\left(a_{1},\dots,a_{n}\right)=\binom{n}{n}_{\boldsymbol{a}}\ppcm\left(a_{1},\dots,a_{n}\right)\in\left\lbrace\binom{n}{k}_{\boldsymbol{a}}\ppcm\left(a_{1},\dots,a_{k}\right);~n/2\leq k\leq n\right\rbrace.\]
Ce qui complète la démonstration du corollaire \ref{sour1}.
\end{proof}

\begin{rmq}\label{rmq1}
En prenant dans le théorème \ref{Ri1} $a_n=n$ $(\forall n\geq 1)$, on obtient l'identité de Farhi \cite{Farhi} (déjà démontrée dans le théorème \ref{Fip}). Par ailleurs, on démontre aisément que le théorème \ref{Ri1}, le corollaire \ref{sour} et le corollaire \ref{sour1} restent valables dans tout anneau factoriel $\mathcal{A}$, pris à la place de $\mathbb{Z}$ (nous renvoyons le lecteur à l'article de Bliss et al. \cite{Bliss} pour la définition et les propriétés des suites à forte divisibilité dans un anneau factoriel). Si l'on prend par exemple $\mathcal{A}=\mathbb{Z}[q]$ et $\boldsymbol{a}=\left(a_n\right)_{n\geq 1}$ la suite polynomiale de $\mathbb{Z}[q]$ d\'efinie par $a_n=[n]_{q}:=\frac{q^n-1}{q-1}$, on obtient l'identit\'e de Guo \cite{Victor} selon laquelle on a pour tout $n\in\mathbb{N}$:
\[\ppcm\left\lbrace\binom{n}{0}_{q},\binom{n}{1}_{q},\dots,\binom{n}{n}_{q}\right\rbrace=\frac{\ppcm\left([1]_q,[2]_q,\dots,[n]_q,[n+1]_q\right)}{[n+1]_{q}},\]
o\`u $[k]_{q}$ et $\binom{n}{k}_{q}$ $(0\leq k\leq n)$ sont les notations standards du $q$-calcul; c'est-à-dire $[k]_{q}:=\frac{q^k-1}{q-1}$ et $\binom{n}{k}_{q}:=\frac{[n]_q[n-1]_q\cdots [n-k+1]_q}{[1]_q[2]_q\cdots [k]_q}$.
\end{rmq}

\section{Estimations du $\ppcm$ de suites de Lucas}

Un important exemple de suites à forte divisibilité nous est fourni par une classe spéciale de suites de Lucas. Plus précisément, en prenant $P$ et $Q$ des entiers non nuls et premiers entre eux et en désignant par $U\left(P,Q\right)$ leur suite de Lucas associée, c'est-à-dire la suite d'entiers d\'efinie r\'ecursivement par: $U_0=0$, $U_1=1$ et $U_{n+2}=PU_{n+1}-QU_{n}$ $(\forall n\in\mathbb{N})$, on montre que la suite $\left|U\left(P,Q\right)\right|$ est \`a forte divisibilit\'e (voir par exemple \cite[p. 9]{div}). Particuli\`erement, la suite de tous les entiers positifs (qu'on obtient en prenant $\left(P,Q\right)=(2,1)$) et la suite de Fibonacci usuelle (qu'on obtient en prenant $\left(P,Q\right)=(1,-1)$) sont des suites \`a forte divisibilit\'e. Par ailleurs, si $P^2-4Q>0$, alors en désignant par $\alpha$ et $\beta$ les deux racines (distinctes) de l'équation quadratique: $X^2-PX+Q=0$, on montre que pour tout entier strictement positif $n$, on a: 
\begin{equation}\label{sa}
U_n=\frac{\alpha^n-\beta^n}{\alpha-\beta}.
\end{equation}
Pour en savoir plus sur le sujet des suites de Lucas, le lecteur est invit\'e \`a consulter le livre de Honsberger \cite{H}. Dans cette section, nous appliquons les corollaires \ref{sour} et \ref{sour1} pour établir des estimations effectives et non triviales du plus petit commun multiple des termes consécutifs de certaines suites de Lucas. On a le théorème suivant:

\begin{thm}\label{aR4}
Soient $P$ et $Q$ deux entiers non nuls et premiers entre eux, tels que $\Delta:=P^2-4Q>0$, et soit $U\left(P,Q\right)$ la suite de Lucas qui leur est associée. Alors, pour tout $n\in\mathbb{N^*}$, on a:
\begin{equation}\label{sh}
\left|\alpha\right|^{\frac{n^2}{4}-\frac{n}{2}-1}\leq\ppcm\left(U_1,U_2,\dots,U_n\right)\leq \left|\alpha\right|^{\frac{n^2}{3}+\frac{7n}{3}-\frac{8}{3}},
\end{equation}
où $\alpha$ est la racine la plus grande en valeur absolue de l'équation $X^2-PX+Q=0$.
\end{thm}

Pour prouver le théorème \ref{aR4}, nous aurons besoin du lemme élémentaire suivant:

\begin{lemme}\label{terrr}
Dans la situation du théorème \ref{aR4}, on a pour tout entier strictement positif $n$:
\[\left|\alpha\right|^{n-2}\leq\left|U_n\right|\leq\left|\alpha\right|^{n}.\] 
\end{lemme}
\begin{proof}
Désignons par $\beta$ la seconde racine de l'équation $X^2-PX+Q=0$; donc $|\beta|<|\alpha|$. On a $|\alpha|-|\beta|\in\left\lbrace \alpha-\beta,\beta-\alpha,\alpha+\beta,-\alpha-\beta\right\rbrace$. Puisque $\alpha-\beta=\pm\sqrt{\Delta}$, $\alpha+\beta=P$ et $|\alpha|-|\beta|>0$, cela entraîne que $|\alpha|-|\beta|\in\left\lbrace \left|P\right|,\sqrt{\Delta}\right\rbrace$. Comme $P\in\mathbb{Z}^*$ et $\Delta\in\mathbb{Z}^{*}_{+}$, on en déduit que:
\begin{equation}\label{supa1}
|\alpha|-|\beta|\geq 1.
\end{equation}
En utilisant les formules \eqref{sa} et \eqref{supa1}, on obtient que pour tout $n\in\mathbb{N^*}$:
\begin{align*}
\left|U_n\right|&=\left|\frac{\alpha^n-\beta^n}{\alpha-\beta}\right|=\left|\beta^{n-1}\sum_{k=0}^{n-1}\left(\frac{\alpha}{\beta}\right)^{k}\right| \\ &\leq\left|\beta\right|^{n-1}\sum_{k=0}^{n-1}\left|\frac{\alpha}{\beta}\right|^{k}\\&=\frac{|\alpha|^n-|\beta|^n}{|\alpha|-|\beta|}\\&\leq|\alpha|^n-|\beta|^n\leq |\alpha|^n.
\end{align*}
D'autre part, on a pour tout entier $n\geq 2$:
\begin{align*}
\left|U_n\right|&=\left|\frac{\alpha^n-\beta^n}{\alpha-\beta}\right|=\frac{\left|\alpha^n-\beta^n\right|}{\left|\alpha-\beta\right|}\geq \frac{\left|\left|\alpha^n\right|-\left|\beta^n\right|\right|}{\left|\alpha-\beta\right|}=\frac{\left|\alpha\right|^n-\left|\beta\right|^n}{\left|\alpha-\beta\right|}\\&=\frac{\left(|\alpha|-|\beta|\right)\left(|\alpha|^{n-1}+|\alpha|^{n-2}\cdot|\beta|+\dots+|\alpha|\cdot|\beta|^{n-2}+|\beta|^{n-1}\right)}{\left|\alpha-\beta\right|}\\&\geq \frac{\left(|\alpha|-|\beta|\right)\left(|\alpha|^{n-1}+|\alpha|^{n-2}\cdot|\beta|\right)}{\left|\alpha-\beta\right|}\\&=\left|\alpha+\beta\right|\cdot |\alpha|^{n-2}\\&=\left|P\right|\cdot |\alpha|^{n-2}\geq |\alpha|^{n-2}.
\end{align*}
En remarquant que $\left|U_n\right|\geq |\alpha|^{n-2}$ est aussi vraie pour $n=1$ (puisque $U_1=1$ et $|\alpha|\geq |\alpha|-|\beta|\geq 1$), on en déduit que pour tout entier strictement positif $n$, on a:
\[|\alpha|^{n-2}\leq \left|U_n\right|\leq |\alpha|^{n},\]
comme il fallait le prouver. Le lemme est ainsi démontré.
\end{proof}

\begin{proof}[Démonstration du théorème \ref{aR4}]
Désignons par $\beta$ la seconde racine de l'équation $X^2-PX+Q=0$; donc $|\beta|<|\alpha|$. En appliquant l'estimation du lemme \ref{terrr} pour $n\geq 2$ et en replaçant $U_1$ par $1$, nous en déduisons immédiatement que pour tous $n,k\in\mathbb{N^*}$ tels que $n\geq k$, on a: 
\begin{equation}\label{R42}
\left|\alpha\right|^{k\left(n-k-2\right)+1}\leq \left|\binom{n}{k}_{\boldsymbol{U}}\right|\leq \left|\alpha\right|^{k(n-k+2)-1}.
\end{equation}
Montrons d'abord l'inégalité de gauche de \eqref{sh}. Pour $n=1$, cette inégalité est triviale. Ensuite, en utilisant successivement le corollaire \ref{sour}, le lemme \ref{terrr} puis \eqref{R42}, on obtient pour tout entier $n\geq 2$:    
\begin{align*}
\ppcm\left(U_1,U_{2},\dots,U_{n}\right)&=\ppcm\left\lbrace U_{1}\binom{n}{1}_{\boldsymbol{U}},U_{2}\binom{n}{2}_{\boldsymbol{U}},\dots,U_{n}\binom{n}{n}_{\boldsymbol{U}}\right\rbrace\\&\geq \max_{1\leq k\leq n}\left\lbrace \left|U_{k}\right|\cdot\binom{n}{k}_{\boldsymbol{U}}\right\rbrace \\&\geq \max_{1\leq k\leq n}|\alpha|^{k(n-k-1)-1}\\&\geq |\alpha|^{\left\lfloor \frac{n}{2}\right\rfloor\left(n-\left\lfloor \frac{n}{2}\right\rfloor-1\right)-1}\\&\geq\left|\alpha\right|^{n^2/4 -n/2-1},
\end{align*}
comme il fallait le prouver. L'inégalité de gauche de \eqref{sh} est prouvée. Montrons maintenant l'inégalité de droite de \eqref{sh}; c'est-à-dire que $\ppcm\left(U_1,U_{2},\dots,U_{n}\right)\leq \left|\alpha\right|^{\frac{n^2}{3}+\frac{7n}{3}-\frac{8}{3}}$ $(\forall n\geq 1)$. Pour ce faire, on procède par récurrence sur $n$. Pour $n=1$, cette inégalité est triviale. Pour $m\geq 1$, supposons que l'inégalité précédente est vraie pour tout entier strictement positif $n<2m$ et montrons qu'elle reste vraie pour $n=2m$ et pour $n=2m+1$. En utilisant successivement le corollaire \ref{sour1}, l'hypothèse de récurrence et \eqref{R42}, on obtient:
\begin{align*}
\ppcm\left(U_1,U_{2},\dots,U_{2m}\right)&\leq \ppcm\left(U_1,U_{2},\dots,U_{m}\right)\cdot\left|\binom{2m}{m}_{\boldsymbol{U}}\right|\\&\leq |\alpha|^{\frac{m^2}{3}+\frac{7m}{3}-\frac{8}{3}}\cdot |\alpha|^{m^2+2m-1}\\&=|\alpha|^{\frac{4m^2}{3}+\frac{13m}{3}-\frac{11}{3}}\\&\leq |\alpha|^{\frac{\left(2m\right)^2}{3}+\frac{7\left(2m\right)}{3}-\frac{8}{3}},      
\end{align*}
comme il fallait le prouver. De même, on a:
\begin{align*}
\ppcm\left(U_1,U_{2},\dots,U_{2m+1}\right)&\leq \ppcm\left(U_1,U_{2},\dots,U_{m+1}\right)\cdot\left|\binom{2m+1}{m+1}_{\boldsymbol{U}}\right|\\&=\ppcm\left(U_1,U_{2},\dots,U_{m+1}\right)\cdot\left|\binom{2m+1}{m}_{\boldsymbol{U}}\right|\\&\leq |\alpha|^{\frac{(m+1)^2}{3}+\frac{7(m+1)}{3}-\frac{8}{3}}\cdot |\alpha|^{m^2+3m-1}\\&=|\alpha|^{\frac{4m^2}{3}+6m-1}\\&\leq |\alpha|^{\frac{4m^2}{3}+6m}=|\alpha|^{\frac{(2m+1)^2}{3}+\frac{7(2m+1)}{3}-\frac{8}{3}}, 
\end{align*}
comme il fallait le prouver. Ce qui achève cette récurrence et confirme que l'inégalité de droite de \eqref{sh} est vraie pour tout entier $n\geq 1$. La preuve du théorème est complète.
\end{proof}

\begin{rmqs}\label{rmq2}~
\begin{enumerate}
\item Dans le contexte du théorème \ref{aR4}, si $P$ et $Q$ sont de signes particuliers (par exemple $P>0$, $Q<0$) alors l'estimation \eqref{sh} peut être légèrement améliorée. Par exemple, pour le cas de la suite de Fibonacci usuelle $\left(F_n\right)_{n\in\mathbb{N}}$ (définie récursi\-vement par: $F_0=0$, $F_1=1$ et $F_{n+2}=F_{n}+F_{n+1}$, $\forall n\in\mathbb{N}$), on montre que l'on a pour tout entier $n\geq 1$:
\begin{equation}\label{fib}
\Phi^{\frac{n^2}{4}-\frac{9}{4}}\leq \ppcm\left(F_1,F_2,\dots,F_n\right)\leq \Phi^{\frac{n^2}{3}+\frac{4n}{3}},
\end{equation}   
où $\Phi$ désigne le nombre d'or $(\Phi:=\frac{1+\sqrt{5}}{2})$. La qualité de l'estimation \eqref{fib} peut être appréciée à partir du célèbre résultat de Matiyasevich et Guy \cite{Mat} qui énonce que:
\[\lim_{n\longrightarrow +\infty}\frac{\log \ppcm\left(F_1,F_2,\dots,F_n\right)}{n^2\log \Phi}=\frac{3}{\pi^2}.\]
En effet, ce résultat implique que si $\lambda_1,\mu_1,\eta_1,\lambda_2,\mu_2,\eta_2\in\mathbb{R}$ vérifient:
\[\Phi^{\lambda_1n^2+\mu_1n+\eta_1}\leq\ppcm\left(F_1,F_2,\dots,F_n\right)\leq \Phi^{\lambda_2n^2+\mu_2n+\eta_2}~~~~\left(\forall n\geq 1\right),\] 
alors, on a nécessairement $\lambda_1\leq \frac{3}{\pi^2}$ et $\lambda_2\geq \frac{3}{\pi^2}$. Puisque \eqref{fib} correspond à $\lambda_1=\frac{1}{4}=0,25$ et $\lambda_2=\frac{1}{3}=0,33\dots$ et que $\frac{3}{\pi^2}=0,303\dots$, nous voyons bien que notre estimation \eqref{fib} est assez précise.\\
Plus généralement, l'estimation du théorème \ref{aR4} peut être appréciée à partir du résultat de Kiss et Matyas \cite{kiss}, généralisant celui de Matiyasevich et Guy \cite{Mat}.
\item Notons aussi que la méthode suivie ici peut être adaptée pour estimer le $\ppcm$ de toute suite à forte divisibilité d'ordre connu, comme celles d'ordre exponentiel.
\end{enumerate}
\end{rmqs}

\bigskip

\begin{center}
\includegraphics[scale=1]{fini}
\end{center}

\chapter{Majorations non triviales du $\ppcm$ d'une suite arithmétique}\label{ch4}

\section{Introduction}

Dans ce chapitre, nous \'etablissons des majorations non triviales du plus petit commun multiple de termes cons\'ecutifs d'une progression arithm\'etique finie. Comme cons\'equence, nous obtenons d'une part un encadrement effectif et un \'equivalent \`a l'infini de:
\[M(n):=\frac{1}{\varphi(n)}\sum_{\substack{1\leq\ell\leq n \\ \pgcd(\ell,n)=1}}\frac{1}{\ell}~~~~(\forall n\in\mathbb{N^*}),\]
où $\varphi$ désigne la fonction indicatrice d'Euler; et d'autre part une majoration non triviale de la fonction de Chebyshev $\theta\left(x;k,\ell\right)$ sous certaines conditions sur $x , k$ et $\ell$. Bien que des majorations de $\theta\left(x;k,\ell\right)$ existent dans la litt\'erature math\'ematique, il convient de noter que celles-ci sont toutes obtenues de fa\c con non \'el\'ementaire (c'est-\`a-dire en utilisant de l'analyse complexe). L'int\'er\^et ici est justement d'en donner des estimations par le biais de l'analyse réelle.

Étant donné $n\in\mathbb{N^*}$, on désigne par $\omega(n)$ le nombre de facteurs premiers distincts de $n$. Étant donnés $n,a,b\in\mathbb{N^*}$, on d\'efinit $L_{a,b,n}:=\ppcm\left(a,a+b,\dots,a+nb\right)$. Afin d'all\'eger certains \'enonc\'es, on pose: $c_1:=41,30142$, $c_2:=12,30641$, $c_3:=1,25507$, $c_4:=3,35609$, $c_5:=1,38402$, $c_6:=1,57681$ et $c_7:=2,1284$.


\section{\'Enonc\'es des r\'esultats}

\begin{thm}\label{9}  
Soient $a$ et $b$ deux entiers strictement positifs tels que $b\geq 2$ et $\pgcd(a,b) = 1$. Alors, pour tout entier $n\geq b+1$, on a:
\begin{equation}
\ppcm\left(a,a+b,\dots,a+nb\right)\leq \left(c_1\cdot b\log b\right)^{n+\left\lfloor \frac{a}{b}\right\rfloor}.
\end{equation}
\end{thm}

\begin{thm}\label{10}
Soient $a$ un entier strictement positif et $b$ un nombre premier tels que $a < b$. Alors, pour tout entier $n\geq b+1$, On a:
\begin{equation}
\ppcm\left(a,a+b,\dots,a+nb\right)\leq \left(c_2\cdot b^{\frac{b}{b-1}}\right)^n.
\end{equation} 
\end{thm}

\begin{coll}\label{B}
Pour tout entier $r\geq 2$, on a: 
\begin{equation}\label{B2}
\log (r+1)\leq rM(r)\leq \log r+\log\log r +\log c_1 .
\end{equation}
\end{coll}

\noindent Le corollaire suivant est imm\'ediat. 
\begin{coll}\label{E}
On a:
\begin{equation}\label{B1}
M(r)\sim_{+\infty}\frac{\log r}{r}.
\end{equation}
\end{coll}

\begin{coll}\label{11}
Soient $k$ un nombre premier et $\ell<k$ un entier strictement positif. Alors, pour tout nombre r\'eel $x\geq k(k+1)$, on a:
\begin{equation}
\theta(x;k,\ell)\leq  x\left(\frac{2c_3}{k}+\frac{\log k}{k-1}\right).
\end{equation} 
\end{coll}

\section{Pr\'eparation}

Les preuves de nos r\'esultats n\'ecessitent les r\'esultats interm\'ediaires suivants: 

\subsection{R\'esultats connus ant\'erieurement}  

\begin{thm}[Rosser et al. \cite{Rosser}]\label{1}
\noindent
\begin{enumerate}
\item Pour tout entier $n\geq 2$, on a: $\sum_{p\leq n}\frac{\log p}{p}\leq \log n$.
\item Pour tout entier $n\geq 6$, on a: $p_n\leq n\left(\log n +\log \log n\right)$.
\item La s\'erie $\sum_{p}\frac{\log p}{p(p-1)}$ converge vers le nombre $0,7553666111\dots<\log c_7$.\label{3}
\end{enumerate}
\end{thm}

\begin{thm}[Robin \cite{Robin}]\label{4}
Pour tout entier $n\geq 3$, on a:
\[\omega(n)\leq c_5 \frac{\log n}{\log\log n}.\]
\end{thm}

Pour mettre à l'aise le lecteur, nous lui rappelons aussi le théorème de Hanson \cite{han} suivant qui est déjà vu au \textsection\ref{prec}.
  
\begin{thm}[Hanson \cite{han}]\label{2}
Pour tout r\'eel $x > 1$, on a:
\[\pi(x)\leq c_3 \frac{x}{\log x}.\]
\end{thm}

\subsection{Lemmes pr\'eparatifs}

\begin{lemme}\label{6}
Soient $a$ et $b$ deux entiers strictement positifs et premiers entre eux tels que $a < b$. Alors, pour tout entier $n \geq b+1$, on a:
\[\prod_{p\leq n}p^{\vartheta_p\left(L_{a,b,n}\right)}\leq \frac{{c_2}^n}{(a+nb)^{\omega(b)}}.\]
\end{lemme}
\begin{proof}
Soit $n \geq b+1$ un entier. On constate que pour tout nombre premier $p$ divisant $b$, on a $\vartheta_p\left(L_{a,b,n}\right)=0$. En effet, $p\nmid b$ entra\^{\i}ne $p\nmid a$ (puisque $\pgcd(a,b)=1$), ce qui entra\^{\i}ne que $p$ ne divise aucun terme de la suite arithm\'etique $(a+kb)_{k\in\mathbb{N}}$; d'o\`u $p\nmid L_{a , b , n}$. D'autre part, pour tout nombre premier $p$, le nombre $p^{\vartheta_p\left(L_{a,b,n}\right)}$ est la plus grande puissance de $p$ qui divise l'un au moins des nombres $a,a+b,\dots,a+nb$; ce qui entra\^ine que $p^{\vartheta_p\left(L_{a,b,n}\right)}\leq a+nb$. On a par cons\'equent:
\[\prod_{p\leq n}p^{\vartheta_p\left(L_{a,b,n}\right)}=\prod_{\begin{subarray}{c} p\leq n \\ p\nmid b \end{subarray}}p^{\vartheta_p\left(L_{a,b,n}\right)}\leq \prod_{\begin{subarray}{c} p\leq n \\ p \nmid b \end{subarray}}(a+nb)=(a+nb)^{\pi(n)-\omega(b)}.\]
Comme par ailleurs, on a: $a+nb\leq n^2$ (car $n\geq b+1>a$), il s'ensuit que:
\[\prod_{p\leq n}p^{\vartheta_p\left(L_{a,b,n}\right)}\leq \frac{n^{2\pi(n)}}{(a+nb)^{\omega(b)}}=\frac{e^{2\pi(n)\log n}}{(a+nb)^{\omega(b)}}.\] 
L'estimation requise en d\'ecoule via le th\'eor\`eme \ref{2}.
\end{proof}

\begin{lemme}\label{5}
Soient $a$ et $b$ deux entiers strictement positifs et premiers entre eux. Alors, pour tout entier naturel $n$, on a:
\begin{equation}\label{A}
\prod_{p>n}p^{\vartheta_{p}\left(L_{a,b,n}\right)}~~\text{divise}~~\frac{a\left(a+b\right)\cdots \left(a+nb\right)\cdot \prod_{\begin{subarray}{c} p\leq n \\ p\mid b\end{subarray}}p^{\vartheta_{p}\left(n!\right)}}{n!\cdot\prod_{\begin{subarray}{c} p\leq n \\ p \nmid b\end{subarray}}p^{\vartheta_{p}\left(n+1\right)}}.
\end{equation}
\end{lemme}
\begin{proof}
La relation \eqref{A} est triviale pour $n\in\{0,1\}$. Supposons pour la suite que $n\geq 2$ et d\'esignons respectivement par $A_n$ et $B_n$ les membres de gauche et de droite de \eqref{A}. Nous allons monter que $\vartheta_{q}\left(A_n\right)\leq \vartheta_{q}\left(B_n\right)$ pour tout nombre premier $q$, ce qui conclura que $A_n$ divise $B_n$. Soit $q$ un nombre premier arbitraire. Dans le cas o\`u $q$ divise $b$, on a $q\nmid a$ (puisque $\pgcd(a,b)=1$) et donc $q$ ne divise aucun terme de la suite arithm\'etique $\left(a+kb\right)_{k\in\mathbb{N}}$; ce qui entra\^ine que $q$ ne divise pas $L_{a,b,n}$ et on a par cons\'equent:
\[\vartheta_{q}\left(B_n\right)=\vartheta_q\left(\frac{\prod_{\begin{subarray}{c} p\leq n \\ p\mid b\end{subarray}}p^{\vartheta_{p}\left(n!\right)}}{n!}\right)=\vartheta_{q}\left(n!\right)-\vartheta_{q}\left(n!\right)=0=\vartheta_{q}\left(A_n\right).\]
Il nous reste donc \`a montrer l'in\'egalit\'e $\vartheta_{q}\left(A_n\right)\leq \vartheta_{q}\left(B_n\right)$ dans le cas o\`u $q$ ne divise pas $b$. Supposons pour toute la suite que $q\nmid b$ et d\'efinissons $S_{a,b,n}:=\left\lbrace a,a+b,\dots,a+nb\right\rbrace$. On distingue les deux cas suivants: \\
$\bullet$ \underline{\textbf{1\textsuperscript{er} cas:}} (si $q\leq n$). Dans ce cas, on a visiblement $\vartheta_q\left(A_n\right)=0$. On doit donc montrer que $\vartheta_q\left(B_n\right)\geq 0$. Pour tout entier strictement positif $\ell$, d\'esignons par $x_{\ell}$ l'unique solution de la congruence $a+bx\equiv 0 \pmod {q^{\ell}}$ dans l'ensemble $\{0,1,\dots,q^{\ell}-1\}$. Le nombre d'\'el\'ements de l'ensemble $S_{a,b,n}$ qui sont multiples de $q^{\ell}$ est alors \'egale au nombre d'entiers $x$ tels que $0\leq x\leq n$ et $x\equiv x_{\ell} \pmod {q^{\ell}}$; ce qui est clairement \'egale \`a $\left\lfloor \frac{n-x_{\ell}}{q^{\ell}}\right\rfloor +1$. On a par cons\'equent:
\begin{align*}
\vartheta_{q}\left(a\left(a+b\right)\cdots \left(a+nb\right)\right)&=\sum_{\ell\geq 1}\left(\left\lfloor \frac{n-x_{\ell}}{q^{\ell}}\right\rfloor +1\right)=\sum_{\ell\geq 1}\left(\left\lfloor \frac{n-x_{\ell}+q^{\ell}}{q^{\ell}}\right\rfloor\right)\\&\geq \sum_{\ell\geq 1}\left\lfloor \frac{n+1}{q^{\ell}}\right\rfloor=\vartheta_{q}\left(\left(n+1\right)!\right).
\end{align*}
D'o\`u: $\vartheta_{q}\left(B_n\right)=\vartheta_{q}\left(a\left(a+b\right)\cdots \left(a+nb\right)\right)-\vartheta_{q}((n+1)!)\geq 0$, comme il fallait le prouver.\\[1mm]
$\bullet$ \underline{\textbf{2\textsuperscript{nd} cas:}} (si $q>n$). Dans ce cas, comme la congruence $a+bx\equiv 0\pmod q$ poss\`ede une et une unique solution dans l'ensemble $\{0,1,\dots,q-1\}$ (car $\pgcd(b,q)=1$) alors elle poss\`ede au plus une solution dans l'ensemble $\{0,1,\dots,n\}$. Autrement dit, $q$ divise au plus un \'el\'ement de l'ensemble $S_{a,b,n}$. On a par cons\'equent: $\vartheta_{q}\left(A_n\right)=\vartheta_q\left(L_{a,b,n}\right)=\vartheta_{q}\left(a\left(a+b\right)\cdots \left(a+nb\right)\right)=\vartheta_{q}\left(B_n\right)$. Ce qui confirme le r\'esultat requis pour ce cas et compl\`ete cette d\'emonstration. 
\end{proof}

\begin{lemme}\label{7}
Pour tout entier $b\geq 3$, on a: $\prod_{p\mid b}{p^{1/p}}\leq c_6 \log b$.
\end{lemme}
\begin{proof}
Si $\omega(b)=1$ alors il existe un nombre premier $q_1$ et un entier strictement positif $m$ tels que $b={q_1}^m$. On a par cons\'equent: $\prod_{p\mid b}p^{1/p}={q_1}^{1/{q_1}}\leq 3^{1/3}\leq c_6\log 3\leq c_6\log b$ (car la fonction $n\mapsto n^{1/n}$ atteint son maximum sur $\mathbb{N^*}$ en $n=3$). Supposons pour la suite que $\omega(b)\geq 2$ et montrons pr\'ealablement que l'on a:
\begin{equation}\label{lem}
\sum_{p\mid b}\frac{\log p}{p}\leq \sum_{p\leq p_{\omega(b)}}\frac{\log p}{p}.
\end{equation}
Dans le cas o\`u $b$ est pair, l'in\'egalit\'e \eqref{lem} d\'ecoule directement de la d\'ecroissance de la fonction $x\mapsto\frac{\log x}{x}$ sur l'intervalle $\left[3,+\infty \right[$. Si maintenant $b$ est impair alors en d\'esignant par $q$ le plus grand facteur premier de $b$, on a (puisque $\omega(b)\geq 2$ par hypoth\`ese): $q\geq 5$. Il s'ensuit alors de la d\'ecroissance de la fonction $x\mapsto\frac{\log x}{x}$ sur l'intervalle $\left[3,+\infty \right[$ que:
\[\sum_{p\mid b}\frac{\log p}{p}=\sum_{\begin{subarray}{c} p\mid b\\ p\neq q\end{subarray}}\frac{\log p}{p}+\frac{\log q}{q}\leq \sum_{3\leq p\leq p_{\omega(b)}}\frac{\log p}{p}+\frac{\log 5}{5} \leq \sum_{p\leq p_{\omega(b)}}\frac{\log p}{p};\]
confirmant ainsi \eqref{lem} \'egalement dans le cas o\`u $b$ est impair.\\ En prenant maintenant les exponentielles des deux membres de \eqref{lem}, on a:
\begin{equation}\label{cor}
\prod_{p\mid b}p^{\frac{1}{p}}\leq \prod_{p\leq p_{\omega(b)}}p^{\frac{1}{p}}.
\end{equation}
Pour $\omega(b)\in\{2,3,4,5\}$, on v\'erifie \`a la main que $\prod_{p\leq p_{\omega(b)}}p^{1/p}\leq c_6\log \left(\prod_{p\leq p_{\omega(b)}}p\right)\leq c_6 \log b$, ce qui conclut (via \eqref{cor}) \`a l'estimation requise par le lemme. Si par contre $\omega(b)\geq 6$, on a (d'apr\`es le th\'eor\`eme \ref{1}):
\[\prod_{p\leq p_{\omega(b)}}p^{\frac{1}{p}}\leq p_{\omega(b)}\leq \omega(b)\left(\log \omega(b) +\log \log \omega(b)\right).\]
En combinant ceci avec le th\'eor\`eme \ref{4} et l'in\'egalit\'e $\omega(b)\leq \log b$ (qui elle-m\^eme d\'ecoule du th\'eor\`eme \ref{4} et du fait que $\log\log b\geq c_5$, vu que $b\geq 2\cdot 3\cdot 5\cdot 7\cdot 11\cdot 13\geq e^{e^{c_5}}$), on aboutit \`a:
\begin{align*}
\prod_{p\leq p_{\omega(b)}}p^{\frac{1}{p}}&\leq c_5 \frac{\log b}{\log \log b}\left(\log c_5+\log \log b -\log\log\log b +\log\log\log b\right)\\&= \left(c_5+\frac{c_5\log c_5}{\log\log b}\right)\log b\\&\leq \left(c_5+\frac{c_5\log c_5}{\log\log\left(2\cdot 3\cdot 5\cdot 7\cdot 11\cdot 13\right)}\right)\log b \leq c_6 \log b.
\end{align*}
Ce qui conclut encore (via \eqref{cor}) \`a l'estimation requise par le lemme. Notre d\'emonstration est compl\`ete.
\end{proof}

\begin{lemme}\label{8}
Soient $b$ et $n$ deux entiers $\geq 2$. Alors, on a: $\prod_{p\mid b}p^{\vartheta_{p}\left(n!\right)}\leq \left(c_4\log b\right)^{n}$.
\end{lemme}  
\begin{proof}
Pour $b=2$, l'estimation du lemme d\'ecoule imm\'ediatement de la formule de Legendre. En effet, on a:
\[\prod_{p\mid 2}p^{\vartheta_{p}\left(n!\right)}=2^{\vartheta_{2}\left(n!\right)}=2^{\lfloor \frac{n}{2}\rfloor+\lfloor \frac{n}{4}\rfloor+\lfloor \frac{n}{8}\rfloor+\dots}\leq 2^{\frac{n}{2}+\frac{n}{4}+\frac{n}{8}+\dots}=2^{n}\leq \left(c_4\log 2\right)^{n}.\]
Supposons maintenant que $b\geq 3$. En utilisant successivement la formule de Legendre, le lemme \ref{7} et le point \ref{3} du th\'eor\`eme \ref{1}, on obtient que:
\begin{align*}
\prod_{p\mid b}p^{\vartheta_{p}\left(n!\right)}&=\prod_{p\mid b}p^{\left\lfloor\frac{n}{p}\right\rfloor + \left\lfloor\frac{n}{p^2}\right\rfloor +\dots}\leq \prod_{p\mid b}p^{\frac{n}{p} + \frac{n}{p^2} +\dots}=\left(\prod_{p\mid b}p^{\frac{1}{p}}\right)^n \left(\prod_{p\mid b}p^{\frac{1}{p(p-1)}}\right)^n\\&\leq \left(c_6 c_7\log b\right)^{n}\leq \left(c_4\log b\right)^{n},
\end{align*}
comme il fallait le prouver. Ce qui compl\`ete notre d\'emonstration.
\end{proof}

\section{Preuves de nos r\'esultats principaux}

\begin{proof}[Démonstration du th\'eor\`eme \ref{9}]
Soit $n \geq b + 1$ un entier. Supposons d'abord que $a<b$. En utilisant successivement les lemmes \ref{6}, \ref{5} et \ref{8}, on obtient que:
\begin{align*}
L_{a,b,n}&=\left(\prod_{p\leq n}p^{\vartheta_p\left(L_{a,b,n}\right)}\right)\left(\prod_{p>n}p^{\vartheta_p\left(L_{a,b,n}\right)}\right)\\&\leq \frac{{c_2}^n}{(a+bn)^{\omega(b)}}\cdot\frac{a\left(a+b\right)\cdots \left(a+nb\right)}{n!}\cdot \prod_{\begin{subarray}{c} p\leq n \\ p\mid b\end{subarray}}p^{\vartheta_{p}\left(n!\right)}\\ &\leq {c_2}^n\frac{(a+nb)}{(a+nb)^{\omega(b)}}\cdot \frac{b(2b)(3b)\cdots \left(nb\right)}{n!}\cdot\prod_{p\mid b}p^{\vartheta_{p}\left(n!\right)}\\&\leq {c_2}^n b^n\left(c_4\log b\right)^n\leq \left(c_1\cdot b\log b\right)^n,
\end{align*}
comme il fallait le prouver. Si maintenant on a au contraire $a>b$, alors en posant $q:=\left\lfloor a/b\right\rfloor$ et $a':=a-qb<b$, on a visiblement: $L_{a,b,n}$ divise $L_{a',b,q+n}$; ce qui permet de conclure au r\'esultat requis en appliquant le premier cas au triplet $(a',b,q+n)$ au lieu de $(a,b,n)$. Notre d\'emonstration est compl\`ete.
\end{proof}

\begin{proof}[Démonstration du th\'eor\`eme \ref{10}]
Il suffit de reprendre la preuve pr\'ec\'edente du th\'eor\`eme \ref{9} (le premier cas pr\'ecis\'ement) et d'utiliser la majoration:
\[\prod_{\begin{subarray}{c} p\leq n \\ p\mid b\end{subarray}}p^{\vartheta_{p}\left(n!\right)}=b^{\vartheta_{b}\left(n!\right)}=b^{\left\lfloor\frac{n}{b}\right\rfloor + \left\lfloor\frac{n}{b^2}\right\rfloor +\dots}\leq b^{\frac{n}{b-1}}\]
au lieu de celle du lemme \ref{8}.
\end{proof}

\begin{proof}[Démonstration du corollaire \ref{B}]
\'Etant donn\'e $r$ un entier $\geq 2$ et $n$ un entier $\geq r+1$, on a en utilisant l'in\'egalit\'e \eqref{ff1} puis le th\'eor\`eme \ref{9}:
\[\frac{n-1}{n}\log (r+1)\leq \frac{\log L_{1,r,n}}{n}\leq \log r+\log\log r +\log c_1.\]
Le r\'esultat requis en d\'ecoule par passage \`a la limite lorsque $n$ tend vers l'infini, tout en se servant de l'estimation \eqref{batem}.
\end{proof}

\begin{proof}[Preuve du corollaire \ref{11}]
Soient $x\geq k(k+1)$ et $m:=\left\lfloor \frac{x}{k}\right\rfloor$. On a clairement:
\[\theta(x;k,\ell)\leq \log\ppcm\left(\ell,\ell+k,\dots,\ell+mk\right)=\log L_{\ell,k,m}.\]
Par ailleurs, puisque $m\geq k+1$ (car $x\geq k(k+1)$), on a en vertu du th\'eor\`eme \ref{10}:
\[\log L_{\ell,k,m}\leq m\left(\log c_2+\frac{k\log k}{k-1}\right)\leq x\left(\frac{2c_3}{k}+\frac{\log k}{k-1}\right).\]
Ce qui conclut au r\'esultat requis.
\end{proof}

\bigskip

\begin{center}
\includegraphics[scale=1]{fini}
\end{center}

\chapter{Un encadrement effectif du $\ppcm$ de la suite $(n^2+1)_n$}\label{ch5}

\section{Introduction}
Dans ce chapitre, il est question d'estimer les nombres entiers:
\[L_n:=\ppcm\left(1^2+1,2^2+1,\dots,n^2+1\right)~~~~(n\in\mathbb{N^*}).\]
On se propose d'une part d'améliorer la minoration de Oon \cite{oon} (pour le cas $c=1$), qui énonce que $L_n\geq 2^n$ $(\forall n\in\mathbb{N^*})$ et d'autre part d'établir une première majoration effective et non triviale de $L_n$. Pour ce faire, nous adaptons la méthode du chapitre \ref{ch4} pour la suite $\left(n^2+1\right)_n$ à la place des suites arithmétiques. Étant donné un entier strictement positif $n$, on désigne par $Q_n$ l'entier strictement positif défini par:
\[Q_n:=\left(1^2+1\right)\left(2^2+1\right)\cdots\left(n^2+1\right).\]
Pour $i\in\{1,3\}$, on définit:
\[\mathcal{P}_{4,i}:=\left\lbrace p~\text{premier};~p\equiv i\pmod 4\right\rbrace .\] 
Afin d'alléger certains énoncés, nous posons: $\alpha_1:=0,7993$, $\alpha_2:=10,3624$, $\alpha_3:=3,9497$, $\beta_1:=0,6722$, $\beta_2:=0,5981$, $\beta_3:=0,281$, \pagebreak

\noindent $c_1:=\frac{1}{2}-\frac{0,4}{3\log 10}+\int_{1}^{10^3}\frac{\theta\left(t;4,3\right)}{t^2}\mathrm{d}t-\frac{3}{2}\log 10+0,4\log\log (10^3)=0,1608548666\dots$, $c_2:=\left(\frac{1}{2}+\frac{0,4}{3\log 10}\right) = 0,5579\dots$, $c_3:=2,1284$, $c_4:=\left(1+\frac{5}{6\log 10}\right) = 1,3619\dots$ et $c_5:=e^{\frac{c_4}{6\cdot 10^3\log 10}} = 1,0001\dots$. Le théorème principal de ce chapitre est le suivant:

\begin{thm}\label{16}
Pour tout entier $n\geq 2$, on a:
\begin{equation}\label{16-1}
\left(\alpha_1\sqrt{n}\left(\log n\right)^{-0,4}\right)^n \leq \ppcm\left(1^2+1,2^2+1,\dots,n^2+1\right)\leq \alpha_2\left(\alpha_3n\left(\log n\right)^{0,8}\right)^n.
\end{equation}
\end{thm}

\section{Préparation}
La preuve du théorème \ref{16} nécessite les résultats intermédiaires suivants:  
\subsection{Résultats antérieurs}

\begin{thm}[Bennett et al. \cite{theta}]\label{14}
Pour tout nombre réel $x\geq 10^3$, on a:
\[\left|\theta\left(x;4,3\right)-\frac{x}{2}\right|\leq 0,4\frac{x}{\log x}~~\text{et}~~\pi\left(x;4,1\right)\leq \frac{x}{2\log x}\left(1+\frac{5}{2\log x}\right).\]
\end{thm}

\begin{thm}[classique]\label{Hensel}
Soient $\ell\in\mathbb{N^*}$, $n\in\mathbb{Z}$ et $p$ un nombre premier impair ne divisant pas $n$. Alors, la congruence $x^2\equiv n\pmod{p^\ell}$ possède exactement $1+\left(\frac{n}{p}\right)$ solutions, où $\left(\frac{\cdot}{\cdot}\right)$ représente le symbole de Legendre.
\end{thm}
\noindent La preuve du théorème \ref{Hensel} peut être trouvée dans le livre de Hua \cite[Theorem 5.1]{Hua}.

\subsection{Lemmes préparatifs}

\begin{lemme}\label{12}
Soit $n$ un entier strictement positif. Alors, le nombre entier ${L_n}^2$ est multiple du nombre rationnel
\[D_n:=\frac{{2}^{\left\lfloor n/2\right\rfloor +1}Q_n}{{n!}^{2}}\prod_{p\in{\mathcal{P}_{4,3}}\cap\left[1,n\right]}p^{2\vartheta_{p}\left(n!\right)}.\] 
\end{lemme}
\begin{proof}
Nous allons montrer que $\vartheta_p\left(D_n\right)\leq \vartheta_p\left(L_{n}^{2}\right)$ pour tout nombre premier $p$, ce qui conclura au résultat requis. Pour $p=2$, on constate que l'on a pour tout entier naturel $k$:
\[\vartheta_{2}\left(k^2+1\right)=\begin{cases}1~~\text{si}~k~{est~impair}\\0~~\text{sinon} \end{cases};\]
ce qui entraîne que:
\[\vartheta_2\left(Q_n\right)=\#\left\lbrace i\in\mathbb{N^*};~i\leq n~,~i~\text{impair}\right\rbrace=\left\lfloor\frac{n-1}{2}\right\rfloor+1.\]
On a par conséquent:
\begin{align*}
\vartheta_2\left(D_n\right)&=\left\lfloor \frac{n}{2}\right\rfloor +1+\vartheta_2\left(Q_n\right)-2\vartheta_2(n!)\\&=\left\lfloor\frac{n-1}{2}\right\rfloor +\left\lfloor\frac{n}{2}\right\rfloor +2 - 2\vartheta_2\left(n!\right)\\&\leq \left\lfloor\frac{n-1}{2}\right\rfloor -\left\lfloor\frac{n}{2}\right\rfloor+2 \leq 2 = \vartheta_2\left(L_{n}^{2}\right),
\end{align*}
comme il fallait le prouver. Prenons pour la suite $p$ un nombre premier $>2$ et montrons que l'on a $\vartheta_p\left(D_n\right)\leq \vartheta_p\left(L_{n}^{2}\right)$. Si $p\in{\mathcal{P}_{4,3}}$ alors, d'après le théorème \ref{Hensel}, l'équation $x^2\equiv -1\pmod p$ n'a pas de solution (car: $\left(\frac{-1}{p}\right)=(-1)^{\frac{p-1}{2}}=-1$). Ce qui entraîne que $\vartheta_{p}\left(D_n\right)=2\vartheta_p\left(n!\right)-\vartheta_p\left({n!}^{2}\right)=0=\vartheta_p\left(L_n\right)=\vartheta_p\left(L_{n}^{2}\right)$ et confirme l'inégalité requise. Si au contraire $p\in{\mathcal{P}_{4,1}}$, on conclura par distinction de cas. Posons préalablement $S_n:=\left\lbrace 1^2+1,2^2+1,\dots,n^2+1\right\rbrace$.\\
\textbullet{} \underline{\textbf{1\textsuperscript{er} cas:}} (si $p>2n$).\\
Nous affirmons d'abord que $p$ divise au plus un seul élément de l'ensemble $S_n$. Procédons par l'absurde en supposant que $p$ divise à la fois $\left(i^2+1\right)$ et $\left(j^2+1\right)$ pour certains $i,j\in\mathbb{N^*}$, tels que $i<j\leq n$. Cela entraîne que $p$ divise $(j^2+1)-(i^2+1)=(j-i)(j+i)$, donc $p$ divise l'un des nombres $(j-i)$ ou $(j+i)$, d'où $p\leq (i+j)<2n$. Ce qui contredit le fait que $p>2n$ et confirme notre affirmation. D'après cette affirmation, on a: $\vartheta_p\left(D_n\right)=\vartheta_p\left(Q_n\right)=\vartheta_p\left(L_n\right)\leq \vartheta_p\left(L_{n}^{2}\right)$, comme il fallait le prouver.\\
\textbullet{} \underline{\textbf{2\textsuperscript{ème} cas:}} (si $n<p<2n$).\\
Dans ce cas, on a clairement $p^2>n^2+1$, donc $p^2$ ne divise aucun élément de l'ensemble $S_n$. Puisque le nombre de solutions de l'équation $x^2\equiv -1\pmod p$ est égal à $2$ (en vertu du théorème \ref{Hensel}), alors $p$ divise au plus deux éléments de l'ensemble $S_n$. D'où, $\vartheta_{p}\left(D_n\right)=\vartheta_p\left(Q_n\right)\leq 2\vartheta_p\left(L_n\right)=\vartheta_p\left(L_{n}^{2}\right)$, comme il fallait le prouver. \\
\textbullet{} \underline{\textbf{3\textsuperscript{ème} cas:}} (si $p\leq n$).\\
En posant $\nu_k:=\left\lfloor\frac{n}{p^{k}}\right\rfloor$ $(\forall k\in\mathbb{N^*})$, il s'ensuit (en vertu du théorème \ref{Hensel}) que: \pagebreak
\begin{align*}
\vartheta_p\left(Q_n\right)&=\sum_{k=1}^{\vartheta_p\left(L_n\right)}\#\left\lbrace i\in\mathbb{N^*};~i\leq n,~p^k\mid i^2+1\right\rbrace\\&\leq\sum_{k=1}^{\vartheta_p\left(L_n\right)}\left[\left(\sum_{\ell=0}^{\nu_k-1}\#\left\lbrace i\in\mathbb{N^*};~\ell p^k+1\leq i\leq (\ell+1)p^k,~p^k\mid i^2+1\right\rbrace\right)+2\right]\\&\leq 2\sum_{k=1}^{\vartheta_p\left(L_n\right)}\left(\left\lfloor \frac{n}{p^k}\right\rfloor+1\right)\leq 2\vartheta_p\left(n!\right)+2\vartheta_p\left(L_n\right).
\end{align*}
D'où l'on a: $\vartheta_p\left(D_n\right)=\vartheta_p\left(Q_n\right)-2\vartheta_p\left(n!\right)\leq 2\vartheta_p\left(L_n\right)=\vartheta_p\left(L_{n}^{2}\right)$, comme il fallait le prouver. Ce qui complète cette démonstration.  
\end{proof}

\begin{lemme}\label{13}
Pour tout entier $n\geq 2$, on a:
\[G_n:=\prod_{p>n}p^{\vartheta_p\left(L_n\right)}~\text{divise}~M_n:=\frac{2^{2\vartheta_{2}\left(n!\right)-\left\lfloor (n-1)/2\right\rfloor -1}Q_n}{{n!}^{2}}\prod_{p\in{\mathcal{P}_{4,3}}\cap\left[1,n\right]}p^{2\vartheta_{p}\left(n!\right)}.\]
\end{lemme}
\begin{proof}[Démonstration]
Soit $p$ un nombre premier. On vérifie facilement que: $\vartheta_p\left(G_n\right)=\vartheta_p\left(M_n\right)=0$ pour $p\in{\mathcal{P}_{4,3}}\cup\{2\}$. Supposons maintenant que $p\in{\mathcal{P}_{4,1}}\cap\left[1,n\right]$ et posons $\nu_k:=\left\lfloor\frac{n}{p^k}\right\rfloor$ $(\forall k\in\mathbb{N^*})$. D'après le théorème \ref{Hensel}, on a:
\begin{align*}
\vartheta_p\left(Q_n\right)&=\sum_{k\geq 1}\#\left\lbrace i\in\mathbb{N^*};~i\leq n~\text{et}~p^k\mid i^2+1\right\rbrace\\&\geq\sum_{k=1}^{\vartheta_p\left(L_n\right)}\sum_{\ell=0}^{\nu_k-1}\#\left\lbrace i\in\mathbb{N^*};~\ell p^k+1\leq i\leq (\ell+1)p^k,~p^k\mid i^2+1\right\rbrace\\&= 2\sum_{k\geq 1}\left\lfloor \frac{n}{p^k}\right\rfloor = 2\vartheta_p\left(n!\right).
\end{align*}
Ce qui montre que $\vartheta_p\left(M_n\right)=\vartheta_p\left(Q_n\right)-2\vartheta_p\left(n!\right)\geq \vartheta_p\left(G_n\right)=0$. Supposons enfin que $p>n$. En raisonnant comme dans la démonstration du lemme \ref{12} (en distinguant les cas $n<p<2n$ et $p>2n$), on montre que $\vartheta_p\left(G_n\right)=\vartheta_p\left(L_n\right)\leq \vartheta_p\left(M_n\right)$. Dans tous les cas, on a bien $\vartheta_p\left(G_n\right)\leq \vartheta_p\left(M_n\right)$; ce qui conclut au résultat requis et complète cette démonstration.
\end{proof}

\begin{lemme}\label{15}
Pour tout entier $n\geq 10^3$, on a:
\begin{equation}\label{15-p}
\left(\beta_1\sqrt{n}\left(\log n\right)^{-0,4}\right)^n\leq\prod_{p\in{\mathcal{P}_{4,3}}\cap\left[1,n\right]}p^{\vartheta_p\left(n!\right)}\leq \left(\beta_2\sqrt{n}\left(\log n\right)^{0,4}\right)^n.
\end{equation}
\end{lemme}
\begin{proof}
Montrons d'abord l'inégalité de gauche de \eqref{15-p}. En vertu de la formule sommatoire d'Abel et du théorème \ref{14}, on a:
\begin{align*}
\sum_{p\in{\mathcal{P}_{4,3}}\cap\left[1,n\right]}\frac{\log p}{p}&=\frac{\theta\left(n;4,3\right)}{n}+\int_{1}^{n}\frac{\theta\left(t;4,3\right)}{t^2}\mathrm{d}t \\ &\geq \frac{1}{2}-\frac{0,4}{3\log 10}+\int_{1}^{10^3}\frac{\theta\left(t;4,3\right)}{t^2}\mathrm{d}t+\frac{1}{2}\int_{10^3}^{n}\frac{1}{t}\mathrm{d}t -0,4\int_{10^3}^{n}\frac{1}{t\log t}\mathrm{d}t\\&\geq \frac{1}{2}\log n -0,4\log\log n +c_1.
\end{align*}
D'autre part, d'après le lemme \ref{14}, on a pour tout entier $n\geq 10^3$:
\[\sum_{p\in{\mathcal{P}_{4,3}}\cap\left[1,n\right]}\log p=\theta\left(n;4,3\right)\leq \frac{n}{2}+0,4\frac{n}{\log n}\leq c_{2}n.\]
Nous déduisons alors de ce qui précède que:
\begin{align*}
\prod_{p\in{\mathcal{P}_{4,3}}\cap\left[1,n\right]}p^{\vartheta_p\left(n!\right)}&\geq \prod_{p\in{\mathcal{P}_{4,3}}\cap\left[1,n\right]}p^{\frac{n}{p}-1}\geq \left(\frac{e^{(c_1-c_2)}\sqrt{n}}{\left(\log n\right)^{0,4}}\right)^{n}\geq\left(\frac{\beta_1\sqrt{n}}{\left(\log n\right)^{0,4}}\right)^{n},
\end{align*}
comme il fallait le prouver. Pour l'inégalité de droite de \eqref{15-p}, on montre de la même façon que pour tout entier $n\geq 10^3$, on a:
\[\prod_{p\in{\mathcal{P}_{4,3}}\cap\left[1,n\right]}p^{\frac{1}{p}}\leq \beta_3\sqrt{n}\left(\log n\right)^{0,4}.\]
Par suite, en utilisant l'estimation $\prod_{p}p^{1/p(p-1)}\leq c_3$ (qui découle du troisième point du théorème \ref{1}), on obtient:
\begin{align*}
\prod_{p\in{\mathcal{P}_{4,3}}\cap\left[1,n\right]}p^{\vartheta_p\left(n!\right)}&\leq \prod_{p\in{\mathcal{P}_{4,3}}\cap\left[1,n\right]}p^{\frac{n}{p}+\frac{n}{p(p-1)}}\\&\leq\left(\prod_{p\in{\mathcal{P}_{4,3}}\cap\left[1,n\right]}p^{\frac{n}{p}}\right)\left(\prod_{p}p^{\frac{n}{p(p-1)}}\right)\\&\leq \left(c_3\beta_3\sqrt{n}\left(\log n\right)^{0,4}\right)^n\\&\leq \left(\beta_2\sqrt{n}\left(\log n\right)^{0,4}\right)^n,
\end{align*}
comme il fallait le prouver. Notre démonstration est complète.
\end{proof}

\section{Preuve de notre résultat principal}

\begin{proof}[Démonstration du théorème \ref{16}]
En se servant d'un logiciel de calcul (Maple ou Mathematica par exemple), on vérifie que \eqref{16-1} est valable pour tout entier $2\leq n<10^3$. Supposons pour la suite que $n\geq 10^3$. En utilisant successivement le lemme \ref{12} et le lemme \ref{15}, on obtient:
\[L_n\geq {\sqrt{2}}^{\left\lfloor n/2\right\rfloor+1}\sqrt{\frac{Q_n}{{n!}^{2}}}\left(\frac{\beta_1\sqrt{n}}{\left(\log n\right)^{0,4}}\right)^n \geq \left(\frac{\sqrt[4]{2}\beta_1\sqrt{n}}{\left(\log n\right)^{0,4}}\right)^n\geq \left(\frac{\alpha_1\sqrt{n}}{\left(\log n\right)^{0,4}}\right)^n.\]
Ce qui conclut à la minoration de \eqref{16-1}. Pour montrer la majoration de \eqref{16-1}, nous écrivons d'abord:
\begin{equation}\label{16-2}
L_n=\prod_{p}p^{\vartheta_p\left(L_n\right)}=\left(\prod_{p\leq n}p^{\vartheta_p\left(L_n\right)}\right)\left(\prod_{p>n}p^{\vartheta_p\left(L_n\right)}\right).
\end{equation} 
D'une part, on a (en vertu du théorème \ref{14}):
\[\prod_{p\leq n}p^{\vartheta_p\left(L_n\right)}=2\prod_{p\in{\mathcal{P}_{4,3}}\cap\left[1,n\right]}p^{\vartheta_p\left(L_n\right)}\leq 2\left(n^2+1\right)^{\pi\left(n;4,1\right)}\leq  2\left(n^2+1\right)^{\frac{n}{2\log n}\left(1+\frac{5}{6\log 10}\right)}\leq 2c_5e^{c_4 n}.\]
D'autre part, d'après les lemmes \ref{13} et \ref{15}, on a:
\[\prod_{p>n}p^{\vartheta_p\left(L_n\right)}\leq M_n\leq e^{\pi^2/6}\left({\beta_2}^{2} 2^{3/2}n\left(\log n\right)^{0,8}\right)^n.\]
En reportant ces estimations dans \eqref{16-2}, on aboutit à:
\[L_n\leq 2c_5e^{\pi^2/6}\left({\beta_2}^{2} 2^{3/2}e^{c_4}n\left(\log n\right)^{0,8}\right)^n\leq \alpha_2\left(\alpha_3n\left(\log n\right)^{0,8}\right)^n,\]
comme il fallait le prouver. Le théorème est ainsi démontré.
\end{proof} 

\bigskip

\begin{center}
\includegraphics[scale=1]{fini}
\end{center}

\chapter*{Conclusion générale}

\addcontentsline{toc}{chapter}{Conclusion générale}

Le sujet initialement fixé de cette thèse consistait à établir des encadrements effectifs du plus petit commun multiple de termes consécutifs de certaines suites d'entiers. 

En utilisant des arguments d'algèbre commutative et d'analyse complexe, nous avons obtenu un diviseur rationnel non trivial du nombre entier: 
\[L_{c,m,n}:=\ppcm\left(m^2+c,(m+1)^2+c,\dots,n^2+c\right),\]
avec $n,m,c\in\mathbb{N^*}$ et $m\leq n$. Comme corollaire, nous en avons déduit des minorations effectives et non triviales pour $L_{c , m , n}$, et cela sans la contrainte ``$1\leq m\leq \left\lceil\frac{n}{2}\right\rceil$'' (qui figure dans le résultat de Oon \cite{oon}). Dans une autre direction, nous avons généralisé quelques identités dues à Farhi \cite{far}, Nair \cite{nair} et Guo \cite{Victor} aux suites à forte divisibilité. Nous avons appliqué ces identités ensuite pour estimer le $\ppcm$ de certaines suites de Lucas. En fait, la méthode utilisée est même applicable pour estimer le $\ppcm$ de toute suite à forte divisibilité d'ordre connu, comme celles d'ordre exponentiel. Par ailleurs, nous avons développé une toute première méthode permettant d'obtenir une majoration effective du $\ppcm$ d'une progression arithmétique finie. Comme conséquence, nous en avons déduit une estimation d'une certaine moyenne harmonique forte intéressante ainsi qu'une estimation de l'expression $\theta(x;k,\ell)$ de Chebyshev. Enfin, nous avons réussi à adapter notre méthode pour la suite quadratique particulière ${(n^2 + 1)}_{n}$ et obtenir ainsi un encadrement pour son $\ppcm$, qui est presque optimal. En particulier, les minorations antérieures de Farhi \cite{far} et de Oon \cite{oon} pour le $\ppcm$ de la même suite sont de beaucoup améliorées.


\chapter*{Perspectives}

\addcontentsline{toc}{chapter}{Perspectives}

Bien évidemment, il reste plusieurs problèmes que nous n'avons pas traités dans cette thèse, mais qui font partie (plus au moins) de sa thématique. En voici quelques-uns qui pourront faire l'objet de nos recherches postérieures:
\begin{enumerate}  
\item Prouver la conjecture de Cilleruelo (voir la conjecture \ref{cocil}) qui énonce que pour tout polynôme irréductible $f\in\mathbb{Z}[X]$, de degré $\geq 3$, le logarithme du plus petit commun multiple des termes consécutifs de la suite $\left(f(n)\right)_n$ est de même ordre que le logarithme du produit de ces termes (lorsque $n$ voisine l'infini).
\item Améliorer notre majoration du $\ppcm$ d'une suite arithmétique et étendre notre méthode à d'autres suites.
\item Améliorer la minoration ``uniforme'' de Hong et al. \cite{hong2} du $\ppcm$ de certaines suites polynomiales (ce qui semble réalisable par des techniques d'algèbre polynomiale).
\item \'Etablir des majorations non triviales du $\ppcm$ d'une suite polynomiale.  
\item \textbf{Difficile!} Tenter d'établir un analogue du théorème des nombres premiers (sous sa version $\log\ppcm(1 , 2 , \dots , n) \sim_{+ \infty} n$) pour d'autres suites d'entiers strictement positifs.
\end{enumerate}

Pour envoyer \begin{titlepage}
\begin{center}
\textbf{\large \underline{Résumé}}
\end{center}

Cette thèse consiste à étudier des estimations effectives du plus petit commun multiple de certaines suites d'entiers. Nous nous focalisons notamment sur une certaine classe de suites quadratiques, ainsi que les progressions arithmétiques et les suites à forte divisibilité. Premièrement, nous avons utilisé des méthodes d'algèbre commutative et d'analyse complexe pour établir de nouvelles minorations non triviales du $\ppcm$ de certaines suites quadratiques. Ensuite, une étude plus profonde des propriétés arithmétiques de suites à forte divisibilité nous a permis d'obtenir trois identités intéressantes concernant le $\ppcm$ de ces suites, ce qui généralise certaines identités antérieures de B. Farhi (2009) et M. Nair (1982). Nous en avons déduit par suite des estimations assez précises du $\ppcm$ d'une suite de Fibonacci généralisée (ce que l'on appelle les suites de Lucas). Nous avons également développé une première méthode permettant d'effectiviser un résultat asymptotique de P. Bateman (2002) sur le $\ppcm$ d'une progression arithmétique. Vers la fin, nous avons constaté que cette dernière méthode peut être adaptée pour encadrer le $\ppcm$ de la suite $(n^2+1)_n$, ce qui nous a permis en particulier d'améliorer les minorations de B. Farhi (2005) et S. M. Oon (2013). La thèse comprend aussi une présentation générale de quelques résultats de littérature. 
\vskip 0.25cm
\noindent \textbf{Mots-clés:} Plus petit commun multiple, plus grand diviseur commun, suite à divisibilité, suite à forte divisibilité, suite de Lucas, suite de Fibonacci, progression arithmétique, suite quadratique, suite polynomiale, théorèmes de Chebychev, théorème des nombres premiers, répartition des nombres premiers.
\end{titlepage}

\begin{titlepage}
\begin{center}   
\textbf{\large \underline{Abstract}}
\end{center}

This thesis is devoted to studying estimates of the least common multiple of some integer sequences. Our study focuses on effective bounding of the $\lcm$ of some class of quadratic sequences, as well as arithmetic progressions and strong divisibility sequences. First, we have used methods of commutative algebra and complex analysis to establish new nontrivial lower bounds for the $\lcm$ of some quadratic sequences. Next, a more profound study of the arithmetic properties of strong divisibility sequences allowed us to obtain three interesting identities involving the $\lcm$ of these sequences, which generalizes some previous identities of B. Farhi (2009) and M. Nair (1982); as consequences, we have deduced a precise estimates for the $\lcm$ of generalized Fibonacci sequence (the so-called Lucas sequences). We have also developed a method that provides an effective version to the asymptotic result of P. Bateman (2002) concerning the $\lcm$ of an arithmetic progression. Finally, we found that the latter method can be adapted to estimate the $\lcm$ of the sequence $(n^2+1)_n$, which allowed us in particular to improve the lower bounds of B. Farhi (2005) and S. M. Oon (2013). The thesis also includes a general presentation of some literature results.
\vskip 0.25cm
\noindent\textbf{Keywords:} Least common multiple, greatest common divisor, divisibility sequences, strong divisibility sequences, Lucas sequences, Fibonacci sequence, arithmetic progressions, quadratic sequences, polynomial sequences, Chebychev theorems, prime number theorem, distribution of prime numbers.
\end{titlepage}

\begin{titlepage}

\begin{center}
\novocalize
{\large \textbf{\underline{\<mulaxxa.s>}}}
\end{center}
\begin{arabtex}
\novocalize
tndrj ha_dih al-'u.trU.hT fy 'i.tAr drAsT tqdyrAt al--m.dA`f al--m^strk al-'a.s.gr lb`.d mt--tAlyAt Al-'a`dAd Al.tby`yT. rkznA drAstnA xu.sU.s"aN" `lY Al.h.sr Alf``Al l--lm.dA`f al--m^strk Al-'a.s.gr lfi'T mn al--mt--tAlyAt Altrby`yT, biAl-'i.dAfT 'ilY al--mt--tAlyAt al-.hsAbyT wa al--mt--tAlyAt _dAt qwT fy qAblyT al-qsmT. bda'nA bAstxdAm .trq mstmdT mn Aljbr Altbdyly w Alt.hlyl al--mrkb l-'iyjAd .hdwd mn Al'adnY jdydT w .gyr tAfhT l--lm.dA`f al--m^strk Al-'a.s.gr lb`.d al--mt--tAlyAt Altrby`yT. 'intqlnA b`dhA 'ilY drAsT `mIqT l--lxwA.s Al'artmA.tqyT l--lmt--tAlyAt _dAt qwT fy qAblyT AlqsmT, 'amkntnA mn 'iyjAd _tlA_t mt.tAbqAt jd hAmT t--t`lq bi-al--m.dA`f al--m^strk Al-'a.s.gr l--l.hdwd Almt`AqbT mn h_dh Almt--tAlyAt. w `lY wjh AldqT, flqd t.h.slnA `lY t`mymAt li--mut.tAbqAt m`rwfT sAbqA, 'awrdhA kl mn {\bf nyr}\LR{(1982)} w {\bf fr.hy}\LR{(2005)}. w trtbt `lY h_dh al--mt.tAbqAt tqdyrAt jdd dqyqT l--lm.dA`f al--m^strk Al-'a.s.gr li--mt--tAlyAt {\bf fybwnAt^sy} al--m`mm--mT \LR{)}'aw mA y`rf bmt--tAlyAt {\bf luwkA}\LR{(}. mn jhT 'axrY, .twrnA .tryqT jdydT 'amkntnA mn 'iyrAd .sy.gT f``AlT l--lntyjT al--muqArbT {\vocalize li--} {\bf bytmAn}\LR{(2002)} w allatI tx.s.s al--m.dA`f al--m^strk Al-'a.s.gr l--lmt--tAlyAt Al.hsAbyT. w fy Al'axyr, qmnA bt.tbyq h_dh Al.tryqT `lY al--mt--tAlyT $(n^2+1)_n$, m--mA sm.h lnA `lY wjh Alx.sw.s bt.hsyn Al.hdwd AldnyA al--m`rwfT sAbqaNA mn .trf {\bf fr.hy}\LR{(2005)} w {\bf 'uwn}\LR{(2013)} {\vocalize li--} $\lcm(1^2+1,2^2+1,\dots,n^2+1)$. h_dA w lqd qdmnA fy h_dh Al'a.trw.hT 'ay.d"aN" `ar.d"aN" ^sAmil"aN" wa mufa.s.sal"aN" li--m`.zm al-nnatA'ij al--m`rwfT sAbiq"aN" fy majAl drAstnA.   

\medskip

\noindent{\large \bf al-kalimAt al--miftA.hiyaT\LR{:}} al--m.dA`f al--m^strk Al-'a.s.gr, al-qAsim al--m^strk Al-'akbr, mt--tAlyAt _dAt qAblyT al-qismaT, mut--tAlyAt _dAt qwT fy qAblyT AlqsmT,  mt--tAlyAt {\bf lUkA}, mt--tAlyT {\bf fybwnAt^sy}, al--mut--tAlyAt al-.hisAbyT, al--mut--tAlyAt al-ttrby`yT, mt--tAlyAt k_tyrAt Al.hdwd, n.zryAt {\bf ^sibI^sAf}, n.zryT Al'a`dAd Al'awlyT, twzy` Al'a`dAd Al'awlyT.
\end{arabtex}
\end{titlepage}

\end{document}